\documentclass[12pt]{amsart}
\usepackage{geometry}   
\usepackage[colorlinks,citecolor = red, linkcolor=blue,hyperindex]{hyperref}
\usepackage{euscript,eufrak,verbatim, mathrsfs}
\usepackage[psamsfonts]{amssymb}
\usepackage[usenames]{color}
\usepackage{bbm}
\usepackage{graphicx}
\usepackage{color}
\usepackage{float}
\usepackage{dsfont}

\usepackage{euscript}
\usepackage{helvet}         
\usepackage{courier}        
\usepackage{type1cm}        
\usepackage{multicol}        
\usepackage[bottom]{footmisc}

\newtheorem{theorem}{Theorem}[section]
\newtheorem*{theorem*}{Theorem B} 
\newtheorem{lemma}[theorem]{Lemma}
\newtheorem{lem-def}[theorem]{Lemma and Definition}

\newtheorem{proposition}[theorem]{Proposition}
\newtheorem{corollary}[theorem]{Corollary}
\newtheorem{definition}[theorem]{Definition}
\newtheorem*{question-A}{Question A}
\newtheorem*{question-B}{Question B}
\newtheorem*{question-C}{Question C}

\newtheorem{convention}[theorem]{Convention}
\newtheorem*{definition*}{Definition}
\newtheorem*{remark*}{Remark}

\newtheorem*{observation*}{Observation}
\newtheorem{assumption}{Assumption}

\newtheorem*{assumption*}{Assumption}
\newtheorem*{ass-1}{Assumption (A1)}
\newtheorem*{ass-2}{Assumption (A2)}
\newtheorem*{question*}{Question}
\newtheorem{remark}{Remark}
\newtheorem{example}{Example}
\newtheorem*{conjecture}{Conjecture A}

\newcommand{\R}{\mathbb{R}}
\newcommand{\N}{\mathbb{N}}
\newcommand{\Z}{\mathbb{Z}}

\newcommand{\D}{\mathbb{D}}
\newcommand{\C}{\mathbb{C}}
\newcommand{\E}{\mathbb{E}}
\newcommand{\T}{\mathbb{T}}

\newcommand{\PP}{\mathbb{P}}

\newcommand{\B}{\mathbb{B}}
\newcommand{\Sph}{\mathbb{S}}

\newcommand{\X}{\mathscr{X}}

\newcommand{\MM}{\mathcal{M}}

\newcommand{\HH}{\mathcal{H}}

\newcommand{\Conf}{\mathrm{Conf}}

\newcommand{\Var}{\mathrm{Var}}
\newcommand{\supp}{\mathrm{supp}}

\newcommand{\Aut}{\mathrm{Aut}}

\newcommand{\an}{\text{\, and \,}}

\begin{document}

\title[Patterson-Sullivan measures of point processes]{Patterson-Sullivan measures for point processes and the reconstruction of harmonic functions}

\author
{Alexander I. Bufetov}
\address
{Alexander I. BUFETOV: 
Aix-Marseille Universit\'e, Centrale Marseille, CNRS, Institut de Math\'ematiques de Marseille, UMR7373, 39 Rue F. Joliot Curie 13453, Marseille, France;
Steklov Mathematical Institute of RAS, Moscow, Russia}
\email{bufetov@mi.ras.ru, alexander.bufetov@univ-amu.fr}

\author
{Yanqi Qiu}
\address
{Yanqi QIU: Institute of Mathematics and Hua Loo-Keng Key Laboratory of Mathematics, AMSS, Chinese Academy of Sciences, Beijing 100190, China; CNRS, Institut de Math\'ematiques de Toulouse, Universit\'e Paul Sabatier
}
\email{yanqi.qiu@amss.ac.cn}

\begin{abstract}
The Patterson-Sullivan construction is proved almost surely to recover every harmonic function in a certain Banach space from its values on the zero set of a Gaussian analytic function on the disk. The argument  relies on the slow growth of variance for linear statistics of the concerned point process. As a corollary of reconstruction result in general abstract setting, Patterson-Sullivan reconstruction of harmonic functions is obtained in real and complex hyperbolic spaces of arbitrary dimension.
\end{abstract}
\dedicatory{To the memory of Alexander Ivanovich Balabanov (1952 -- 2018)}

\maketitle


\setcounter{equation}{0}

\section{The reconstruction problem} 

\subsection{The reconstruction problem}\label{sec-formulation}
Let $dA$ be the normalized Lebesgue measure on the unit disk $\D$ and  consider the Bermgan space $A^2(\D)\subset L^2(\D, dA)$ of  holomorphic functions on $\D$:
\[
A^2(\D): = \left\{f : \D\rightarrow \C\Big| \text{$f$ is holomorphic and $\int_\D | f(z)|^2 dA(z) <\infty$}\right\}.
\]
The space $A^2(\D)$ admits a reproducing kernel $K_\D$ given by the formula
\begin{align}\label{def-K-D}
K_\D(z, w) = \frac{1}{( 1 - z \bar{w})^2}.
\end{align}

Let $(g_n)_{n\ge 0}$ be a sequence of independent complex Gaussian random variables with expectation $0$ and variance $1$. The random series 
\[
\mathfrak{g}_\D(z) = \sum_{n=0}^\infty g_n z^n
\]
 almost surely has  radius of convergence $1$ and therefore defines a holomorphic function on $\D$.  Peres and Vir\'ag \cite{PV-acta} proved  that the law of the zero set $Z(\mathfrak{g}_\D) \subset \D$ 
of  the random holomorphic function $\mathfrak{g}_\D$  is the {\it determinantal point process} on $\D$ induced by the Bergman kernel $K_{\D}$, the detailed definition of determinantal point processes is recalled in \S \ref{sec-pre} below.  

For almost every realization $X = Z(\mathfrak{g}_\D)$,  we proved in \cite{BQS-LP} that any function $f\in A^2(\D)$,  equal to $0$ in restriction 
to $X$, must be the zero function; in other words,  $X$ is a  {\it uniqueness set} for $A^2(\D)$ and any function $f \in A^2(\D)$ is uniquely determined by its restriction $f|_X$ onto this countable subset $X$ of $\D$.  It is thus natural to ask the following

\begin{question-A}
For a realization  $X = Z(\mathfrak{g}_\D)$, how to recover a  Bergman function $f \in A^2(\D)$ from its restriction to the subset $X$ ?  
\end{question-A}

We first consider the reconstruction for a fixed Bergman function and then the simultaneous reconstruction for all Bergman functions in a subspace. 

\bigskip

{\flushleft \bf I. Reconstruction for a fixed $f\in A^2(\D)$.} 

\medskip

For a fixed $f\in A^2(\D)$,  the answer to Question A is given in Proposition \ref{prop-intro-single-bergman}, which can be viewed as a {\it discrete analogue of the mean-value property} for Bergman functions.  Proposition \ref{prop-intro-single-bergman} relies on the Patterson-Sullivan construction described as follows. 

Set 
\begin{equation} \label{lob-dist}
d_\D(x, z) := \log  \frac{{\displaystyle 1 +  \Big|\frac{z-x}{1 - \bar{x}z}\Big|}}{{\displaystyle 1 - \Big|\frac{z-x}{1 - \bar{x}z}\Big|}} \quad \text{ for $x, z \in \D$},
\end{equation}
and recall that the disk $\D$ endowed with the distance $d_\D(\cdot, \cdot)$ is the Poincar{\'e} model for the Lobachevsky plane. Then as we shall prove in Proposition \ref{prop-intro-single-bergman} below,  for fixed $f \in A^2(\D)$ and fixed $z\in \D$, along a subsequence $s_n \to 1^{+}$,  
\begin{align}\label{weight-av}
f(z) =   \lim_{n\to\infty}  \frac{ \displaystyle \sum_{k =0}^\infty \sum_{x\in Z(\frak{g}_\D)\atop k \le d_\D(x, z) < k + 1} e^{-s_n d_\D(z, x)} f(x)}{\displaystyle \sum_{x \in Z(\frak{g}_\D)} e^{-s_n d_\D(z, x)} }, \quad \text{almost surely.}
\end{align}
 The equality \eqref{weight-av} follows from a {\it Law of Large Numbers} of the linear statistics of the random subset $Z(\frak{g}_\D)$. More precisely, for any $f\in A^2(\D)$ and  any $z\in \D$, it will be shown that for any $s> 1$ we may define  
\begin{align}\label{def-Q-s}
Q_s(Z(\frak{g}_\D); f, z): =  \frac{ \displaystyle \sum_{k =0}^\infty \sum_{x\in Z(\frak{g}_\D)\atop k \le d_\D(x, z) < k + 1} e^{-s d_\D(z, x)} f(x)}{\displaystyle \E \Big( \sum_{x \in Z(\frak{g}_\D)} e^{-s d_\D(z, x)}\Big) }.
\end{align}
This random variable defined in \eqref{def-Q-s} has a mean  $f(z)$ (cf. Lemma \ref{lem-observation}): 
\[
  \E [Q_s(Z(\frak{g}_\D); f, z)] = f(z), \quad \text{for any $s> 1$,}
\]
and  has an asymptotically vanishing variance as $s\to 1^{+}$:
\[
\lim_{s\to 1^{+}} \Var(Q_s(Z(\frak{g}_\D); f, z)) =0,
\]
whence \eqref{weight-av}.

Several comments are needed here. First, the double summation in the numerator of \eqref{weight-av} or \eqref{def-Q-s}  is necessary since, for almost every realization $X = Z(\frak{g}_\D)$,  the absolute convergence of the series 
\[
 \sum_{x \in Z(\frak{g}_\D) } e^{-s d_\D(x, z)} f(x)
\]  for a general Bergman function $f\in A^2(\D)$ is not clear when the exponent $s$ is close to the critical value $1$, cf. conjecture A and Remark \ref{no-abs-conv} below.  Nonetheless, our result says that fix any $s>1$, then for  almost every  realization $X = Z(\frak{g}_\D)$, there is sufficient cancellation inside the partial summation 
\[
\sum_{x\in X\atop k \le d_\D(x, z) < k + 1} e^{-s d_\D(z, x)} f(x)
\]
which leads to the convergence of the double summation. 

Second, the exponent $s = 1$ is a critical exponent since 
\[
\E \Big( \sum_{x \in Z(\frak{g}_\D)} e^{-s d_\D(z, x)}\Big) \left\{ \begin{array}{cc} <\infty & \text{if $s> 1$}  \vspace{3mm}\\ = \infty & \text{if $s =1$} \end{array}  \right..
\]
The value of the critical exponent is equal to the volume entropy of the hyperbolic area of $\D$, cf. Lemma \ref{lem-disk-ent}.

Third, one may replace the denominator ${\displaystyle\sum_{x \in Z(\frak{g}_\D)} e^{-s d_\D(z, x)}}$  in \eqref{weight-av} by the denominator ${\displaystyle \E \Big( \sum_{x \in Z(\frak{g}_\D)} e^{-s d_\D(z, x)}\Big)}$ in \eqref{def-Q-s} since the ratio
\[
\frac{\displaystyle\sum_{x \in Z(\frak{g}_\D)} e^{-s d_\D(z, x)}}{\displaystyle \E \Big( \sum_{x \in Z(\frak{g}_\D)} e^{-s d_\D(z, x)}\Big)}
\]
 almost surely converges to $1$ as the exponent $s$ converges to the critical exponent $s = 1$, cf. \eqref{av-or-not-1} below.

\bigskip

{\flushleft \bf II. Reconstruction for all $f\in A^2(\D)$.}

\medskip
We are next interested in {\it simultaneous} reconstruction for families of Bergman functions, in particular, for all $f\in A^2(\D)$. 

It is natural to ask how strong the implied set of measure zero arises from \eqref{weight-av}  depends on $f$. In particular, it is natural to ask the question  whether there is a common set of measure zero for which the equality \eqref{weight-av} holds simultaneous for all $f\in A^2(\D)$, which is equivalent to the following
\begin{question-B}
Does the Patterson-Sullivan construction holds for the reproducing kernel-valued function  $x\mapsto K_\D(\cdot, x)$? More precisely, for fixed $z\in \D$, does there exist a subsequence $s_n\to 1^{+}$ such that
\begin{align}\label{one-attempt}
K_\D(\cdot, z) =   w\text{-}\lim_{n\to\infty}  \frac{ \displaystyle \sum_{k =0}^\infty \sum_{x\in Z(\frak{g}_\D)\atop k \le d_\D(x, z) < k + 1} e^{-s_n d_\D(z, x)} K_\D(\cdot, x)}{\displaystyle \sum_{x \in Z(\frak{g}_\D)} e^{-s_n d_\D(z, x)} }, \quad \text{almost surely ? }
\end{align}
Here   ${\displaystyle w\text{-}\lim_{n\to\infty}}$ denotes the weak convergence in the Hilbert space $A^2(\D)$.
\end{question-B}

Applying the Patterson-Sullivan construction for fixed (anti-analytic) Bergman function, for any fixed $w\in \D$ and hence any fixed countable subset of $\D$, almost surely,  the evaluation of \eqref{one-attempt}  at points in this subset  converges to $K_\D(w, z)$. It is however unclear to us whether almost surely, the right hand side of \eqref{one-attempt}  converges to the function $K_\D(\cdot, x)$ at all points in $\D$. It is also unclear to us   whether the function \eqref{one-attempt} converges to the function $K_\D(\cdot, x)$ weakly in $A^2(\D)$. 

In fact, it is even unclear to us whether  the expression of the numerator in \eqref{one-attempt} is well-defined for $s$ near the critical exponent $s = 1$.  In  Proposition \ref{prop-sharp}, we shall prove that   for any $1< s \le \frac{3}{2}$,
\[
\sup_{N\in \N}    \E\left( \Big\| \sum_{k= 0}^N  \sum_{x \in Z(\frak{g}_\D) \atop k \le d_\D(x, z) < k + 1}e^{-s d_\D(x, z)} K_\D(\cdot, x) \Big\|_{A^2(\D)}^2\right) = \infty.
\]

We have the following 
\begin{conjecture}
For $1< s \le \frac{3}{2}$, almost surely, we have 
\[
\sup_{N\in \N }   \Big\| \sum_{k = 0}^N  \sum_{x \in Z(\frak{g}_\D) \atop k \le d_\D(x, z) < k + 1}e^{-s d_\D(x, z)} K_\D(\cdot, x)\Big\|_{A^2(\D)} = \infty.
\]
In particular,  for $1< s \le \frac{3}{2}$, almost surely,  the series 
\begin{align}\label{vec-linear-stat}
\sum_{k= 0}^\infty  \sum_{x \in Z(\frak{g}_\D) \atop k \le d_\D(x, z) < k + 1}e^{-s d_\D(x, z)} K_\D(\cdot, x)
\end{align}
does not converge weakly in $A^2(\D)$ and therefore does not represent a element in $A^2(\D)$. 
\end{conjecture}

\begin{remark*}
In a forthcoming paper,  Conjecture A will be proved for $1< s < \frac{3}{2}$, the proof of which requires substantial efforts. 
\end{remark*}

For getting round the weak convergence issue of the series \eqref{vec-linear-stat}, in stead of considering the weights  $z \mapsto e^{-s d_\D(z, x)}$, we are naturally led to consider {\it compactly supported weights}. It should be noticed that, in the situation of the reconstruction problem for a fixed Bergman function, it is indeed possible to find an explicit one-parameter family of compactly supported  weights  replacing  the weights  $z \mapsto e^{-s d_\D(z, x)}$  inside  \eqref{weight-av} such that the reconstruction of a fixed $f\in A^2(\D)$ at a fixed point via the weighted average still holds: for instance, we may take the weights 
\[
W_r(x) = (1 - |x|^2)\mathds{1}( |x|\le r), \quad W_r^z(x) = W_r\Big(\frac{x-z}{1 - \bar{z}x}\Big).
\]

However, even for general compactly supported weights, we  still obtain a negative result, cf. Theorem \ref{prop-failure-intro}, stating that the generalized Patterson-Sullivan construction with compactly supported weights does not give  simultaneous reconstructions of all functions in $A^2(\D)$, which in particular shows that for a given subset $X\subset \D$, to find an explicit simultaneous reconstruction of all functions in $A^2(\D)$ from their restrictions onto $X$ is in general much harder than to prove that $X$ is an  $A^2(\D)$-uniqueness set. 

An  informal explanation is given  as follows.   Let $W: \D \rightarrow \R^{+}$ be any compactly supported {\it radial} weight function and  set $W^z(x) = W\left(\frac{x-z}{1 - \bar{z}x}\right)$ for any $z\in \D$.    Then  for any $z\in \D$, there exists a universal constant $c_z>0$, such that   
\begin{align}\label{Neg-intro}
\inf_{W}  \E \left[    \left\| \frac{  \displaystyle  \sum_{x\in Z(\mathfrak{g}_\D)} W^z(x) K_\D(\cdot, x)}{\displaystyle  \E  \sum_{x\in Z(\mathfrak{g}_\D)} W^z(x)} - K_\D(\cdot, z) \right\|_{A^2(\D)}^2\right] \ge c_z. 
\end{align}
The  radial assumption  on $W$ is natural (see the analysis of the equality \eqref{Q-s-kernel-mean} below for more details) since it gives, see Lemma \ref{lem-observation} below,  the equality 
\[
\E\Big[\sum_{x \in Z(\mathfrak{g}_\D)} W^z(x) g(x)\Big]  = g(z) \E\Big[\sum_{x \in Z(\mathfrak{g}_\D)} W^z(x)\Big]
\]  for any harmonic function $g$ on $\D$ and any $z\in \D$.

\bigskip

{\flushleft \bf III. Reconstructions for smaller families of functions}

\medskip

Since there is no simultaneous reconstruction for all Bergman functions in $A^2(\D)$, we are naturally led to consider the simultaneous reconstructions for functions in smaller families. And we are able to obtain simultaneous reconstructions  for functions in other spaces of holomorphic or harmonic functions on $\D$ including 
\begin{itemize}
\item weighted Bergman spaces on $\D$ with respect to  weights  growing  rapidly at the boundary of $\D$, cf. Theorem \ref{thm-w-berg}. In particular,  the following weights for any $T>0$: 
\begin{align}\label{def-w-t}
\omega_T(z) = \frac{1}{(1 - |z|^2) \log \left(\frac{4}{1 - |z|^2}\right) \log^{1+ T} \left( \log\left(\frac{4}{1 - |z|^2}\right)\right) }.
\end{align}
\item  reproducing kernel Hilbert space inside the space of harmonic functions on $\D$ whose reproducing kernel satisfies a certain growth condition, cf. Theorem \ref{thm-RKHS-intro} and Example \ref{ex-harm-RKHS}. 
\item Hardy-type spaces such as the Poisson transformations of certain families of finite measures on the torus $\T$, cf. Theorem \ref{thm-intro-poi}.  
\end{itemize} 

Informal descriptions of these simultaneous reconstructions are as follows.    The main  step in obtaining these simultaneous reconstructions is to find checkable sufficient conditions on the kernels (the reproducing kernels or the Poisson kernels) under which, we are able to prove, by establishing a Law of Large Numbers for the kernel-valued linear statistics, that  the kernels satisfy a discrete analogue of the mean-value property: that is, the kernels themselves satisfy the Patterson-Sullivan construction. 

     Let $K(\cdot, \cdot)$ be a reproducing kernel of a reproducing kernel Hilbert space, denoted by $\mathscr{H}(K)$.
For any $z\in \D$ and $s>1$,  similarly to the definition \ref{def-Q-s}, set 
\begin{align}\label{def-Q-kernel}
Q_s(Z(\frak{g}_\D); K(\cdot, \cdot), z):  = \frac{ \displaystyle  \sum_{k =0}^\infty \sum_{x\in Z(\frak{g}_\D)\atop k \le d_\D(x, z) < k + 1} e^{-s d_\D(z, x)} K(\cdot, x)}{\displaystyle \E  \Big(\sum_{x \in Z(\frak{g}_\D)} e^{-s d_\D(z, x)} \Big)}. 
\end{align}
We then need to impose  conditions on the kernel $K(\cdot, \cdot)$ such that 
\begin{itemize}
\item[(R-1)] for fixed $z\in \D$ and $s>1$,  the series appeared as the numerator in the definition  \eqref{def-Q-kernel} converges almost surely in $\mathscr{H}(K)$. 
\item[(R-2)] for fixed $z\in \D$ and $s>1$, the mean of the $\mathscr{H}(K)$-vector-valued random variable $Q_s(Z(\frak{g}_\D); K(\cdot, \cdot), z)$ is equal to $K(\cdot, z)$: 
\begin{align}\label{Q-s-kernel-mean}
\E [ Q_s(Z(\frak{g}_\D); K(\cdot, \cdot), z) ] = K(\cdot, z). 
\end{align} 
\item[(R-3)] for fixed $z\in \D$, the $\mathscr{H}(K)$-vector-valued random variable $Q_s(Z(\frak{g}_\D); K(\cdot, \cdot), z)$ has an asymptotically vanishing variance as $s\to 1^{+}$:
\begin{align}\label{asy-vanish-var}
\lim_{s\to 1^{+}} \Var(Q_s(Z(\frak{g}_\D); K(\cdot, \cdot), z)) =0.
\end{align}
\end{itemize}
Fix any countable subset $\mathcal{D}\subset \D$. 
Clearly, \eqref{Q-s-kernel-mean} and \eqref{asy-vanish-var} together imply that for any fixed $z\in \D$ (and thus also for any $z \in \mathcal{D}$), along a certain fixed subsequence $s_n\to 1^{+}$, we have 
\begin{align}\label{Q-s-to-K}
\lim_{n\to\infty}   \Big\|Q_{s_n}(Z(\frak{g}_\D); K(\cdot, \cdot), z) - K(\cdot, z)\Big\|_{\mathscr{H}(K)} = 0, \quad \text{almost surely,}
\end{align}
and as a consequence, recalling the definition  \eqref{def-Q-s} for $Q_s(Z(\frak{g}_\D); f, z)$, for all $z$ in the fixed countable subset $ \mathcal{D}\subset \D$,  almost surely,  we have 
\begin{align}\label{sim-all-f}
\lim_{n\to\infty}   Q_{s_n}(Z(\frak{g}_\D); f, z)  = f(z), \quad \text{for all $f\in \mathscr{H}(K)$.}
\end{align}

\medskip

The item (R-2) is simpler to analyze.     Write  the expectation $\E [ Q_s(Z(\frak{g}_\D); K(\cdot, \cdot), z) ]$ in integral form in terms of the first intensity of the random set $Z(\frak{g}_\D)$, then the equality  \eqref{Q-s-kernel-mean} is reduced to  a mean-value property of the $\mathscr{H}(K)$-valued function $x\mapsto K(\cdot, x)$. Therefore, we assume that the function $x \mapsto K(\cdot, x)$ is a $\mathscr{H}(K)$-valued harmonic function and the equality \eqref{Q-s-kernel-mean} follows,  cf. Lemma \ref{lem-observation}. Equivalently, we will make the following assumption 
\begin{itemize}
\item[(A-1')] the kernel $K(\cdot, \cdot)$ is harmonic in the second variable.
\end{itemize}
Since $K(\cdot, \cdot)$ is Hermitian, it is also harmonic in the first variable.  We note that the harmonicity  assumption  of the kernel in one variable is equivalent to the assumption that the corresponding reproducing kernel Hilbert space  $\mathscr{H}(K)$ satisfies 
\begin{itemize}
\item[(A-1)]  $\mathscr{H}(K)$ is contained in the space of all harmonic functions on $\D$. 
\end{itemize}

\medskip

Finding  sufficient conditions for the items (R-1) and (R-3) require more efforts.  A checkable and non-trivial such sufficient condition (cf. Assumption \ref{ass-gcb} and  Remark \ref{rem-int-berg} below in \S \ref{sec-chs}) for both (R-1) and (R-3) is  the following
\begin{itemize}
\item[(A-2)]  growth condition on the diagonal $K(w,w)$ of the kernel $K(\cdot, \cdot)$: 
\begin{align}\label{weight-ass}
\int_\D  K(w, w) (1 - |w|^2)^\alpha dA(w)  = o(1/\alpha^2), \quad \text{as $\alpha \to 0^{+}$.}
\end{align}
\end{itemize}
Here $o(1/\alpha^2)$ is the usual little ``$o$" notation, meaning that the left hand side of \eqref{weight-ass}, divided by $1/\alpha^2$, converges to $0$ as $\alpha \to 0^{+}$. 

Indeed, the assumptions (A-1) and (A-2) imply  (cf. Definition \ref{def-MVP}, Proposition \ref{prop-L2-L1-sum} and Lemma \ref{lem-mvp-rep}) that for any $s> 1$, 
\[
\E \Big( \sum_{k =0}^\infty   \Big\|\sum_{x\in Z(\frak{g}_\D)\atop k \le d_\D(x, z) < k + 1} e^{-s d_\D(z, x)} K(\cdot, x)\Big\|_{\mathscr{H}(K)} \Big)<\infty,
\]
which clearly implies (R1). The assumptions (A-1) and (A-2) also imply  (cf. Proposition \ref{prop-var-up-bd-rep}) that  for any $s> 1$, 
\begin{align}\label{var-up-ratio}
\Var(Q_s(Z(\frak{g}_\D); K(\cdot, \cdot), z))  \le   \frac{\displaystyle \int_\D e^{-2s d_\D(x, z)} K(x,x) d\mu_\D(x)}{\displaystyle \Big[\int_{\D} e^{-s d_\D(x, z)} d\mu_{\D} (x)\Big]^2} \xrightarrow{\,\, s \to 1^{+}} 0,
\end{align}
where $\mu_\D$ is the M\"obius invariant area-measure on $\D$, see \eqref{lob-vol} for the precise formula. A key fact in the inequality \eqref{var-up-ratio} is near the critical exponent $s=1$, the term $e^{-2s d_\D(x,z)}$ converges to $0$ as $x$ approaches the boundary of $\D$ and this allows us to impose an extra growth condition on $K(x, x)$ such that the ratio converges to $0$ as $s \to 1^{+}$.


\bigskip

{\flushleft \bf IV. Generalization to other point processes}

\medskip

It is also natural to study  the above reconstruction problems in the setting of general random subsets of $\D$, or more precisely, general {\it point processes} $\Pi$ on $\D$ (precise definitions and background 
on point processes are recalled in  \S \ref{sec-pre}).

In fact, except for  the negative result, Theorem \ref{prop-failure-intro},  on the simultaneous reconstruction for all $f\in A^2(\D)$, which is based on the second correlation structure of the point process $Z(\frak{g}_\D)$, all the other results on simultaneous reconstruction for harmonic functions are based on the following small variance assumption on the linear statistics of the point process $\Pi$ in question: there exists a constant $C>0$ depending only on $\Pi$ such that 
for any function $f: E\rightarrow \C$ with $\E_\Pi [\sum_{x\in \X} | f(x)| + | f(x)|^2] <\infty$, we have  
\begin{align*}
\Var_\Pi\Big(\sum_{x\in \X} f(x)\Big) \le C   \cdot \E_\Pi \Big( \sum_{x\in \X} | f(x)|^2\Big).
\end{align*}
In particular, the above assumption is satisfied by Poisson point processes, determinantal point processes with Hermitian kernel and more generally all negatively correlated point processes, cf. Proposition \ref{lem-var-LS}. 

Although in most of our reconstruction results, we will use the conformal invariance of the first intensity of $Z(\frak{g}_\D)$, this conformal invariance can be relaxed to a certain extend. This point will not be fulled expanded in this paper, we will only give Theorem \ref{thm-intro-non-conf} for a concrete case of point processes with non-conformal-invariant first intensity.

\bigskip

{\flushleft \bf V. Generalization to other backgroud spaces}

\medskip

We have until now only discussed our results in the case of unit disk $\D$. However, our formalism works for an abstract setting in terms of measure metric spaces. In particular, the harmonicity is replaced by a mean-value property, see \S \ref{sec-general-Hilbert} for precise definitions. In concrete cases, this formalism works in the case of complex, real and quaternionic  hyperbolic spaces of arbitrary dimension. 

In this paper we limit ourselves to the complex and the real hyperbolic spaces, but we expect our formalism  to apply more generally, in particular, to more general symmetric spaces.

\subsection{Structure of the paper}

The paper is organized as follows. 
\begin{enumerate}
\item \S \ref{sec-one-d} is devoted to the main results in one dimensional disk case and contains the following results. 
\begin{enumerate}
\item In \S \ref{sec-fix-Bergman}, we state the reconstruction result, Proposition \ref{prop-intro-single-bergman},  for a fixed Bergman function. 
\item In \S \ref{sec-imp-Bergman}, a negative result on simultaneous reconstruction for Bergman functions is stated in  Theorem \ref{prop-failure-intro}. 
\item \S \ref{sec-pos-sim-rec} contais simultaneous reconstruction results in the case of one dimensional disk:
\begin{enumerate}
\item Theorem \ref{thm-w-berg} deals with the simultaneous reconstruction for weighted Bergman functions and the point process $Z(\frak{g}_\D)$.
\item Theorem \ref{thm-RKHS-intro} deals with the simultaneous reconstruction for harmonic functions inside a reproducing kernel Hilbert space and the point process $Z(\frak{g}_\D)$.
\item Theorem \ref{thm-intro-poi} deals with the simultaneous reconstruction for for Hardy-type functions and the point process $Z(\frak{g}_\D)$.
\item Theorem \ref{thm-general-pp} deals with similar simultaneous reconstruction results as  stated in Theorems \ref{thm-w-berg}, \ref{thm-RKHS-intro} and \ref{thm-intro-poi}   for   weighted Bergman functions, harmonic functions inside a reproducing kernel Hilbert space  and Hardy-type functions  but for more general point process with conformally invariant first intensity. 
\item  Theorem \ref{thm-intro-non-conf} deals with a new simultaneous reconstruction result for the classical Bergman functions in the case of general point processes with a specific first intensity which is in particular non-conformal. 
\item \S \ref{sec-sharp} is devoted to the discussion of sharpness in a certain sense of our simultaneous reconstruction results in the case of the determinantal point process $Z(\frak{g}_\D)$. 
\end{enumerate}
\end{enumerate}
\item \S \ref{sec-pre} is devoted to the preliminaries on the following materials.
\begin{enumerate}
\item Complex hyperbolic spaces: Bergman metric, weighted Bergman spaces and $\MM$-harmonic functions. 
\item Real hyperbolic spaces.
\item Point processes, determinantal point processes, negative correlation of point processes. 
\end{enumerate}
\item  \S \ref{sec-gen} is devoted to the reconstruction of Hilbert-space vector-valued harmonic functions in abstract setting of measure-metric spaces. 
\begin{enumerate}
\item  In \S  \ref{sec-general-Hilbert}, we establish, under appropriate assumptions,  the reconstruction of a fixed Hilbert space vector-valued harmonic function at a fixed point. 
\item In \S \ref{sec-hil-pos}, the almost sure statements on the reconstructions obtained in \ref{sec-general-Hilbert} is extended, under an additional positivity assumption on the function and appropriate additional assumptions on the point process, to the stronger almost sure statements simultaneously for all points. The limit along a subsequence $s_n \to h_M$ is extended to the limit $s\to h_M$, where $s$ are exponent in the definitions of the weight functions and $h_M$ is the critical value for the exponent.   
\end{enumerate}
\item   In \S \ref{sec-chs} and \S \ref{sec-rhs}, the abstract formalism in \S \ref{sec-gen} is applied to the concrete case of complex and real hyperbolic spaces of arbitrary dimension. 
\item In \S \ref{sec-more-disk}, we deal with determinantal point processes on complex hyperbolic spaces induced by classical weighted Bergman kernels and prove  the negative result on simultaneous reconstructions for all Bergman functions. 
\end{enumerate}

\bigskip

\noindent {\bf{Acknowledgements.}} Mikhael Gromov taught the Patterson-Sullivan theory to the older of us  in 1999; we are greatly indebted to him. We are deeply grateful to Alexander Borichev, S\'ebastien Gou\"ezel,  Pascal Hubert, Alexey Klimenko and Andrea Sambusetti for useful discussions. Part of this work was done during a visit to the Centro  De Giorgi della Scuola Normale Superiore di Pisa. We are deeply grateful to the Centre for its warm hospitality.
AB's research has received funding from the European Research Council (ERC) under the European Union's Horizon 2020 research and innovation programme under grant agreement No 647133 (ICHAOS). YQ's research is supported by  the National Natural Science Foundation of China, grants NSFC Y7116335K1, NSFC 11801547 and NSFC 11688101.

\section{Main results for one dimensional disk}\label{sec-one-d}
\subsection{Reconstruction of a fixed Bergman functions}\label{sec-fix-Bergman}
 By adapting the classical argument in  Patterson-Sullivan construction, see Patterson \cite{Patterson-acta} and Sullivan \cite{Sullivan-IHES}, for {\it fixed} $f \in A^2(\D)$ and {\it fixed countable} subset  $\mathcal{D} \subset \D$,  we obtain a preliminary Proposition \ref{prop-intro-single-bergman} for recovering the values $f(z)$ for $z\in \mathcal{D}$ from the restriction $f|_X$ of the function $f$ onto a typical realization $X = Z(\mathfrak{g}_\D)$.


\begin{proposition}[Reconstruction for fixed Bergman function]\label{prop-intro-single-bergman}
Fix $f \in A^2(\D)$, a countable dense subset $\mathcal{D} \subset \D$ and a  sequence $(s_n)_{n\ge 1}$ such that $s_n>  1$ and  $\sum_{n=1}^\infty (s_n-1)^2 <\infty$.  Almost surely, the realization $X = Z(\mathfrak{g}_\D)$ satisfies: For all $n\in\N$ and all $z\in \mathcal{D}$, we have 
\begin{align}\label{pv-ab-cv}
 \sum_{k=0}^\infty \Big| \sum_{x\in X\atop k \le d_\D(x, z) < k + 1} e^{-s_n d_\D(z, x)} f(x) \Big|<\infty
\end{align}
and moreover
\begin{align}\label{rec-single-f-intro}
 f(z) =   \lim_{n\to\infty} \frac{ \displaystyle \sum_{k =0}^\infty \sum_{x\in X\atop k \le d_\D(x, z) < k + 1} e^{-s_n d_\D(z, x)} f(x)}{\displaystyle \sum_{x \in X} e^{-s_n d_\D(z, x)} }.
\end{align}
\end{proposition}

 The almost sure statement in Proposition \ref{prop-intro-single-bergman} is of course immediately extended to any fixed countable dense family $\mathscr{F} \subset A^2(\D)$: Almost surely, for each function $f \in \mathscr{F}$, its values $f(z)$ with $z\in \mathcal{D}$ can be reconstructed by the explicit  limit equality \eqref{rec-single-f-intro} from its restriction onto $X  = Z(\mathfrak{g}_\D)$. The implied subset for the configuration $X$ of full measure in our statement might however depend on this  family $\mathscr{F}$ and there is no reason that the limit equality \eqref{rec-single-f-intro} can be extended to its closure $\overline{\mathscr{F}}^{A^2(\D)}    = A^2(\D)$.

\subsection{Impossibility of simultaneous reconstruction for Bergman functions}\label{sec-imp-Bergman}

Recall that for  compactly supported {\it radial} weight function $W: \D \rightarrow \R^{+}$ and  for any $z\in \D$, we set $W^z(x) = W\left(\frac{x-z}{1 - \bar{z}x}\right)$. 
We obtain the following  result showing the impossibility of our simultaneous reconstruction in the general setting of radial weights.

\begin{theorem}\label{prop-failure-intro}
For any $z\in \D$, we have 
\[
\inf_{W}  \E \left[    \left\| \frac{  \displaystyle  \sum_{x\in Z(\mathfrak{g}_\D)} W^z(x) K_\D(\cdot, x)}{\displaystyle  \E  \sum_{x\in Z(\mathfrak{g}_\D)} W^z(x)} - K_\D(\cdot, z) \right\|_{A^2(\D)}^2\right] > 0,
\]
where the infimum runs over all compactly supported bounded radial weights $W: \D\rightarrow \R^{+}$.   
\end{theorem}

The proof of  Theorem \ref{prop-failure-intro} relies on the following estimate: there exists $c>0$ such that for any $z\in \D$ and any compactly supported bounded radial $W: \D \rightarrow \R^{+}$, 
\begin{multline}\label{low-bd-var}
  \E \left[   \left\|   \sum_{x\in Z(\mathfrak{g}_\D)} W^z(x) K_\D(\cdot, x) -  K_\D(\cdot, z) \Big[\E  \sum_{x\in Z(\mathfrak{g}_\D)} W^z(x) \Big]\right\|_{A^2(\D)}^2\right]  \ge
\\
 \ge  \frac{c}{( 1 - |z|^2)^{2}} \int_{\D} \int_{\D} \frac{| W(w_1) - W(w_2)|^2}{(1 - |w_1|^2 |w_2|^2)^5}  dA(w_1) dA(w_2).
\end{multline}
The inequality \eqref{low-bd-var} and a more precise equality  in Remark \ref{rem-precise-f} below  will be proved using the holomorphic structure of $A^2(\D)$, see Proposition \ref{prop-Sob} for the details.  The generalization  to arbitrary dimension of the above lower bound \eqref{low-bd-var} is given in Proposition \ref{prop-new-var-f}.  

\subsection{Simultaneous reconstructions}\label{sec-pos-sim-rec}

\subsubsection{Weighted Bergman spaces}

One may consider the simultaneous reconstruction for weighted Bergman functions. More precisely, let $A^2(\D, \omega)$ denote the weighted Bergman space  with respect to a weight $\omega$ on $\D$ (see \S \ref{sec-pre} for precise definition). Under mild condition on $\omega$, the space $A^2(\D,\omega)$ is a reproducing kernel Hilbert space and let $K_\omega$ denote its reproducing kernel.

\begin{theorem}\label{thm-w-berg}
Fix any countable dense subset $\mathcal{D} \subset \D$. Assume that the kernel $K_\omega(\cdot, \cdot)$ satisfies the assumption \eqref{weight-ass}. Then   there exists a sequence $(s_n)_{n\ge 1}$ in $(1, \infty)$ converging to $1$ such that almost surely, the realization $X = Z(\mathfrak{g}_\D)$ satisfies: the absolute convergence \eqref{pv-ab-cv} and the limit equality \eqref{rec-single-f-intro} hold simultaneously for all $f \in A^2(\D, \omega)$ and all $z\in \mathcal{D}$. 
\end{theorem}

\begin{example}
For any $T >0$,  let $A^2(\D, \omega_T)$ denote the weighted Bergman space  with respect to the weight $\omega_T$ on $\D$ defined by the formula  \eqref{def-w-t}. 
Then the reproducing kernel of $A^2(\D, \omega_T)$ satisfies the assumption \eqref{weight-ass}.  See Proposition \ref{prop-w-cb} for the details.
\end{example}

\subsubsection{Reproducing kernel Hilbert spaces}

Let $\mathrm{Harm}(\D)$ denote the space of all complex-valued harmonic functions on $\D$. Then weighted Bergman spaces in Theorem \ref{thm-w-berg} can be replaced by more general  {\it reproducing kernel Hilbert space (RKHS for abreviation)}  $\mathscr{H}(K) \subset \mathrm{Harm}(\D)$ whose reproducing kernel $K$ satisfies certain growth condition.

\begin{theorem}\label{thm-RKHS-intro}
Let $\mathscr{H}(K) \subset \mathrm{Harm}(\D)$ be a  reproducing kernel Hilbert space whose reproducing kernel $K$ satisfying the condition \eqref{weight-ass}. Fix any countable dense subset $\mathcal{D} \subset \D$. Then  there exists a sequence $(s_n)_{n\ge 1}$ in $(1, \infty)$ converging to $1$ such that almost surely, the realization $X = Z(\mathfrak{g}_\D)$ satisfies: the absolute convergence \eqref{pv-ab-cv} and the limit equality \eqref{rec-single-f-intro} hold simultaneously for all $f \in \mathscr{H}(K)$ and all $z\in \mathcal{D}$. 
\end{theorem}

\begin{example}\label{ex-harm-RKHS}
The following are the simplest examples of reproducing kernels of a RKHS inside $\mathrm{Harm}(\D)$ verifying the condition \eqref{weight-ass}: 
\[
K(z, w) = \sum_{n=0}^\infty a_n z^n \bar{w}^n + \sum_{n=1}^\infty a_{-n} \bar{z}^n w^n,
\]
where $(a_n)_{n\in \Z}$ is a sequence of non-negative numbers  such that 
\begin{align*}
\lim_{|n | \to \infty } \frac{a_n }{ \log (|n|+2) } = 0.
\end{align*}
See Theorem \ref{thm-rep-cb} and Proposition \ref{prop-cb-kernel} below for the details. 
\end{example}

\subsubsection{Hardy-type spaces}

For harmonic functions which are {\it Poisson transformations} of  signed measures on the unit circle, the simultaneous reconstruction can be significantly improved as follows.

Denote by $\T = \partial \D$ the unit circle. Recall that the Poisson kernel  $P: \D \times \T \rightarrow \R^{+}$ is given by the formula 
\begin{align}\label{def-Poi-kernel}
P(z, \zeta)=   \frac{1- |z|^2}{ |1 - \bar{z} \zeta|^2} =  \frac{1- |z|^2}{ |z - \zeta|^2}.
\end{align}
The Poisson transformation of a signed Borel measure $\nu$  on $\T$ of finite total variation  is a harmonic function on $\D$ and  is  defined by 
\begin{align}\label{poi-trans}
P[\nu] (z): = \int_\T P(z,  \zeta)  d \nu ( \zeta).
\end{align}
Given any finite Positive Borel measure $\mu$ on $\T$, set
\begin{align}\label{def-s-hardy}
h^2(\D; \mu) : &= \left\{f: \D\rightarrow \C\Big|  f = P[g\mu], \, g \in L^2(\mu)\right\},
\end{align}
where $L^2(\mu) = L^2(\mu; \C)$ is the space of $\C$-valued $\mu$-square-integrable functions on $\T$.

\begin{theorem}\label{thm-intro-poi}
Let $\mu$ be any Borel probability measure on $\T$. Then almost surely, the realization $X = Z(\mathfrak{g}_\D)$ satisfies: for  any $ f  \in h^2(\D; \mu)$, any $z\in \D$ and any $s >1$, the series  $\displaystyle \sum_{x \in X}   e^{-s d_\D(x, z)} f(x)$ converges absolutely  and 
\[
f(z) =    \lim_{s\to 1^{+}}  \frac{ \displaystyle{ \sum_{x \in X}   e^{-s d_\D(x, z)} f(x) } }{\displaystyle{  \sum_{x \in X}   e^{-s d_\D(x, z)} }}. 
\]
\end{theorem}

\subsubsection{General point processes}
The very general simultaneous reconstruction results are obtained in Theorem \ref{thm-H-L}, Theorem \ref{thm-general-pos} and Theorem \ref{thm-n-sub}  which have immediate consequences in the cases of complex, real and quaternionic hyperbolic spaces of arbitrary dimension.  

For concreteness, here we state the immediate consequences of our main results in the one dimensional hyperbolic disk case. And for further reference, we introduce our assumption on the point processes in general setting.  Let $E$ be a complete separable metric  space, and let  $\Pi$ be a point process on $E$. Our main assumption on $\Pi$ is the following small variance assumption  \eqref{var-less-2-moment} on the linear statistics.

We will use Convention \ref{convention-notation} in \S \ref{sec-pre}. 
\begin{assumption}[Small variance]\label{ass-pp}There exists a constant $C>0$ depending only on $\Pi$ such that 
for any function $f: E\rightarrow \C$ with $\E_\Pi [\sum_{x\in \X} | f(x)| + | f(x)|^2] <\infty$, we have  
\begin{align}\label{var-less-2-moment}
\Var_\Pi\Big(\sum_{x\in \X} f(x)\Big) \le C   \cdot \E_\Pi \Big( \sum_{x\in \X} | f(x)|^2\Big).
\end{align}
\end{assumption}

Assumption \ref{ass-pp} immediately implies  
\begin{proposition}\label{rem-scalar-vector} For any Hilbert space $\mathscr{H}$, the  inequality \eqref{var-less-2-moment}  holds, with the same constant,  for any function $f: E \rightarrow \mathscr{H}$  satisfying $\E_\Pi [\sum_{x\in \X} \| f(x)\|_{\mathscr{H}} + \| f(x)\|_{\mathscr{H}}^2] <\infty$.
\end{proposition}

One checks easily, see  Lemma \ref{lem-var-LS} below, that Poisson point processes,  determinantal point processes induced by Hermitian correlation kernels and more generally all {\it negatively correlated} point processes (see \S \ref{sec-pre} for the precise definition) satisfy Assumption \ref{ass-pp}.

We start with the situation where the first intensity measure of the point process on $\D$ is conformally invariant, that is, proportional to the  Lobachevsky volume measure $\mu_\D$ on $\D$  given by  
\begin{align}\label{lob-vol}
d \mu_\D(z) = \frac{dA(z)}{(1 - |z|^2)^2}.
\end{align}

\begin{theorem}\label{thm-general-pp}
Let $\Pi$ be a point process on $\D$ satisfying  Assumption \ref{ass-pp} and having conformallly invariant first intensity measure. Then the simultaneous reconstruction statements in Theorems \ref{thm-w-berg},  \ref{thm-RKHS-intro} and \ref{thm-intro-poi} all hold for the point process $\Pi$. 
\end{theorem}

We also obtain a simultaneous reconstruction  of point processes with non-conformal intensity. Recall that for any $\alpha> -1$, let $A_\alpha^2(\D)$ denote the weighted Bergman space associated with the classical weight $( 1- |z|^2)^{\alpha}$. 

\begin{theorem}\label{thm-intro-non-conf}
 Let $\alpha> -1$ and $\beta \ge \alpha + 2> 1$.  Fix any countable dense subset $ \mathcal{D} \subset \D$ and a  sequence $(s_n)_{n\ge 1}$  with $s_n > \beta$ such that $\sum_{n= 1}^\infty (s _n -\beta) <\infty$.  Let $\Pi$ be a point process on $\D$ satisfying  Assumption \ref{ass-pp} with first intensity measure proportional to $\frac{dA(x)}{(1 - |x|^2)^{\beta + 1}}$.  Then $\Pi$-almost any $X \subset \D$ satisfies: For all $f \in A^2_\alpha(\D)$, all $z\in \mathcal{D}$ and   any $n$, we have 
\[
\sum_{k=0}^\infty \Big| \sum_{x \in X \atop k\le d_\D (z,x) < k+1} e^{-s_n d_\D(z, x)}  \left(\frac{|1 - x \bar{z}|^2}{1 - |z|^2}\right)^{\beta-1} f(x) \Big|< \infty
\]
and moreover
\[
 f(z) = \lim_{n\to\infty} \frac{\displaystyle{\sum_{k=0}^\infty \sum_{x \in X \atop k\le d_\D (z,x) < k+1} e^{-s_n d_\D(z, x)} \left(\frac{|1 - x \bar{z}|^2}{1 - |z|^2}\right)^{\beta-1} f(x) }}{\displaystyle{  \sum_{x \in X} e^{-s_n d_\D(z,x)} \left(\frac{|1 - x \bar{z}|^2}{1 - |z|^2}\right)^{\beta-1} }}.
\]
\end{theorem}

\subsubsection{Sharpness of the simultaneous reconstruction in determinantal case.}\label{sec-sharp}

Now we focus on the determinantal point process given by $Z(\frak{g}_\D)$ introduced in  \S \ref{sec-formulation}.  Recall that for any $\alpha > -1$,   the classical weighted Bergman space defined by 
\[
A_{\alpha }^2(\D): = \left\{f : \D\rightarrow \C \Big| \text{$f$ is holomorphic and $\int_\D | f(z)|^2 (1 - |z|^2)^\alpha dA(z) <\infty$}\right\}.
\]
Recall the definition   \eqref{def-w-t} of the weight $\omega_T$. Since for any $T> 0$ and $\alpha > -1$, there exists a constant $C_{T, \alpha}>0$, such that 
\[
(1-|z|^2)^\alpha \le  \frac{C_{T, \alpha}}{(1 - |z|^2) \log \left(\frac{4}{1 - |z|^2}\right) \log^{1+ T} \left( \log\left(\frac{4}{1 - |z|^2}\right)\right) }, \quad \text{for all $z\in \D$,}
\]
 we have 
\[
 \bigcup_{T> 0} A^2(\D, \omega_T)  \subset \bigcap_{\alpha > -1} A_\alpha^2(\D).  
\]

We may consider the weight $\omega_T$ as limit situation of the weight $(1 - |z|^2)^{\alpha}$ as $\alpha$ approaches the critical value $-1$. Our simultaneous reconstructions are sharp in the  following sense: in the case of the determinantal point process $Z(\frak{g}_\D)$,  the Patterson-Sullivan construction gives simultaneous reconstructions for the family $ \bigcup_{T> 0} A^2(\D, \omega_T)$, while fails to give the simultaneous reconstructions for the family $ \bigcap_{\alpha > -1} A_\alpha^2(\D)$. More precisely, let $K^\alpha (\cdot, \cdot)$ denote the reproducing kernel of $A^2_\alpha(\D)$, we have 
\begin{proposition}\label{prop-sharp}
Let $\alpha > -1$ and  $1 < s \le \frac{3 + \alpha}{2}$. The series  of $A^2_\alpha(\D)$-vector-valued random variables 
\[
\sum_{k= 0}^\infty  \sum_{x \in Z(\frak{g}_\D) \atop k \le d_\D(x, z) < k + 1}e^{-s d_\D(x, z)} K^\alpha(\cdot, x)
\]
does not converge in  the space $L^2(A^2_\alpha(\D))$ of square-integrable  $A^2_\alpha(\D)$-vector-valued random variables.  More precisely, we have 
\[
\sup_{N\in \N}    \E\left( \Big\| \sum_{k= 0}^N  \sum_{x \in Z(\frak{g}_\D) \atop k \le d_\D(x, z) < k + 1}e^{-s d_\D(x, z)} K^\alpha(\cdot, x) \Big\|_{A^2_\alpha(\D)}^2\right) = \infty.
\]
\end{proposition}


\subsection{Open problems}
 It is natural to  suggest the following open problems. 
\begin{itemize}
\item Construct in deterministic way  explicit subsets $X\subset \D$ satisfying the simultaneous reconstruction properties as in Theorems \ref{thm-w-berg}, \ref{thm-intro-poi} or \ref{thm-intro-non-conf}.   
\item Give geometric sufficient conditions for  subsets $X\subset \D$ satisfying the simultaneous reconstruction properties as in Theorems \ref{thm-w-berg}, \ref{thm-intro-poi} or \ref{thm-intro-non-conf}. 
\item Similar problems for complex hyperbolic spaces of higher dimensions.   
\end{itemize}



\section{Preliminaries}\label{sec-pre}

\subsection{Complex hyperbolic spaces}\label{sec-intro-chs}
\subsubsection{Bergman metric on complex hyperbolic spaces}
Let $d\in \mathbb N$.  For any $z, w \in \C^d$, write $z \cdot w  = \sum_{k=1}^d z_k w_k;  \bar{z} = (\bar{z}_1, \cdots, \bar{z}_d);  |z| = \sqrt{ z \cdot \bar{z}}$.
Let $\D_d  = \{z  \in \C^d: |z|  < 1\}$ be the unit ball in $\C^d$ endowed with the normalized Lebesgue measure $dv_d(z)$ such that $v_d(\D_d) = 1$. 

Recall that any bounded complex domain carries a natural Riemannian metric, called the Bergman metric, cf. e.g. Krantz \cite[Chapter 1]{Krantz},  defined in terms of the reproducing kernel of the space of square-integrable holomorphic functions on our domain and thus, by definition, invariant under biholomorphisms. In the particular case of $\D_d$, the Bergman metric takes the form
\[
d s_B^2 : =  4 \frac{| dz_1|^2 + \cdots + |dz_d|^2}{1 - |z|^2} + 4 \frac{| z_1 dz_1 + \cdots  + z_d dz_d|^2 }{(1 - |z|^2)^2}.
\]
Let $d_B(\cdot, \cdot)$ denote the distance under the Bergman metric. The ball $\D_d$ endowed with the metric $d_B$ is a model for the complex hyperbolic space. 

For $w \in \D_d \setminus \{0\}$, set
\begin{align}\label{inv-auto}
\varphi_w(z) := \frac{w - \frac{z \cdot \bar{w}}{|w|^2}w - \sqrt{1 - |w|^2}\big(z - \frac{z \cdot \bar{w}}{|w|^2} w\big)  }{1 - z \cdot \bar{w}}.
\end{align}
For $w = 0$, set $\varphi_w(z) = - z$.  We mention that in dimensional $d=1$ case, $\varphi_w(z) = \frac{w-z}{1 - z \bar{w}}$ is the classical M\"obius transformation of the unit disk.  By \cite[Theorem 2.2.2]{Rudin-ball}, the map $\varphi_w$ defines a biholomorphic involution of $\D_d$ 
interchanging $w$ and $0$: we have $\varphi_w(0) = w, \varphi_w(w) = 0, \varphi_w (\varphi_w (z)) = z$ for all $z\in \D_d$. For any $z, w \in \D_d$, we have
\begin{align}\label{Berg-metric-def}
d_B(z, w) = \log \left(\frac{1 +  |\varphi_w(z)|}{ 1- | \varphi_w(z)|}\right).
\end{align}

The volume measure $\mu_{\D_d}$  associated to the Bergman metric, up to a multiplicative constant, is given by 
\begin{align}\label{Berg-vol-def}
d\mu_{\D_d} (z) = \frac{ d v_d(z)}{(1 - |z|^2)^{d+1}},
\end{align}

\subsubsection{Weighted Bermgan spaces}

By a weight on $\D_d$, we mean a {\it finite} positive Borel measure on $\D_d$. If a weight $\omega$ is absolutely continuous with respect to the Lebesgue measure $dv_d$, we will not distinguish the measure $\omega$  and the density $\frac{d \omega}{dv_d}(z)$. Given a weight $\omega$ on $\D_d$ and  $1\le p < \infty$, we denote by $A^p(\D, \omega)$ the weighted Bergman space defined by
\[
A^p(\D_d, \omega): = \left\{f : \D_d\rightarrow \C\Big| \text{$f$ is holomorphic and $\int_{\D_d} | f(z)|^p d\omega(z) <\infty$}\right\}.
\]
When $p = 2$, under milder assumption on the weight $\omega$, the space $A^2(\D_d, \omega)$ is a reproducing kernel Hilbert space, that is, there exists a non-negative definite kernel $K^\omega: \D_d \times \D_d\rightarrow \C$ such that for any $x\in \D_d$, the function $z \mapsto K^\omega(z, x)$ belongs to $A^2(\D_d, \omega)$ and 
 \[
f(x) = \langle f, K(\cdot, x) \rangle_{A^2(\D_d, \omega)} = \int_{\D_d}  K^\omega(x, z) f(z) d\omega(z)
\]
 for any $f \in A^2(\D_d, \omega)$ and any $x \in \D_d$. 

\subsubsection{$\MM$-harmonic functions}
The Bergman Laplacian  $\widetilde{\Delta}$ on $\D_d$ is given by the formula
$
\widetilde{\Delta} = (1 - |z|^2) \sum_{i, j } (\delta_{ij} - z_i \bar{z}_j) \frac{\partial^2}{\partial z_i \partial\bar{z}_j}.
$
A function $f \in C^2(\D_d)$ is called  $\MM$-harmonic if $\widetilde{\Delta} f\equiv 0$ on $\D_d$. While holomorphic functions on $\D_d$ are $\MM$-harmonic, an Euclidean harmonic function on $\D_d$ need not to be.

Set $\Sph_d = \{z \in \C^d: |z|=1\}$, let  $\sigma_{\Sph_d}$ be the normalized surface measure on $\Sph_d$. The Poisson-Szeg\"o kernel $P^b: \D_d \times \Sph_d \rightarrow \R^{+}$  is defined by the formula
\begin{align}\label{def-Poisson-Szego}
P^b(w, \zeta)  = \frac{(1-|w|^2)^d}{|1 - \zeta\cdot \bar{w}|^{2d}}, \quad w \in \D_d, \zeta \in \Sph_d.
\end{align}
The Poisson transformation of a signed Borel measure $\nu$  on $\Sph_d$ of finite total variation is an  $\MM$-harmonic function on $\D_d$ and   is  defined by 
\begin{align}\label{def-cb-poi}
P^b[\nu] (z): = \int_{\Sph_d} P^b(z,  \zeta)  d \nu ( \zeta).
\end{align}

\subsection{Real hyperbolic spaces}\label{sec-intro-rhs}
Let $m  \ge 2$ be a positive integer and let $\B_m \subset \R^m$ be the open unit ball  endowed with the normalized Lebesgue measure $dV_m(x)$ such that $V_m(\B_m) = 1$. The Poincar\'e metric on $\B_m$, see e.g. Stoll \cite{Stoll-hyper-ball}, is defined by the formula
\begin{equation}\label{metric-real-ball}
ds_h^2 =  4 \frac{dx_1^2 + \cdots + dx_m^2}{(1 - |x|^2)^2}.
\end{equation}
Let $d_h(\cdot, \cdot)$ denote the distance under the Poincar\'e metric.  The ball $\B_m$ endowed with the metric $d_h$ is a model for the $m$-dimensional Lobachevsky space.  For any $a\in \B_m$, set 
\begin{align}\label{def-psi-a}
\psi_a (x): =  \frac{a | x-a|^2  + (1 - |a|^2) (a - x)}{ | x- a|^2 + ( 1- |a|^2) (1 - |x|^2)}.
\end{align}
By \cite[Theorems 2.1.2 and 2.2.1]{Stoll-hyper-ball},  $\psi_a$ is an involutive isometry of $\B_m$ interchanging $0$ and $a$: 
$\psi_a(0) = a, \psi_a(a) = 0, \psi_a (\psi_a (x)) = x$ for all $x\in \B_m$.  For $a,b\in \B_m$, we have
\begin{align}\label{def-d-h}
d_h(a, b) = \log \left(  \frac{1 + | \psi_a(b) | }{1  - | \psi_a(b)|}\right). 
\end{align}
The volume measure $\mu_{\B_m}$ associated to the metric \eqref{metric-real-ball}, up to a multiplicative constant,   is given by 
\begin{align}\label{hyper-vol-def}
d\mu_{\B_m} (x) : = \frac{dV_m(x)}{(1 - |x|^2)^m},
\end{align}
and the  hyperbolic Laplacian $\Delta_h$  on $\B_m$ associated to the metric \eqref{metric-real-ball} is given by the formula $
\Delta_h = ( 1 -|x|^2)^2 \sum_{i = 1}^m \frac{\partial^2}{\partial x_i^2} + 2 (m-2) (1  - |x|^2) \sum_{i = 1}^m x_i \frac{\partial }{\partial x_i}$.  A function $f \in C^2(\B_m)$ satisfying $\Delta_h f \equiv 0$ is called $\HH$-harmonic.

Let $S^{m-1} = \partial \B_m$ be the unit sphere in $\R^m$, equipped with  the normalized surface measure $\sigma_{S^{m-1}}$. The hyperbolic Poisson kernel $P^h: \B_m \times S^{m-1} \rightarrow \R^{+}$  of $\B_m$ is given by 
\begin{align}\label{def-hyper-p-kernel}
P^h (x, t) = \left(\frac{ 1 - |x|^2}{| x - t|^2}\right)^{m-1}.
\end{align}
The Poisson transformation of a signed Borel measure $\nu$  on $S^{m-1}$ of finite total variation  is an $\HH$-harmonic function on $\B_m$  and is  defined by 
\begin{align}\label{def-rb-poi}
P^h[\nu] (x): = \int_{S^{m-1}} P^h(z,  \zeta)  d \nu ( \zeta).
\end{align}

\subsection{Point processes}\label{sec-pp}

\subsubsection{Spaces of configurations and point processes}\label{sec-conf-pp}
Let $E$ be a locally compact metric complete separable space.
A configuration $X$ on $E$ is  a  collection of points of $E$, possibly with multiplicities and considered without regard to order,  such that any relatively compact subset  $B\subset E$ contains only finitely many points; the number of point of $X$ in $B$ is denoted $\#_B(X)$.  A configuration is called simple if all points inside have multiplicity one.  Let $\Conf(E)$ denote the space of all configurations on $E$.  A configuration $X \in \Conf(E)$  may be identified  with a purely atomic Radon measure $\sum_{x \in X} \delta_x$,  where $\delta_x$ is the Dirac mass at the point $x$, and  the space $\Conf(E)$ is a complete separable metric space with respect to the vague topology on the space of Radon measures on $E$. The Borel sigma-algebra on $\Conf(E)$ is the smallest sigma-algebra on $\Conf(E)$ that  makes all the mappings $X \mapsto \#_B(X)$ measurable, with $B$ ranging over all relatively compact Borel subsets of $E$. 
A Borel probability measure $\Pi$ on $\Conf(E)$ is called a {\it point process} on $E$. A point process $\Pi$ is called simple, if $\Pi$-almost every configuration is simple.  For further background on  point processes, see, e.g.,   Daley and Vere-Jones \cite{DV-1}, Kallenberg \cite{Kallenberg}. 

Let $\Pi$ be a simple point process on $E$.  For any integer $n\ge 1$, we say that a $\sigma$-finite measure $\xi_\Pi^{(n)}$ on $E^n$ is the $n$-th correlation measure of $\Pi$ if for any bounded compactly supported function $\phi: E^n \rightarrow \C$, we have 
\begin{align*}
\E_\Pi\left( \sum_{x_1, \dots, x_n \in \X}^*  \phi(x_1, \dots, x_n) \right)=   \int_{E^n} \phi(y_1, \dots, y_n)  d\xi_\Pi^{(n)} (y_1, \cdots, y_n). 
\end{align*}
 Endow $E$ with a reference $\sigma$-finite Radon measure $\mu$. If $\xi_\Pi^{(n)}$ is absolutely continuous to the measure $\mu^{\otimes n}$, then the Radon-Nikodym derivative 
\[
\rho_n^{(\Pi)}(x_1, \cdots, x_n) := \frac{d\xi_\Pi^{(n)}}{d\mu^{\otimes n}} (x_1, \cdots, x_n) \quad \text{where $(x_1, \cdots, x_n) \in E^n$,}
\]
is called the $n$-th correlation function  of $\Pi$ with respect to the reference measure $\mu$. The first correlation measure  of a point process  is also called its first intensity.

\begin{convention}\label{convention-notation}
 We make a distinction between a fixed configuration and the configurations used as the integration variable by writing $X$ as a fixed configuration  and $\X$ as integration variable. For instance,  we  frequently use the notation  
\[
\frac{\displaystyle \sum_{x \in X} f(x) }{\displaystyle \E_\Pi \Big( \sum_{x \in \X} f(x)  \Big)  }: = \frac{\displaystyle \sum_{x \in X} f(x) }{\displaystyle \int_{\Conf(E)} \sum_{x \in \X} f(x)   d \Pi (\X) }.
\]
\end{convention}
 
\subsubsection{Determinantal point processes}
 Let $K$ be a {\it locally trace class positive contractive} operator on the complex Hilbert space $L^2(E,\mu)$. The local trace class assumption implies that $K$ is an integral operator and by slightly abusing the notation, we denote the kernel of the operator $K$ again by $K(x, y)$.  By  a theorem obtained by Macchi  \cite{Macchi-DP} and  Soshnikov \cite{DPP-S}, as well as by Shirai and Takahashi  \cite{ST-DPP}, the kernel $K$ induces a unique simple point process $\Pi_K$ on $E$ such that for any positive integer $n\in \N$, the $n$-th correlation function, with respect to the reference measure $\mu$,  of the point process $\Pi_K$ exists and is given by 
\[
\rho_n^{(\Pi_K)} (x_1, \cdots, x_n)  = \det(K(x_i, x_j))_{1 \le i, j \le n}.
\]
The point process $\Pi_{K}$ is called  the determinantal point process induced by the kernel $K$.

\subsubsection{Negatively correlated point processes}

Fix a  reference Radon measure $\mu_E$   on $E$, assume that the correlation measures of $\Pi$ are absolutely continuous with respect to  corresponding tensor powers of $\mu_E$, let   $\rho_1^{(\Pi)}$ and $\rho_2^{(\Pi)}$ be the first and the second order correlation functions  of $\Pi$ with respect to  $\mu_E$.  We say that $\Pi$ is {\it negatively correlated}, if  $\rho_2^{(\Pi)}(x, y) \le \rho_1^{(\Pi)}(x) \rho_1^{(\Pi)}(y)$ for $\mu_E \otimes \mu_E$-almost every $(x, y) \in E\times E$.

\begin{lemma}\label{lem-var-LS}
A negatively correlated point process satisfies Assumption \ref{ass-pp}.
\end{lemma}

{\flushleft \it Proof.}  Assume that $\Pi$ is negatively correlated and let $f: E\rightarrow \C$ be such that $\E_\Pi [\sum_{x\in \X} |f(x) | + |f(x)|^2]<\infty$. Assume first that $f$ is non-negative. Then
\begin{multline*}
 \Var_\Pi \Big(\sum_{x\in \X} f(x)\Big)  =  \int_{E} | f(x)|^2 \rho_1^{(\Pi)}(x) d\mu(x)    + 
\\
  + \int_{E \times E} f(x) f(y) \Big[ \rho_2^{(\Pi)}(x, y) - \rho_1^{(\Pi)}(x) \rho_1^{(\Pi)}(y)\Big] \mu(dx) \mu(dy) \le 
\\
 \le   \int_{E} | f(x)|^2 \rho_1^{(\Pi)}(x) \mu(dx) = \E_
\Pi \left( \sum_{x\in \X} | f(x)|^2 \right).
\end{multline*}
For general complex-valued function, by writing $f = g_1 - g_2 +  \sqrt{-1} (g_3 - g_4)$ with $g_1, \cdots, g_4$  non-negative such that $|f|^2  = |g_1 |^2 + | g_2|^2 + |g_3|^2 + | g_4|^2$, we obtain  
\[
\Var_\Pi \Big(\sum_{x\in \X} f(x)\Big) \le 16  \E_\Pi \Big( \sum_{x\in \X} | f(x)|^2\Big). \qed
\]

\section{Reconstruction of Hilbert-space vector-valued harmonic functions}\label{sec-gen}
Let  $(M, d, \mu_M)$ be a proper  complete metric space equipped with a positive Radon measure. Let $o \in M$ be a distinguished point.   Note that we fix a point $o \in M$ only for convenience,  all our results will be independent of the choice of the fixed point.

 Recall  that   a non-negative function $\Lambda$ on $\R^{+}$ (or on $\N$) is called  sub-exponential if $\lim_{t\to + \infty} \Lambda(t) e^{-\alpha t} = 0$  for any $\alpha >0$. A pair $(L, U)$, $L \le U$, of continuous  non-negative functions on $[0, \infty)$ is called {\it controllable} if
$\liminf_{r\to \infty} L(r)> 0$  and the functions $L, U, \sup_{k \in \N} \frac{U(kr)}{L((k+1)r)}$ are all sub-exponential.
For example, if $0 < c< C< \infty$ and $\alpha \ge 0, \beta \ge 0$ and   $0 < \gamma < 1$, then the  pair  $(L, U)$ defined by $L(r) = c r^\alpha  \exp (\beta r^\gamma),  U(r) = C r^\alpha \exp (\beta r^\gamma)$, is  controllable.

In this section, we will always assume that the Radon measure $\mu_M$ satisfies the following 
\begin{assumption}\label{ass-M-A1}
There exists a constant $h_M >0$ and a controllable pair $(L, U)$  such that
\begin{align}\label{ass-rate}
L(r) e^{r \cdot h_M } \le \mu_M(B(o, r)) \le   U(r) e^{r\cdot h_M} \quad \text{for all $r >0$}.
\end{align}
\end{assumption}

The constant $h_M> 0$ in the above assumption will be called the volume entropy for the triple $(M, d, \mu_M)$.

\subsection{Hilbert-space vector-valued harmonic functions}\label{sec-general-Hilbert}

Let  $\mathscr{H}$  be a Hilbert space over $\R$ or $\C$ and let $L^2_{\mathrm{loc}}(\mu_M; \mathscr{H})$ denote the set of  locally square-integrable functions $f: M \rightarrow \mathscr{H}$.   We say that a function $f\in L^2_{\mathrm{loc}}(\mu_M; \mathscr{H})$ satisfies  mean value property (or is harmonic), if for any $z \in M$ and any $r> 0$, 
\begin{align}\label{mvp-f}
f(z)  = \frac{1}{\mu_M(B(z, r))}\int_{B(z, r)} f(x) d\mu_M(x). 
\end{align}
 Let $\widetilde{\mathrm{MVP}}(\mu_M; \mathscr{H}) \subset L^2_{\mathrm{loc}}(\mu_M; \mathscr{H})$ denote the subset of all functions satisfying the above mean value property.

\begin{definition}\label{def-MVP}
Let  $\mathrm{MVP}(\mu_M; \mathscr{H})\subset   \widetilde{\mathrm{MVP}}(\mu_M; \mathscr{H})$ denote the subset of all functions $f \in \widetilde{\mathrm{MVP}}(\mu_M; \mathscr{H})$ satisfying the following properties:
\begin{itemize}
\item  There exists a non-decreasing sub-exponential function $\Lambda: \N \rightarrow \R^{+}$, depending on the function  $f$,  such that for any $k\in \N$, we have 
\begin{align}\label{sub-exp-correction}
\int_{A_k(o)}\| f(x) \|_{\mathscr{H}}^2 d\mu_M(x) \le \Lambda ( k ) e^{2 k h_M},
\end{align}
where $A_k(o) =\{x\in M: k\le d(x, o) < k + 1\}$. 
\item The following limit equality holds:
\begin{align}\label{global-growth-f}
 \lim_{s \to h_M^{+}} \frac{\displaystyle  \int_M  e^{-2s d(x, o)}  \| f(x)\|_{\mathscr{H}}^2   d\mu_M(x) }{\displaystyle \Big[  \int_M e^{-s d(x,o)} d\mu_M(x) \Big]^2} = 0. 
\end{align}
\end{itemize}
\end{definition}

\begin{theorem}\label{thm-H-L}
Let  $\mu_M$ be a positive Radon measure on $M$ satisfying Assumption \ref{ass-M-A1}.  Let $\Pi$  be a point process on $M$ satisfying Assumption \ref{ass-pp} and  with first intensity measure $\mu_M$. Let $\mathscr{H}$ be a Hilbert space and  $f \in \mathrm{MVP}(\mu_M; \mathscr{H})$ a fixed function.  Fix $z\in M$ and  any sequence $(s_n)_{n\ge 1}$ in $(h_M, \infty)$ converging to $h_M$ and satisfying 
\begin{align}\label{fast-to-cri}
 \sum_{n=1}^\infty   \frac{\displaystyle \int_M  e^{-2s_n d(x, o)}   \| f(x)\|_{\mathscr{H}}^2   d\mu_M(x)}{\displaystyle  \Big[\int_M e^{-s_n d(x, o)} d\mu_M(x) \Big]^2 }  <\infty.
\end{align}
 Then $\Pi$-almost every $X$ satisfies:
\begin{enumerate}
\item For any $n\in \N$, we have 
\begin{align}\label{thm-H-L-small-goal}
 \sum_{k  =0}^{ \infty } \Big\| \sum_{x\in  X \atop k \le d(x, z) < k + 1 } e^{-s_n d(x, z)} f(x)\Big\|_{\mathscr{H}} < \infty.
\end{align}
\item  The following limit equality holds in the norm topology of $\mathscr{H}$: 
\begin{align}\label{thm-H-L-goal}
 f(z) = \lim_{n\to\infty}  \frac{\displaystyle \sum_{k  =0}^{ \infty }  \sum_{x\in  X \atop k \le d(x, z) < k+1 } e^{-s_n d(x, z)} f(x) }{\displaystyle{  \sum_{x \in X}  e^{-s_n d(x, z)} }}.
\end{align}
\end{enumerate}
\end{theorem}

\subsubsection{The space $\mathrm{MVP}(\mu_M, \mathscr{H})$}\label{sec-MVP}
We give examples of functions in $\mathrm{MVP}(\mu_M, \mathscr{H})$, which will be useful  later.

\begin{lemma}\label{lem-c-f}
If $\mathscr{H} = \C$ is the one-dimensional Hilbert space, then the constant function belongs to $\mathrm{MVP}(\mu_M; \C)$. 
\end{lemma}
\begin{proof}
It is clear that the constant function satisfies the mean value property \eqref{mvp-f} and the growth condition \eqref{sub-exp-correction}. Write 
\begin{align}\label{int-rep-exp}
e^{-s t} = \int_{\R^{+}}   s e^{-s r}  \cdot \mathds{1}(t < r) dr,
\end{align}
whence for any $s> h_M$, 
\begin{multline*}
\int_M e^{-2 s d(x, o)} d\mu_M(x) = \int_M   \int_{\R^{+}} 2s e^{- 2s  r} \mathds{1}( d(x, o) < r) dr d\mu_M(x) 
\\
  = 2 s \int_{\R^{+}} e^{-2s r}  \mu_M(B(o, r)) dr \le 2s \int_{\R^{+}} e^{-2 s r} U(r) e^{h_M r} dr 
\\
 \le 2s \sup_{r \in \R^{+}} \Big ( U(r) e^{-h_M r} \Big)  \int_{\R^{+}} e^{- 2 (s - h_M)r} dr = \frac{s \sup_{r \in \R^{+}} \Big ( U(r) e^{-h_M r} \Big) }{s-h_M}.
\end{multline*}
While on the other hand, for any $s> h_M$, 
\begin{multline*}
\int_M e^{-s d(x, o)} d\mu_M(x)  =  \int_M   \int_{\R^{+}} s e^{- s  r} \mathds{1}( d(x, o) < r) dr d\mu_M(x) 
\\
 = s \int_{\R^{+}} e^{-s r} \mu_M(B(o, r)) dr  \ge  s \int_{\R^{+}} e^{-s r} L(r) e^{h_M r} dr = \frac{s}{s-h_M} \int_{\R^{+}}  e^{-t} L(t/(s-h_M)) dt.
\end{multline*}
Using the assumption $\liminf_{t\to \infty} L(t) > 0$, we have
\[
\liminf_{s\to h_M^{+}} \int_{\R^{+}} e^{-t} L(t/(s-h_M)) dt \ge   \int_{\R^{+}} \liminf_{s\to h_M^{+}}  e^{-t} L(t/(s-h_M)) dt  = \liminf_{t\to \infty} L(t) > 0.
\]
Thus the constant function  satisfies the condition \eqref{global-growth-f} and the proof is complete. 
\end{proof}

In the situation of negatively curvatured Riemannian manifolds, the rowth rate of the volume of the geodesic balls could be of the form $r^\beta e^{h r}$ with $\beta \ge 0$ and $h>0$ the volume entropy. In such situation,  we give a useful sufficient condition for a function $f \in \widetilde{\mathrm{MVP}} (\mu_M; \mathscr{H})$ to be in the class $\mathrm{MVP}(\mu_M; \mathscr{H})$.

\begin{proposition}[Mean-growth condition]\label{prop-mgc}
Assume that there exist constants $\beta \ge 0, h_M > 0$,  $C>1$ and $r_0\in \N$ such that 
\[
C^{-1} r^\beta  e^{h_M r} \le \mu_M( B(o, r)) \le C r^\beta e^{h_M r}, \quad \text{for all $r\ge r_0$.}
\]
If a function $f \in \widetilde{\mathrm{MVP}}(\mu_M; \mathscr{H})$  satisfies the following mean-growth condition: there exists a non-increasing function $\Theta: \R^{+}\rightarrow \R^{+}$ with $\lim_{t \to \infty} \Theta(t)  =0$ 
such that for any $k\in \N$, 
\begin{align}\label{mean-cond-in-MVP}
\int_{A_k(o)}\| f(x)\|_{\mathscr{H}}^2 d\mu_M(x)  \le   \Theta(k )  \cdot (k+1)^{1 + 2 \beta}   \cdot e^{2 k h_M}, 
\end{align}
then $f \in \mathrm{MVP}(\mu_M; \mathscr{H})$. 
\end{proposition}

\begin{proof}
It suffices to show that under the assumption of the proposition, the inequality \eqref{mean-cond-in-MVP} implies the limit equality \eqref{global-growth-f}.  

From the proof of Lemma \ref{lem-c-f}, for any $s \in ( h_M, h_M + 1)$, we have 
\begin{multline}\label{D-below}
\int_M e^{-s d(o, x)} d\mu_M(x)  = s \int_{\R^{+}} e^{-s r} \mu_M(B(o, r)) dr \ge  s C^{-1} \int_{r_0}^\infty e^{-s r}  r^{\beta} e^{h_M r} dr = 
\\
 = \frac{ s C^{-1}}{(s - h_M)^{1 + \beta}} \int_{ (s - h_M) r_0}^\infty e^{-t}  t^{\beta} dt  \ge \frac{C'}{ (s - h_M)^{1 + \beta}},
\end{multline}
where $C'> 0$ is a constant.  

  For any $ s \in (h_M,   h_M  + 1)$, by using \eqref{mean-cond-in-MVP},  we have  
\begin{multline*}
 \int_M  e^{-2s d(o, x)}  \| f(x)\|_{\mathscr{H}}^2   d\mu_M(x)  
\le \sum_{k = 0}^\infty e^{-2s k}  \int_{ A_k(o)} \| f(x)\|_{\mathscr{H}}^2   d\mu_M(x)    \le 
\\
\le  \sum_{k = 0}^\infty e^{-2(s - h_M) k} \Theta(k) (k+1)^{1 +  2\beta}.
\end{multline*}
Using $s \in (h_M,   h_M  + 1)$, we obtain that for any $k \in \N \cup \{0\}$, 
\[
1 \ge \frac{e^{-2(s-h_M) (k+1)}}{e^{-2(s-h_M)k}} \ge  e^{-2} \an  2  \ge  \frac{k+2}{k+1} \ge 1.
\]
Therefore, since $\Theta$ is non-increasing,  by setting $\Theta|_{[-1, 0]} = \Theta(0)$,   there exists a constant $C''> 0$ such that 
\begin{align*}
 \int_M  e^{-2s d(o, x)}  \| f(x)\|_{\mathscr{H}}^2   d\mu_M(x) \le  C'' \int_{0}^\infty  \underbrace{e^{-2 (s - h_M) t} \Theta(t -1) (1 + t)^{1 + 2 \beta} }_{\text{denoted by $F(t)$}} dt. 
\end{align*}
It is clear that there exist constants $C''', C'''' > 0$, such that 
\begin{multline}\label{N-up}
\int_{0}^\infty F(t) dt \le \int_{0}^2 F(t) dt  + \int_{2}^\infty F(t) dt  
\le  C''' + C'''' \int_0^\infty e^{- 2 (s-h_M) t} \Theta(t) t^{1+ 2 \beta} dt = 
\\
  =  C'''  + \frac{C''''}{(s-h_M)^{2 + 2 \beta}}\int_0^\infty   e^{-2 t}  \Theta( t/(s-h_M)) t^{1 + 2 \beta} dt.
\end{multline}
By Dominated Convergence Theorem and the assumption $\lim_{t \to \infty } \Theta(t) = 0$, we have 
\[
\lim_{s \to h_M^{+}} \int_0^\infty e^{-2 t} \Theta( t/(s-h_M)) t^{1 + 2 \beta} dt = 0. 
\]
Combining \eqref{D-below} and \eqref{N-up}, we obtain the desired  limit equality \eqref{global-growth-f}.   
\end{proof}

The following immediate corollary of Proposition \ref{prop-mgc} will be used later. 
\begin{corollary}[Pointwise-growth condition]\label{cor-in-MVP}
Under the assumption of Proposition \ref{prop-mgc}, if a function $f \in \widetilde{\mathrm{MVP}}(\mu_M; \mathscr{H})$  satisfies the following pointwise growth condition: there exists a non-increasing  function $\Theta: \R^{+}\rightarrow \R^{+}$ with $\lim_{t \to \infty} \Theta(t)  =0$ 
such that 
\[
\| f(x)\|_{\mathscr{H}}^2 \le   \Theta( d(o, x))  \cdot d(o, x)^{1 + \beta}   \cdot e^{h_M d(o, x)}, 
\]
then $f \in \mathrm{MVP}(\mu_M; \mathscr{H})$. 
\end{corollary}

\subsubsection{Proof of Theorem \ref{thm-H-L}}
Recall that  $\mu_M$ is a positive Radon measure on $M$ satisfying  Assumption \ref{ass-M-A1}.  Fix a point process $\Pi$ on $M$, a function  $f \in \mathrm{MVP}(\mu_M; \mathscr{H})$ and  a sequence $(s_n)_{n\ge 1}$ satisfying all assumptions of Theorem \ref{thm-H-L}. 

 Introduce some notation as follows: for any configuration $X\in \Conf(M)$, any $z\in M$ any $s>0$ and any non-negative integer $k \ge 0$,  set 
\begin{align}\label{notation-T-sigma}
\begin{split}
T_k^f(z, s; X):  = \sum_{x \in X  \atop k \le d(x, z) < k+1} &  e^{-s d(x, z)} f(x) \in \mathscr{H},
\\
 t_k(z, s; X) = \sum_{x \in X  \atop k \le d(x, z) < k+1}  &  e^{-s d(x, z)}  \in [0, \infty),
\\
\sigma(z, s; X):  = \sum_{x\in X} e^{-s d(x, z)} \in [0, \infty], & \quad \overline{\sigma}(z, s):  =  \E_\Pi (\sigma(z, s; \X)) \in [0, \infty],
\\
 R_f(z,s; X) : = \frac{\sum_{k=0}^\infty  T_k^f (z, s; X)}{\sigma(z, s; X)}\in \mathscr{H}, \quad &  \underline{R}_f(z, s; X) : = \frac{\sum_{k=0}^\infty T_k^f(z, s; X)}{\overline{\sigma}(z, s)}\in \mathscr{H}.
\end{split}
\end{align}
 Here we used  Convention \ref{convention-notation} in the Appendix for $\E_\Pi (\sigma(z, s; \X))$.

For a Banach space $\mathfrak{B}$, let $L^2(\Conf(M), \Pi; \mathfrak{B})=L^2(\Pi; \mathfrak{B})$ be the Banach space of $\mathfrak{B}$-valued  functions defined on $\Conf(M)$ and square-integrable with respect to $\Pi$.

\begin{proposition}\label{prop-L2-L1-sum}
For any $s> h_M$ and any $z\in M$, we have
\begin{align}\label{lem-L2-L1-sum-goal-1}
\sum_{k=0}^\infty    \| T_k^f(z, s; \X)\|_{L^2(\Pi; \mathscr{H})} <\infty.
\end{align}
In particular, we have 
\begin{align}\label{lem-L2-L1-sum-goal-2}
\sum_{k=0}^\infty    \E_\Pi  \Big( \| T_k^f(z, s; \X)\|_{ \mathscr{H}} \Big) <\infty.
\end{align}
\end{proposition}

\begin{lemma}\label{lem-ball-av-M}
For any $s> 0, R>0$ and any $z\in M$, we have 
\[
\E_\Pi \left( \sum_{x\in \X \cap B(z, R)}  e^{-s d(z, x)}  f(x) \right) =f(z) \cdot \E_\Pi \left( \sum_{x\in \X\cap B(z, R)}  e^{-s d(z, x)} \right).
\]
\end{lemma}
\begin{proof}
Using the identity 
\[
e^{-s t} \mathds{1} ( t < R) = \int_0^R s e^{-s r} \mathds{1}(t< r) dr + \int_R^\infty s e^{-s r} \mathds{1}(t < R) dr
\]
and the equality \eqref{mvp-f},  we obtain 
\begin{multline*}
\E_\Pi \left( \sum_{x\in \X \cap B(z,R)}  e^{-s d(z, x)}  f(x) \right)  = \int_{M} e^{-s d(z, x)} \mathds{1}( d(z, x) < R)  f(x) d\mu_M(x)  
\\
 = \int_M   f(x) d\mu_M(x) \int_0^R s e^{-s r} \mathds{1}( d(z, x) <r)  dr    + \int_M  f(x) d\mu_M(x) \int_R^\infty s e^{-s r} \mathds{1} (d(z, x) < R)  dr
\\
 = \int_0^R s e^{-s r} dr \int_{B(z, r)}   f(x)  d\mu_M(x)    +  \int_R^\infty s e^{-s r}  dr \int_{B(z, R)} f(x) d\mu_M(x)
\\
 = f(z) \cdot \left(  \int_0^R s e^{-s r} \mu_M(B(z, r))  dr +\int_R^\infty s e^{-sr}    \mu_M( B(z, R)) dr\right).
\end{multline*}
Similarly, the above compution applies to the constant function yields
$$
\E_\Pi \left( \sum_{x\in \X \cap B(z,R)}  e^{-s d(z, x)}  \right)  = \int_0^R s e^{-s r} \mu_M(B(z, r))  dr +\int_R^\infty s e^{-sr}    \mu_M( B(z, R))dr.
$$
Combining the above two equalities, we complete the proof of the lemma. 
\end{proof}
Lemma \ref{lem-ball-av-M} immediately implies
\begin{corollary}\label{cor-MVP}
For any $s> 0, z \in M$ and $k\in \N$ , we have 
\[
\E_\Pi \left(   T_k^f( z, s; \X)\right) =  f(z) \cdot \E_\Pi \left( t_k(z, s; \X) \right).
\]
\end{corollary}

\begin{lemma}\label{lem-lower-bdd}
For any $s> h_M$ and any $z\in M$,  we have $\overline{\sigma}(z, s) <\infty$. 
\end{lemma}

\begin{proof}
Write
\begin{multline*}
 \overline{\sigma}(z, s) = \E_{\Pi}\Big[\sum_{x \in \X}  e^{-s d(z, x)}\Big]   \le   e^{s d(o, z)}  \E_{\Pi}\Big[\sum_{x \in \X}  e^{-s d(o, x)}\Big] \le e^{s d(o, z)}   \int_M e^{-s d(o, x)} d\mu_M(x) 
\\
\le  e^{s d(o, z)}  \sum_{k = 0}^\infty e^{-s k} \mu_M( B(o, k+1)) \le e^{s d(o, z)} \sum_{k=0}^\infty e^{-s k } e^{h_M (k+1)} U(k+1),
\end{multline*}
since $U$ is sub-exponential, we have 
$
\sum_{k=0}^\infty e^{- (s - h_M) k} U(k+1)<\infty.
$
\end{proof}

{\flushleft \it Proof of Proposition \ref{prop-L2-L1-sum}. }
Fix $s> h_M, z \in M$, set $A_k(z) : = \{x\in M| k \le d(x, z) < k+1  \}$. Let $N = N_z \in \N$ be the smallest integer such that $N\ge d(z, o)$, then for any $k \ge N$, 
\begin{align}\label{cover-balls}
A_k(z) \subset \bigcup_{\ell =0}^{2N}  A_{k- N + \ell} (o). 
\end{align}
We have 
\begin{align*}
  \|  T_k^f(z, s; \X)  \|_{L^2(\Pi;  \mathscr{H})}^2     =      \Var_\Pi \left(  T_k^f(s, z; \X) \right)  + \left\|  \E_\Pi \left(  T_k^f(s, z; \X)\right) \right\|_{\mathscr{H}}^2.
\end{align*}
Proposition \ref{rem-scalar-vector}, the inequalities \eqref{sub-exp-correction} and \eqref{cover-balls}  imply that, under Assumption \ref{ass-pp}, for any positive integer $k \ge N$, we have
\begin{multline*}
 \Var_\Pi \left(  T_k^f( z, s; \X)\right)  
\le  C \int_{A_k(z)}  e^{-2s d(x, z)} \| f(x)\|_{\mathscr{H}}^2 d\mu_M(x) \le 
\\
\le C e^{-2s (k+1)} \int_{A_k(z)} \| f(x)\|_{\mathscr{H}}^2 d\mu_M(x) \le C e^{-2s (k+1)} \sum_{\ell =0}^{2N} \int_{A_{k-N+\ell}(o)} \|f (x)\|_{\mathscr{H}}^2 d\mu_M(x)\le 
\\
\le C e^{-2s (k+1)} \sum_{\ell =0}^{2N}    \Lambda ( k - N  + \ell ) e^{2 (k- N + \ell) h_M} \le (2N +1) C e^{2 N h_M - 2s }  \Lambda (k + N) e^{- 2 ( s- h_M) k }.
\end{multline*}
Consequently,  there exists $C_z>0$, such that 
\[
 \Var_\Pi \left(   T_k^f(z, s; \X) \right)\le C_z e^{- (2s - 2h_M) k}   \Lambda(k+N).
\]
Therefore, by applying Corollary \ref{cor-MVP}, we obtain
 \begin{multline*}
\|   T_k^f(z, s; \X)\|_{L^2(\Pi; \mathscr{H})} \le  \sqrt{C_z} e^{- (s - h_M)k } \sqrt{ \Lambda(k+N)} +   \| f(z)\|_{\mathscr{H}}   \cdot  \E_\Pi \left(  t_k(z, s; \X) \right).
\end{multline*}
Since $\Lambda$ is sub-exponential and $s> h_M$, we have the convergence
\[
\sum_{k=0}^\infty e^{- (s - h_M)k } \sqrt{ \Lambda(k+N )} <\infty,
\]
which,  combined  with the inequality $\overline{\sigma} (z, s)<\infty$ for any $s>h_M$ proved in Lemma \ref{lem-lower-bdd}, implies the desired convergence \eqref{lem-L2-L1-sum-goal-1}. 

The convergence \eqref{lem-L2-L1-sum-goal-2} follows from  \eqref{lem-L2-L1-sum-goal-1} by observing the elementary inequality
\[
 \E_\Pi  \Big( \| T_k^f(z, s; \X)\|_{ \mathscr{H}} \Big) \le \| T_k^f(z, s; \X)\|_{L^2(\Pi; \mathscr{H})}. \qed
\]

\begin{proposition}\label{prop-var-up-bd-rep}
For any $s > h_M$ and any $z\in M$, we have
\begin{equation}\label{prop-var-up-bd-rep-goal}
\Var_\Pi  \left (  \sum_{k=0}^\infty T_k^f ( z, s; \X) \right) \le  C   \int_M e^{-2s d(x, z)} \| f(x) \|_{\mathscr{H}}^2 d\mu_M(x),
\end{equation}
where the constant $C>0$ is the same constant as in  the inequality \eqref{var-less-2-moment}. 
\end{proposition}
\begin{proof}
 Corollary \ref{cor-MVP} and the convergence \eqref{lem-L2-L1-sum-goal-2} imply,
for any $s> h_M$,  the equality 
\begin{equation}\label{av-pv-f}
\E_\Pi \left(  \sum_{k=0}^\infty T_k^f (z, s; \X)   \right) =  f(z) \cdot \E_\Pi \left(  \sum_{x\in \X}  e^{-s d(x, z)} \right)  = f(z) \cdot \int_M e^{-s d(x,z)} d\mu_M(x).
\end{equation} 
For any positive integer $N$, by Proposition \ref{rem-scalar-vector} and  Lemma \ref{lem-ball-av-M},  we have 
\begin{multline*}
  \left \| \sum_{k = 0}^{N-1}  T_k^f (z, s; \X)   \right\|_{L^2(\Pi; \mathscr{H})}^2  =   \left \| \sum_{x\in \X \cap B(z, N)}  e^{-s d(x, z)} f(x) \right\|_{  L^2 (\Pi;  \mathscr{H})}^2 =
\\
 =   \Var_\Pi \left(    \sum_{x\in \X \cap B(z, N)}  e^{-s d(x, z)} f(x) \right)  +   \left \|     \E_\Pi \left( \sum_{x\in \X \cap B(z, N)}  e^{-s d(x, z)} f(x)  \right) \right\|_{\mathscr{H}}^2\le
\\
\le  C \int_{B(z, N)} e^{-2s d(x, z)} \| f(x) \|_{ \mathscr{H}}^2 d\mu_M(x) + \| f(z) \|_{\mathscr{H}}^2  \left( \int_{B(z, N)}  e^{-s d(x,z)} d\mu_M(x) \right)^2\le
\\
\le C   \int_M e^{-2s d(x, z)} \| f(x) \|_{\mathscr{H}}^2 d\mu_M(x) + \| f(z) \|_{\mathscr{H}}^2  \left( \int_M  e^{-s d(x,z)} d\mu_M(x) \right)^2.
\end{multline*}
Therefore, by applying \eqref{lem-L2-L1-sum-goal-1} and  \eqref{av-pv-f},   we obtain  the desired inequality \eqref{prop-var-up-bd-rep-goal}. 
\end{proof}

The following remark will be used in the proof of Theorem \ref{thm-H-L}. 

\begin{remark}\label{rem-f-1}
Assume that $\mu_M( \{x:  f(x) \ne 0\}) > 0$. Then 
\[
\inf_{s \in (h_M, 2h_M)}\int_M  e^{-2s d(x, o)}   \| f(x)\|_{\mathscr{H}}^2   d\mu_M(x)  > 0. 
\]
On the other hand, it is clear that 
\[
\sup_{s \in (h_M, \infty)}\int_M  e^{-2s d(x, o)}    d\mu_M(x)   = \int_M  e^{-2 h_M d(x, o)}    d\mu_M(x)  <\infty. 
\]
Therefore, there exists $C>0$, such that  for any $s \in (h_M, 2 h_M)$, we have 
\[
\int_M  e^{-2s d(x, o)}    d\mu_M(x)  \le C \int_M  e^{-2s d(x, o)}  \|f (x)\|_{\mathscr{H}}^2  d\mu_M(x). 
\]
In particular, the convergence \eqref{fast-to-cri}  for a non-identically zero function $f\in  \mathrm{MVP}(\mu_M; \mathscr{H})$ implies the same convergence \eqref{fast-to-cri} for the constant $f\equiv 1$.  
\end{remark}

\begin{proof}[Proof of Theorem \ref{thm-H-L}]
Assume that $f \in \mathrm{MVP}(\mu_M; \mathscr{H})$. We may assume that $\mu_M(\{x: f(x) \ne 0\}) >0$.  Fix $z\in M$.  The convergence \eqref{lem-L2-L1-sum-goal-2} implies that \eqref{thm-H-L-small-goal} holds for $\Pi$-almost every $X$ and all $n$.  Recall the notation \eqref{notation-T-sigma}. Fix any $s> h_M$. The equality \eqref{av-pv-f} implies 
\[
\E_\Pi \Big[   \underline{R}_f(z, s; \X) \Big] = f(z). 
\]
Whence by Proposition \ref{prop-var-up-bd-rep}, we have
\begin{equation*}
\E_\Pi \Big[ \left\|   \underline{R}_f(z, s; \X)  - f(z)  \right\|_{\mathscr{H}}^2 \Big] \le  \frac{  \displaystyle C \int_M  e^{-2s d(x, o)} \| f(x)\|_{\mathscr{H}}^2  d\mu_M(x)}{  \displaystyle \Big[\int_M e^{-s d(x, o)} d\mu_M(x) \Big]^2 } . 
\end{equation*}
Now by using the assumption \eqref{fast-to-cri} on the sequence $(s_n)_{n\ge 1}$, we obtain 
\[
\sum_{n=1}^\infty \E_\Pi \Big[ \left\|   \underline{R}_f(z, s_n; \X)  - f(z)  \right\|_{\mathscr{H}}^2 \Big] < \infty. 
\]
It follows that for $\Pi$-almost every $X$, we have the limit equality 
\begin{align}\label{av-ratio-limit}
\lim_{n\to\infty}  \left\|   \underline{R}_f(z, s_n; \X)  - f(z)  \right\|_{\mathscr{H}} = 0. 
\end{align}
By Remark \ref{rem-f-1}, taking $\mathscr{H} = \C$ and $f \equiv 1$ the constant function  in the limit equality \eqref{av-ratio-limit}, we obtain that for $\Pi$-almost every $X$, 
\begin{align}\label{av-or-not-1}
\lim_{n\to\infty} \frac{\sigma(z, s_n; X)}{\overline{\sigma}(z, s_n)} = 1. 
\end{align}
Combining \eqref{av-ratio-limit}, \eqref{av-or-not-1} with the elementary equality 
\begin{align}\label{R-f-av-or-not}
 R_f(z, s_n; X)  =   \frac{\overline{\sigma}(z, s_n)} {\sigma(z, s_n; X)} \cdot  \underline{R}_f(z, s_n; X), 
\end{align}
we obtain for $\Pi$-almost every $X$ that $\lim_{n\to\infty}  \left\|   R_f(z, s_n; \X)  - f(z)  \right\|_{\mathscr{H}} = 0$. The proof of the theorem is complete. 
\end{proof}

The following proposition will be used later. 
\begin{proposition}\label{prop-const-fn}
Under the asumptions of Theorem \ref{thm-H-L}, for $\Pi$-almost every $X$,  the limit equality  \eqref{av-or-not-1} holds for all $z\in M$.   
\end{proposition}
\begin{proof}
In the proof of Theorem \ref{thm-H-L}, we see that under the asumptions of Theorem \ref{thm-H-L}, for any fixed $z\in M$,  the limit equality \eqref{av-or-not-1} holds for $\Pi$-almost every $X$. Now fix a countable dense subset $\mathscr{A} \subset M$. then there exists a subset $\Omega\subset \Conf(M)$ with $\Pi(\Omega)  = 1$ such that the limit equality   \eqref{av-or-not-1} holds for all $X \in \Omega$ and all $z\in \mathscr{A}$. 

Now  we show that for any $X\in \Omega$ and  any $z' \in M$, the limit equality  \eqref{av-or-not-1} holds. Let $(z_k)_{k\ge 1}$ be a sequence in $\mathscr{A}$ converging to $z'$. Using the elementary inequalities
\[
e^{-s_n d(z_k, z')} \sigma (z_k, s_n; X) \le \sigma(z', s_n; X) \le e^{s_n d(z_k, z')} \sigma (z_k, s_n; X)
\]
and 
\[
e^{-s_n d(z_k, z')}\overline{ \sigma} (z_k, s_n) \le \overline{\sigma}(z', s_n) \le e^{s_n d(z_k, z')} \overline{\sigma} (z_k, s_n), 
\]
we obtain 
\[
e^{-2 s_n d(z_k, z')} \frac{\sigma(z_k, s_n; X)}{\overline{\sigma}(z_k, s_n)} \le  \frac{\sigma(z', s_n; X)}{\overline{\sigma}(z', s_n)} \le e^{2 s_n d(z_k, z')} \frac{\sigma(z_k, s_n; X)}{\overline{\sigma}(z_k, s_n)}.
\]
Therefore, for any $k\in \N$, we have 
\[
    e^{- 2 h_M d(z_k, z')} \le \liminf_{n\to\infty}   \frac{\sigma(z', s_n; X)}{\overline{\sigma}(z', s_n)} \le \limsup_{n\to\infty}   \frac{\sigma(z', s_n; X)}{\overline{\sigma}(z', s_n)} \le e^{2 h_M d(z_k, z')}. 
\]
Since $k\in \N$ is arbitrary in the above inequality, we obtain 
\[
1 \le  \liminf_{n\to\infty}   \frac{\sigma(z', s_n; X)}{\overline{\sigma}(z', s_n)} \le \limsup_{n\to\infty}   \frac{\sigma(z', s_n; X)}{\overline{\sigma}(z', s_n)} \le 1. 
\]
The proof of the proposition is complete.  
\end{proof}

\subsection{Hilbert-space vector-valued non-negative harmonic functions}\label{sec-hil-pos}
In \S \ref{sec-general-Hilbert}, the Hilbert space  $\mathscr{H}$ is a general abstract Hilbert space over $\R$. In this section, by considering Hilbert spaces with  certain positive cone structure and stronger assumptions on the asymptotic of the volume of balls and of the function $f$, we can improve our results.

\subsubsection{Hilbert spaces with positive cone structure}\label{sec-pos-cone}
A pair $(\mathscr{B}, \ge)$  of a Banach space $\mathscr{B}$ over $\R$, equipped with a partial order $\ge$ is called a partially
ordered vector space, if the relation  $u_1 \ge u_2$ implies  $\alpha u_1 \ge  \alpha u_2$ for all $\alpha \in \R^{+}$ and $u_1 + v \ge  u_2 + v$ for
all $v \in \mathscr{H}$.     
The subset $\mathscr{B}_{+} = \{u \in \mathscr{B}: u \ge 0\}$ is called the positive cone of $(\mathscr{B}, \ge)$ and its elements are referred to as positive vectors. Note that if $u, v \in \mathscr{B}_{+}$ are such that $u \ge v$, then for any $\alpha \ge \beta \ge 0$, we have $\alpha u \ge \beta v \ge 0$.

\begin{definition}\label{def-com}
We say the partial order $\ge$ of a partially ordered Banach space $(\mathscr{B}, \ge)$  is compatible with the norm $\|\cdot\|_\mathscr{B}$ if there exists a constant $C>0$ such that for any triple  $(u_1, u_2, u_3)$ of elements in $\mathscr{B}$, if $ u_1 \le u_2 \le u_3$, then $\| u_2\|_{\mathscr{B}} \le C \max( \| u_1 \|_{\mathscr{B}}, \| u_3\|_{\mathscr{B}})$.  
\end{definition}\label{def-order-norm}
For example, for any $1 \le p \le \infty$, on the real Banach space $L^p(\Omega, \nu; \R)$, where $(\Omega, \nu)$ is a measure space,  the natural partial order, coming from pointwise inequality of functions, is compatible with the $L^p$-norm with a constant $C = 1$.

 Let $(\mathscr{H}, \ge)$ be a Hilbert space over $\R$ equipped with a partial order which is compatible with the norm.  Let $C(M; \mathscr{H}_{+})$ denote the space of norm continuous functions $f: M \rightarrow \mathscr{H}_{+}$. Recall the definition of $\mathrm{MVP}(\mu_M; \mathscr{H})$ in \S \ref{sec-general-Hilbert}. 

\begin{definition}\label{def-cmvp}
 Set  
\[
\mathrm{CMVP}(\mu_M; \mathscr{H}_{+}) :=  \mathrm{MVP}(\mu_M; \mathscr{H}) \cap C(M; \mathscr{H}_{+}).  
\]
\end{definition}
Note that if $\lim_{z' \to z} \mu_M( B(z, r) \Delta B(z', r)) = 0$ holds  for any $r > 0, z \in M$, where $\Delta$ denote the symmetric difference of two sets, then $\mathrm{MVP}(\mu_M; \mathscr{H}) \subset C(M; \mathscr{H})$.

Let $\mathfrak{B}$ be a Banach space. Recall that a convergent series  $\sum_{n=1}^\infty x_n$ in $\mathfrak{B}$  is said to {\it converge 
absolutely} if $\sum_{n=1}^\infty \| x_n\|_{\mathfrak{B}}<\infty$ and to {\it converge unconditionally} if its sum does not change under any reordering of the terms.
For example, under some additional assumptions, see Lemma \ref{prop-uc-cv} below, a convergent functional series of positive  functions converges unconditionally.

The main result of this section is the following
\begin{theorem}\label{thm-general-pos}
Let $\Pi$  be a point process on $M$ satisfying Assumption \ref{ass-pp} and  with first intensity measure $\mu_M$. Fix a function   $f \in \mathrm{CMVP}(\mu_M; \mathscr{H}_{+})$ and a sequence $(s_n)_{n\ge 1}$ in $(h_M, \infty)$ satisfies \eqref{fast-to-cri}.  Then $\Pi$-almost every $X$ satisfies: 
\begin{itemize}
\item The series 
\begin{align}\label{uncon-cv-s}
  \sum_{x\in  X} e^{-s d(x, z)} f(x)  
\end{align}
converges unconditionally in $\mathscr{H}$ for all $s>h_M$ and all $z\in M$. 
\item The limit equality 
\begin{align}\label{ratio-uc-cv}
f(z) = \lim_{n\to\infty}  \frac{\displaystyle \sum_{x\in  X } e^{-s_n d(x, z)} f(x) }{\displaystyle{  \sum_{x \in X}  e^{-s_n d(x, z)} }}
\end{align}
holds in the norm topology of $\mathscr{H}$ for all $z\in M$. 
\end{itemize}
\end{theorem}

For proving Theorem \ref{thm-general-pos},  we prepare some Lemmata.
\begin{lemma}\label{prop-uc-cv}
Let $(\mathscr{B}, \ge)$ be a partially ordered Banach space whose partial order is compatible with the norm. For any sequence $(u_n)_{n\ge 1}$ in $\mathscr{B}_{+}$, we have 
\begin{itemize}
\item If the series $\sum_{n = 1}^\infty u_n$ converges in $\mathscr{B}$, then it converges  unconditionally.
\item Let $(a_n)_{n\ge 1}$ and $(b_n)_{n\ge 1}$ be two sequences of positive numbers with $\sup_{n\in \N} \frac{b_n}{a_n} <\infty$. If the series  $\sum_{n = 1}^\infty a_n u_n$ converges unconditionally, then so is the series $\sum_{n= 1}^\infty b_n u_n$. 
\end{itemize}
\end{lemma}

{\flushleft \it Proof.} Positivity of the terms of our series implies,  for any bijection $\tau$ of $\N$ and any $N_1, N_2 \in \N$, $N_1\le N_2$, the inequality
\begin{equation}\label{perm-bnd}
\left\| \sum_{n = N_1}^{N_2} u_{\tau(n)} \right \|_{\mathscr{B}} \le C \left \| \sum_{k = \min \{\tau (n): N_1 \le n \le N_2 \}}^{\max\{\tau (n): N_1 \le n \le N_2 \}} u_k \right\|_{\mathscr{B}},
\end{equation}
where $C$ is the same constant in Definition \ref{def-order-norm}.  
As $N_1$ grows, $\min\{ \tau(n): N_1 \le n \le N_2\}$ also grows,thus \eqref{perm-bnd} implies that  
the series $\sum_{n=1}^\infty u_{\tau(n)}$ converges in $\mathscr{B}$. For $N\in \N$, set 
\[
L_N: = \min\Big( \tau(\{1, \cdots, N\})  \Delta \{1,\cdots, N\}\Big), \quad U_N: = \max\Big( \tau(\{1, \cdots, N\})  \Delta \{1,\cdots, N\}\Big).
\]
Using positivity again, we have 
\begin{multline*}
\left\| \sum_{n = 1}^N  u_{\tau(n)}  - \sum_{n = 1}^N  u_{n}   \right \|_{\mathscr{B}} = \left\| \sum_{k \in \tau(\{1, \cdots, N\}) \setminus \{1,\cdots, N\}}  u_k  - \sum_{\ell \in \{1, \cdots, N\} \setminus  \tau(\{1,\cdots, N\})} u_{\ell}   \right \|_{\mathscr{B}} \le
\\
\le
 \left\| \sum_{k \in \tau(\{1, \cdots, N\}) \setminus \{1,\cdots, N\}}  u_k  \right \|_{\mathscr{B}}  +  \left\|  \sum_{\ell \in \{1, \cdots, N\} \setminus  \tau(\{1,\cdots, N\})} u_{\ell}   \right \|_{\mathscr{B}}  \le 2  C \left\| \sum_{n  =  L_N}^{U_N} u_n \right \|_{\mathscr{B}}.
\end{multline*}
Since $L_N \to\infty$ as $N\to\infty$, we obtain  $\sum_{n = 1}^\infty  u_{\tau(n)}   =  \sum_{n = 1}^\infty u_{n}$. This proves the first statement of the lemma. 

For the second statement, we only need to observe that for any $N_1, N_2 \in \N, N_1 \le N_2$, we have $0 \le \sum_{n=N_1}^{N_2} b_n u_n \le   \sup_{n'\in \N} (b_{n'} a_{n'}^{-1})\sum_{n=N_1}^{N_2} a_n u_n $ and hence 
\[
\left\| \sum_{n=N_1}^{N_2} b_n u_n \right\| \le C  \sup_{n'\in \N} (b_{n'} a_{n'}^{-1}) \left\|\sum_{n=N_1}^{N_2} a_n u_n\right\|. \qed
\]

\begin{lemma}\label{lem-dense-to-all}
Let $(\mathscr{B}, \ge)$ be a partially ordered Banach space equipped with a partial order compatible with the norm.  Let $(u_\ell)_{\ell \ge 1}$ be a sequence in $\mathscr{B}_{+}$ and $(a_\ell)_{\ell\ge 1}$ be  a sequence of positive functions on $\N\times \N$ such that for any $\ell, n, k\in\N$, the limits 
\[
a_\ell(n, \infty):  =  \lim_{k'\to\infty} a_\ell(n, k'), \quad a_\ell (\infty, k): = \lim_{n'\to\infty } a_\ell (n', k) 
\]
exist.  Assume that for any $n, k \in \N $, all the series 
\begin{align*}
S(n, k) :  = \sum_{\ell=1}^\infty  a_\ell(n, k) u_\ell, &  \quad S(n, \infty) :  = \sum_{\ell=1}^\infty  a_\ell(n, \infty) u_\ell 
\\
s(n, k): = \sum_{\ell=1}^\infty a_\ell(n, k), & \quad s(n, \infty): = \sum_{\ell=1}^\infty a_\ell(n, \infty)
\end{align*}
are  convergent in $\mathscr{B}$ and in $\R$ respectively. Assume also that the following limits 
\[
  L(k): =   \lim_{n\to\infty} \frac{S(n, k)}{s(n,k)}, \quad L(\infty): =  \lim_{k\to\infty} L(k)
\]
exist and the convergences take place in the norm topology of $\mathscr{B}$. Assume moreover there exists $D(n, k)\ge 1$ for any $n, k \in \N$ such that 
\begin{align}\label{lim-d-1}
\lim_{k\to\infty} \sup_{n\in\N} | D(n, k)^2  - 1|=0
\end{align}
and
\begin{align}\label{def-dnk}
\frac{1}{D(n, k)} \le \frac{a_\ell (n, \infty)}{a_\ell (n, k)} \le D(n, k) \quad \text{for all $\ell, n, k \in \N$}.
\end{align} Then 
\[
\lim_{n\to\infty} \frac{S(n, \infty)}{s(n, \infty)}  = \lim_{k\to\infty} \Big(\lim_{n\to\infty} \frac{S(n,k)}{s(n, k)}\Big). 
\]  
\end{lemma}

{\flushleft  \it Proof. }
By \eqref{def-dnk}, for any $n, k\in\N$, we have  
\begin{align*}
D(n, k)^{-1} S(n, k) \le S(n, \infty) \le   D(n, k) S(n, k);
\\ 
D(n, k)^{-1} s(n, k) \le s(n, \infty) \le   D(n, k) s(n, k).
\end{align*}
where the inequalities in the above first line are the order inequalities in $\mathscr{B}_{+}$. 
It follows that for any $n, k \in\N$, we have 
\[
\frac{1}{D(n, k)^2} \frac{S(n,k )}{s(n,k)} \le \frac{S(n, \infty)}{s(n, \infty)} \le D(n, k)^2 \frac{S(n, k)}{ s(n, k)}
\]
and hence 
\[
\frac{1}{D(n, k)^2} \frac{S(n,k )}{s(n,k)}  - L(\infty) \le \frac{S(n, \infty)}{s(n, \infty)}  - L (\infty) \le D(n, k)^2 \frac{S(n, k)}{ s(n, k)} - L(\infty). 
\]
Using the assumption on the partial order and Definition \ref{def-com} and writing $\| \cdot\| = \|\cdot\|_{\mathscr{B}}$ for brevity, for any $n, k\in \N$, we have 
\begin{multline*}
\left\| \frac{S(n, \infty)}{s(n, \infty)}  - L (\infty)\right\| \le C \max_{\pm} \left\| D(n, k)^{\pm 2} \frac{S(n, k)}{ s(n, k)} - L(\infty)\right\| \le 
\\
\le   C  \left[\max_{\pm} \left\| D(n, k)^{\pm 2} \frac{S(n, k)}{ s(n, k)} - L(k)\right\| +  \|L(k) - L(\infty) \|\right] \le 
\\
\le C  \left[\max_{\pm}  \left(  D(n, k)^{\pm 2} \left\|  \frac{S(n, k)}{ s(n, k)} - L(k)\right\| +  | D(n, k)^{\pm 2} - 1 |\| L(k)\|\right) +  \|L(k) - L(\infty) \|\right]
\\
\le C  \left[  \sup_{n'\in \N} D(n', k)^2 \left\|  \frac{S(n, k)}{ s(n, k)} - L(k)\right\| +  \sup_{n'\in\N} (D(n', k)^2 - 1) \| L(k)\| +  \|L(k) - L(\infty) \|\right].
\end{multline*}
Consequently, for any  $k\in\N$, we have 
\[
\limsup_{n\to\infty} \left\| \frac{S(n, \infty)}{s(n, \infty)}  - L (\infty)\right\| \le C  \left[   \sup_{n'\in\N} (D(n', k)^2 - 1) \| L(k)\| +  \|L(k) - L(\infty) \|\right].
\]
Since $k$ is arbitrary, by the assumption \eqref{lim-d-1} and  $\lim_{k\to\infty} \| L(k) - L(\infty)\| =0$, we obtain the desired limit equality 
\[
\limsup_{n\to\infty} \left\| \frac{S(n, \infty)}{s(n, \infty)}  - L (\infty)\right\| = 0. \qed
\]

\begin{proof}[Proof of Theorem \ref{thm-general-pos}]
Fix a function $f \in \mathrm{CMVP}(\mu_M; \mathscr{H}_{+})$, a sequence $(s_n)_{n\ge 1}$ in $(h_M, \infty)$ satisfying \eqref{fast-to-cri} and a countable dense subset $\mathscr{A}\subset M$, . Since $\mathscr{A}$ is countable, by Theorem \ref{thm-H-L}, there exists a subset $\Omega \subset \Conf(M)$ with $\Pi(\Omega)  = 1$ such that   the convergence \eqref{thm-H-L-small-goal} and the limit equality \eqref{thm-H-L-goal} hold for any  $X\in \Omega$ at all points $z\in \mathscr{A}$. 

Now we show that any $X\in \Omega$ satisfies all the assertions of Thereom \ref{thm-general-pos}. In what follows, we fix any $X \in \Omega$,  any $s>h_M$ and $z'\in M$.

For proving the first statement of Theorem \ref{thm-general-pos}, we fix any point $z_0\in \mathscr{A}$ and any integer $n_0$, large enough such that $s_{n_0} \le s$.  By the first statement of Lemma \ref{prop-uc-cv}, since $f$ takes values in $\mathscr{H}_{+}$, the convergence
\[
 \sum_{k  =0}^{ \infty }  \sum_{x\in  X \atop k \le d(x, z_0) < k + 1 } e^{-s_{n_0} d(x, z_0)} f(x)
\]
 implies the unconditional convergence of the series 
\[
\sum_{x\in  X} e^{-s_{n_0} d(x, z_0)} f(x).
\]
  Now by the second statement of Lemma \ref{prop-uc-cv} and the following order inequality 
\[
0\le e^{-s d(x, z')} f(x) \le   e^{-s d(z', z_0)} e^{-s_{n_0} d(x, z_0)} f(x), 
\]
we immediately obtain the unconditional convergence of the series
\[
  \sum_{x\in  X} e^{-s d(x, z')} f(x). 
\]

Now we proceed to the proof of the second statement of Theorem \ref{thm-general-pos}. Since $\mathscr{A}\subset M$ is dense,  we may choose a sequence  $(z_k)_{k\ge 1}$ in $\mathscr{A}$ converges to our fixed point $z' \in M$. By the first statement of the theorem,  for any $s>h_M$ and any $z\in M$, the series  \eqref{uncon-cv-s} convergence unconditionally and  thus in particular,  we have the equality 
\begin{align}\label{pv-un-eq}
 \sum_{k  =0}^{ \infty }  \sum_{x\in  X \atop k \le d(x, z) < k + 1 } e^{-s d(x, z)} f(x) =  \sum_{x\in  X } e^{-s d(x, z)} f(x).
\end{align}
Since $z_k \in \mathscr{A}$ for all $k$, the limit equality \eqref{thm-H-L-goal}  holds for $z_k$ and our fixed configuration $X$, which, when combined with the equality \eqref{pv-un-eq}, implies 
\begin{align}\label{z-k-ok}
f(z_k)= \lim_{n\to\infty} \frac{ \displaystyle \sum_{x\in X} e^{-s_n d(x, z_k)} f(x)}{\displaystyle \sum_{x\in X} e^{-s_n d(x, z_k)} }, \quad \text{for all $k\in \N$.}
\end{align}  
For applying Lemma \ref{lem-dense-to-all}, for any $x \in X$ (if particles in the configuration $X$ has multiplicities, we repeat their multiplicities),  set 
\[
a_x (n, k) = e^{-s_n d(x, z_k)}. 
\]
Then we have 
\[
a_x(\infty, k) = \lim_{n'\to\infty} a_x(n', k) = e^{- h_M d(x, z_k)}, \quad a_x(n, \infty) = \lim_{k'\to\infty} a_x(n, k') = e^{-s_n d(x, z')}. 
\]
By our choice of $\Omega \subset \Conf(M)$ and the definition of $a_x(n,k)$, it is clear that  all assumptions of Lemma \ref{lem-dense-to-all} are satisfied. For instance, 
\[
e^{-s_n d(z_k, z')} \le \frac{a_x(n, \infty)}{a_x(n, k)} = e^{-s_n d(x, z') + s_n d(x, z_k)} \le e^{ s_n d(z_k, z')} = : D(n, k), \quad \text{for all $x\in X$}
\]
and since $\gamma  = \sup_{n\in \N} s_n \in \R^{+}$,  we have 
\[
\lim_{k\to\infty}\sup_{n\in \N} |D(n, k)^2 -1|  = \lim_{k\to\infty} \Big(e^{ 2 \gamma d(z_k, z')}  - 1\Big) = 0. 
\]
 Therefore, by Lemma \ref{lem-dense-to-all}, the equalities \eqref{z-k-ok} and the assumption that $f \in C(M; \mathscr{H})$,  we have  
\[
\lim_{n\to\infty} \frac{ \displaystyle \sum_{x\in X} e^{-s_n d(x, z')} f(x)}{\displaystyle \sum_{x\in X} e^{-s_n d(x, z')} }  = \lim_{k\to\infty} \lim_{n\to\infty} \frac{ \displaystyle \sum_{x\in X} e^{-s_n d(x, z_k)} f(x)}{\displaystyle \sum_{x\in X} e^{-s_n d(x, z_k)} } = \lim_{k\to\infty} f(z_k) = f(z'). 
\]
This completes the proof of the theorem. 
\end{proof}

\subsubsection{Generalized Margulis functions}
The following stronger assumption on $(M, d, \mu_M)$ will be used later. 
\begin{assumption}\label{ass-M-pvg}
There exist $h_M>0, \beta \ge 0$ and, for any $z\in M$, $c_z>0$, such that 
\[
\lim_{r\to\infty} \frac{\mu_M(B(z, r))}{c_z r^\beta e^{rh_M}} = 1.
\]
\end{assumption}

If $(M, d, \mu_M)$ satisfies  Assumption \ref{ass-M-pvg}, then the function $M \ni z \mapsto c_z \in \R_{>0}$ will be referred to as the generalized Margulis function for $(M, d, \mu_M)$.   

We prepare a simple lemma. Recall the notation $\overline{\sigma}(z, s)$ introduced in \eqref{notation-T-sigma}. 
\begin{lemma}\label{lem-R-1-seq}
Assume that  $(M, d, \mu_M)$ satisfies Assumption  \ref{ass-M-pvg}. Let $\Pi$  be a point process on $M$ with first intensity $\mu_M$. Then for any $z\in M$, there exists $C_z > 0$, such that 
\begin{align}\label{lem-R-1-goal}
\lim_{s\to h_M^{+}}  (s - h_M)^{1 + \beta} \overline{\sigma}(z, s)   = C_z. 
\end{align}
\end{lemma}

{\flushleft \it Proof. }
Set
$
\kappa_z(r) : = \frac{\mu_M(B(z, r))}{c_z r^\beta e^{rh_M}}. 
$
By Assumption \ref{ass-M-pvg}, $\lim_{r\to\infty} \kappa_z(r) = 1$. From \eqref{int-rep-exp} we obtain
\begin{multline*}
\overline{\sigma}(z, s) = \int_M e^{-sd(x, z)} d\mu_M(x)  = \int_M d\mu_M(x)\int_{\R^{+}}   s e^{-s r}  \cdot \mathds{1}( d(x, z) < r) dr = 
\\
 =  s \int_{\R^{+}} e^{- s r}  \mu_M( B(z, r)) dr 
= c_z s \int_{\R^{+}} \kappa_z(r) r^\beta e^{- (s - h_M) r} dr  = 
\\
= \frac{s c_z}{(s-h_M)^{1 + \beta}} \int_{\R^{+}} \kappa_z \Big(\frac{t}{s - h_M}\Big) t^\beta e^{-t}dt. 
\end{multline*}
The Dominated Convergence Theorem then gives  
\[
\lim_{s \to h_M^{+}}  (s - h_M)^{1 + \beta} \overline{\sigma}(z, s)= h_M c_z  \int_{\R^{+}}  t^\beta e^{-t} dt > 0.  \qed
\]

Now assume that $(\mathscr{H}, \ge)$ is a Hilbert space over $\R$ equipped with a partial order compatible with the norm. 
\begin{definition}\label{def-cmvp-sharp}
 Let $\mathrm{CMVP}^{\sharp}(\mu_M; \mathscr{H}_{+}) \subset \mathrm{CMVP}(\mu_M; \mathscr{H}_{+})$ be a subset consisting all functions $f \in \mathrm{CMVP}(\mu_M; \mathscr{H}_{+})$ such that there exists a strictly decreasing  sequence $(\varepsilon_n)_{n\ge 1}$, depending on the function $f$,  of positive numbers converging to $0$ and 
\begin{align}\label{e-n-gap}
\lim_{n\to\infty} \frac{\varepsilon_{n+1}}{\varepsilon_n} = 1
\end{align}
such that 
\begin{align}\label{gap-sum}
\sum_{n=1}^\infty \varepsilon_n^{2+2 \beta} \int_M e^{-2 \varepsilon_n d(x, o)}  \| f(x)\|_{\mathscr{H}}^2 e^{-2h_M d(x, o)} d\mu_M(x) <\infty. 
\end{align}
\end{definition}

\begin{example}
If a function $f\in \mathrm{CMVP}(\mu_M; \mathscr{H}_{+})$ satisfies:  there exist constants $C>0, \alpha >0$ such that for any $\varepsilon\in (0, 1)$, we have 
\begin{align}\label{bdd-e-a}
\int_M e^{-2 \varepsilon d(x, o)} \| f(x)\|_{\mathscr{H}}^2 e^{-2h_M d(x, o)} d\mu_M(x) \le C  \varepsilon^{-2 - 2 \beta}\cdot \varepsilon^\alpha, 
\end{align}
then $f\in \mathrm{CMVP}^{\sharp}(\mu_M; \mathscr{H}_{+})$ and we can choose $\varepsilon_n = n^{-2/\alpha}$. If in the upper estimate \eqref{bdd-e-a}, the term $\varepsilon^\alpha$ is replaced by $\log^{-1- \alpha} ( 1/\varepsilon)$, then we also have $f \in \mathrm{CMVP}^{\sharp}(\mu_M; \mathscr{H}_{+})$ and we can choose $\varepsilon_n = e^{- n^\gamma}$ with $(1 + \alpha)^{-1} < \gamma < 1$. 
\end{example}

\begin{theorem}\label{thm-n-sub}
Assume that $(M, d, \mu_M)$ satisfies Assumption  \ref{ass-M-pvg}. Let $\Pi$  be a point process on $M$ satisfying Assumption \ref{ass-pp} and  with first intensity measure $\mu_M$. Fix a function   $f \in \mathrm{CMVP}^{\sharp}(\mu_M; \mathscr{H}_{+})$. Then $\Pi$-almost every $X$ satisfies the following properties:
\begin{itemize}
\item The series 
\[
  \sum_{x\in  X} e^{-s d(x, z)} f(x)  
\]
converges unconditionally in $\mathscr{H}$ for all $s>h_M$ and all $z\in M$.  
\item For all $z\in M$, we have 
\begin{align}\label{thm-n-sub-goal}
f(z) = \lim_{s \to h_M^{+}}  \frac{\displaystyle \sum_{x\in  X } e^{-s d(x, z)} f(x) }{\displaystyle{  \sum_{x \in X}  e^{-s d(x, z)} }},
\end{align}
where the convergence takes place in the norm-topology of $\mathscr{H}$. 
\end{itemize}
\end{theorem}

\begin{lemma}\label{lem-ans}
Let $(\mathscr{B}, \ge)$ be a partially ordered Banach space equipped with a partial order compatible with the norm.  Let $(u_\ell)_{\ell \ge 1}$ be a sequence in $\mathscr{B}_{+}$ and $\bar{\eta}, \eta_\ell, \ell\ge 1$ be non-increasing positive functions on $(0, \infty)$. Assume that there exists a strictly decreasing sequence $(\varepsilon_n)_{n\ge 1}$ in $(0, \infty)$ such that the following limit equalities hold: 
\begin{align}\label{ans-ass}
\lim_{n\to\infty} \frac{\bar{\eta}(\varepsilon_{n+1})}{\bar{\eta}(\varepsilon_{n})} = 1, \quad \lim_{n\to\infty} \frac{\sum_{\ell\in\N} \eta_\ell (\varepsilon_n)}{\bar{\eta}(\varepsilon_n)} = 1, \quad  \lim_{n\to\infty} \frac{\sum_{\ell\in\N} \eta_\ell (\varepsilon_n)u_\ell}{\sum_{\ell\in\N} \eta_\ell (\varepsilon_n)} = v\in \mathscr{B}.  
\end{align}
Then 
\begin{align}\label{lem-ans-goal}
\lim_{t \to 0^{+}} \frac{\sum_{\ell\in\N} \eta_\ell (t)u_\ell}{\sum_{\ell\in\N} \eta_\ell (t)} = v.
\end{align}
\end{lemma}

\begin{proof}
It suffices to show that 
\[
\lim_{t \to 0^{+}} \frac{\sum_{\ell\in\N} \eta_\ell (t)}{\bar{\eta} (t)} = 1 \an \lim_{t \to 0^{+}} \frac{\sum_{\ell\in\N} \eta_\ell (t)u_\ell}{\bar{\eta}(t)} = v.
\]
We prove the second limit equality, the proof of the first limit equality is similar. Consider any $t \in (0, \varepsilon_1)$.  For such $t$,  since the sequence $(\varepsilon_n)_{n\ge 1}$ strictly decreases to $0$, there exists a unique  $n_t \in \N$ such that 
\[
 \varepsilon_{n_t + 1}\le t <  \varepsilon_{n_t}. 
\]
Since $(u_\ell)_{\ell\ge 1}$ is a sequence in $\mathscr{B}_{+}$ and all the functions $\bar{\eta}, \eta_\ell$ are non-increasing, we have 
\begin{align*}
\bar{\eta} (\varepsilon_{n_t}) \le & \bar{\eta} (t) \le \bar{\eta} (\varepsilon_{n_t+1});
\\
 \sum_{\ell\in\N} \eta_\ell (\varepsilon_{n_t })u_\ell \le & \sum_{\ell\in\N} \eta_\ell (t)u_\ell \le \sum_{\ell\in\N} \eta_\ell (\varepsilon_{n_t + 1})u_\ell.
\end{align*}
It follows that 
\begin{align}\label{3-rat}
\frac{\sum_{\ell\in\N} \eta_\ell (\varepsilon_{n_t})u_\ell}{\bar{\eta} (\varepsilon_{n_t+1})} \le \frac{ \sum_{\ell\in\N} \eta_\ell (t)u_\ell }{\bar{\eta}(t)} \le \frac{\sum_{\ell\in\N} \eta_\ell (\varepsilon_{n_t + 1})u_\ell}{\bar{\eta} (\varepsilon_{n_t})}.
\end{align}
For brevity, write 
\begin{align*}
R(t) = \frac{ \sum_{\ell\in\N} \eta_\ell (t)u_\ell }{\bar{\eta}(t)}, \quad & R^{-}(t)= \frac{\sum_{\ell\in\N} \eta_\ell (\varepsilon_{n_t})u_\ell}{\bar{\eta} (\varepsilon_{n_t})}, \quad R^{+}(t): = \frac{\sum_{\ell\in\N} \eta_\ell (\varepsilon_{n_t+1})u_\ell}{\bar{\eta} (\varepsilon_{n_t+1})},
\\
  & \beta(t) = \frac{\bar{\eta}(\varepsilon_{n_t+1})}{\bar{\eta}(\varepsilon_{n_t})}\ge 1.
\end{align*}
Then \eqref{3-rat} can be rewritten as 
\[
R^{-}(t) \beta(t)^{-1} \le  R(t) \le R^{+}(t) \beta(t).
\]
Hence 
\[
R^{-}(t) \beta(t)^{-1}  -v \le  R(t) - v  \le R^{+}(t) \beta(t) -v.
\]
Consequently, writing $\|\cdot\| = \| \cdot\|_{\mathscr{B}}$ for brevity and using the compatibility assumption of the partial order on $\mathscr{B}$ (see Definition \ref{def-com}), we obtain 
\begin{multline*}
\| R(t)-v\| \le C\max_{\pm} \| R^{\pm}(t) \beta(t)^{\pm 1} - v\| \le C \max_{\pm} \Big( \beta(t)^{\pm 1} \| R^{\pm} (t) - v\|  +  | \beta(t)^{\pm 1} - 1 | \| v\|\Big)\le
\\
\le  C \Big( \beta(t)  \max_{\pm}\| R^{\pm} (t) - v\|  +  | \beta(t) - 1 | \| v\|\Big).
\end{multline*}
It is clear that $t \to 0^{+}$ if and only if $n_t \to \infty$. Therefore, the  assumption \eqref{ans-ass} implies $\lim_{t \to 0^{+}} \beta(t) = 1,   \lim_{t\to 0^{+}} R^{\pm} (t)  = v$, 
whence $\lim_{t \to 0^{+}} R(t) = 0$. 
\end{proof}

{\flushleft \it Proof of Theorem \ref{thm-n-sub}. }
 Since $\mathrm{CMVP}^\sharp(\mu_M; \mathscr{H}_{+})\subset \mathrm{CMVP}(\mu_M; \mathscr{H}_{+})$, the first statement of Theorem \ref{thm-n-sub} follows from the first statement of Theorem \ref{thm-general-pos}.  

Using the assumption $f\in \mathrm{CMVP}^\sharp(\mu_M; \mathscr{H}_{+})$,  we can choose a strictly decreasing  sequence $(\varepsilon_n)_{n\ge 1}$ of positive numbers converging to $0$ satisfying both \eqref{e-n-gap} and \eqref{gap-sum}.  Recall the notation $\sigma(z, s; X)$ and $\bar{\sigma}(z, s)$ introduced in \eqref{notation-T-sigma}.  By  \eqref{lem-R-1-goal} and \eqref{gap-sum},  the sequence $(s_n = h_M + \varepsilon_n)_{n\ge 1}$ satisfies the condition \eqref{fast-to-cri}. Thus by Theorem \ref{thm-general-pos}  and also Proposition \ref{prop-const-fn}, there exists a subset $\Omega \subset \Conf(M)$ with $\Pi(\Omega) = 1$ such that for all $X\in \Omega, z \in M$, we have 
\begin{align}\label{f-ep-s}
  f(z) = \lim_{n\to\infty}  \frac{\displaystyle \sum_{x\in  X } e^{- (h_M + \varepsilon_n) d(x, z)} f(x) }{\displaystyle{  \sum_{x \in X}  e^{- (h_M + \varepsilon_n) d(x, z)} }}  \an  \lim_{n\to\infty} \frac{\sigma(z, h_M + \varepsilon_n; X)}{\overline{\sigma}(z, h_M + \varepsilon_n)} = 1. 
\end{align}
Now fix $X\in \Omega$ and $z\in M$.  For applying Lemma \ref{lem-ans}, for $t>0$ and $x\in X$,  set 
\[
\bar{\eta} (t) : = \bar{\sigma}(z, h_M + t), \quad  \eta_x(t) : = e^{- (h_M  + t) d(x, z)}, \quad u_x:  = f(x).
\]
We now verify that all assumptions of Lemma \ref{lem-ans} are satisfied.  First of all, it is clear that $\bar{\eta}, \eta_x$ are non-increasing and by assumption of $f$, we have $u_x \in \mathscr{H}_{+}$.   Note that the equalities \eqref{lem-R-1-goal} and \eqref{e-n-gap} imply 
\[
\lim_{n\to\infty} \frac{\bar{\eta}(\varepsilon_{n+1})}{\bar{\eta}(\varepsilon_n)}  = \lim_{n\to\infty}  \frac{\overline{\sigma}(z, h_M + \varepsilon_{n+1})}{\overline{\sigma}(z, h_M + \varepsilon_{n})} = 1. 
\]
The limit equalities \eqref{f-ep-s} can be rewritten as 
\[
\lim_{n\to\infty}  \frac{\displaystyle \sum_{x\in X} \eta_x(\varepsilon_n) u_x}{ \displaystyle \sum_{x\in X} \eta_x(\varepsilon_n)}   = f(z), \quad \lim_{n\to\infty} \frac{ \displaystyle \sum_{x\in X} \eta_x(\varepsilon_n)}{\bar{\eta}(\varepsilon_n)}=1. 
\]
This completes the verification of all assumptions of Lemma \ref{lem-ans}.  Therefore, by applying Lemma \ref{lem-ans}, we obtain the desired limit equality
\[
f(z)= \lim_{t \to 0^{+}} \frac{\displaystyle \sum_{x\in X} \eta_x(t) u_x}{ \displaystyle \sum_{x\in X} \eta_x(t)} = \lim_{s\to h_M^{+}}  \frac{\displaystyle \sum_{x\in  X } e^{- s d(x, z)} f(x) }{\displaystyle{  \sum_{x \in X}  e^{- s d(x, z)} }}. \qed
\]

\section{The complex hyperbolic spaces}\label{sec-chs}
\subsection{Main results for complex hyperbolic spaces}
Let $d\ge 1$ be an integer.  We apply the results in \S \ref{sec-gen}  to the triple $(\D_d, d_B, \mu_{\D_d})$ consisting of the ball $\D_d$ equipped with the Bergman metric $d_B(\cdot, \cdot)$ given by \eqref{Berg-metric-def} and the associated Bergman volume measure $\mu_{\D_d}$ given by \eqref{Berg-vol-def}.  We choose the origin $0\in \D_d$ as the base point.  The measure $\mu_{\D_d}$ is invariant under the group action of biholomorphisms of $\D_d$. Recall,  cf. e.g. Rudin \cite[Theorem 2.2.5]{Rudin-ball}, that any biholomorphism 
 $\psi \in \Aut(\D_d)$ has the form $\psi = U_1\varphi_a = \varphi_{b} U_2$, with $U_1, U_2$ are unitary transformations of $\C^d$ and $\varphi_a, \varphi_b\in \Aut(\D_d)$ the involutions defined by \eqref{inv-auto}.

Recall the definition in \S \ref{sec-intro-chs} of  the $\MM$-harmonic functions on $\D_d$.  We  denote by $\MM\HH(\D_d)$ the set of all complex-valued $\MM$-harmonic functions on $\D_d$.  If $\frak{B}$ is a Banach space, then we denote by $\MM\HH(\D_d; \frak{B})$ the set of all $\frak{B}$-valued $\MM$-harmonic functions on $\D_d$.  Consider a reproducing kernel Hilbert space $\mathscr{H}(K) \subset \MM \HH (\D_d)$ with  a reproducing kernel $K$.  To any $x\in \D_d$, we assign a function $K_x \in \mathscr{H}(K)$ by setting 
\[
\text{$K_x(y) = K(y, x)$ for $y\in \D_d$.}
\] 

\begin{assumption}\label{ass-gcb}
There exists a non-decreasing sub-exponential function $\Lambda: \N \rightarrow \R^{+}$ such that for any $k\in \N$, we have 
\begin{align}\label{rough-es-cb}
\int_{A_k^{(d)}(0)} K(z, z) d\mu_{\D_d}(z)   \le \Lambda (  k )   e^{2 d k}, 
\end{align}
where $A_k^{(d)}(0) = \{z\in\D_d: k \le d_B(z, 0) < k + 1\}$. The limit equality holds: 
\begin{align}\label{fine-es-cb}
\lim_{\alpha \to 0^{+}}  \alpha^2 \int_{\D_d} K(z, z) (1- |z|^2)^{\alpha + d -1}  dv_d(z) = 0.
\end{align}
\end{assumption}

\begin{remark}\label{rem-int-berg}
For any $\alpha>0$ and any Banach space $\frak{B}$, set 
\[
B_{\alpha + d -1}^2(\D_d; \mathscr{H}): = \left\{f \in \MM\HH(\D_d; \frak{B}) \Big| \int_{\D_d} \| f(z) \|_{\frak{B}}^2 ( 1 - |z|^2)^{\alpha + d - 1} dv_d(z) <\infty \right\}.
\]
Then the condition \eqref{fine-es-cb} means that the function $f: \D_d \rightarrow \mathscr{H}(K)$ defined by 
$
\D_d \ni x \mapsto  K_x \in \mathscr{H}(K)
$
satisfies
\[
f \in \bigcap_{\alpha > 0} B_{\alpha+d -1}^2(\D_d; \mathscr{H}(K)) \an \text{$
\| f\|_{B_{\alpha+d-1}^2}  = o (\alpha^{-1})$ as $\alpha \to 0^{+}$.}
\] 
\end{remark}

\begin{example}
 Consider the non-negative definite kernel $K: \D_d \times \D_d \rightarrow \C$ given by 
\begin{align}\label{rep-kernel-cb}
K(z, w) = \sum_{n\in \N_0^d} a_n z^n \bar{w}^n,
\end{align}
where $\N_0  = \N\cup \{0\}$ and  $z^n: = z_1^{n_1} \cdots z_d^{n_d}$ and $a_n \ge 0$.   Given an $n \in \N_0^d$, we write $| n| = n_1 + \cdots + n_d$.  If the coefficients $a_n$ in \eqref{rep-kernel-cb} satisfy 
\begin{align}\label{cb-sm-coef}
\lim_{k\to\infty} \frac{ \displaystyle\sum_{n: | n| = k} a_n }{k^{d-1} \log (k+2)} = 0,
\end{align}
then  the reproducing kernel defined in \eqref{rep-kernel-cb} satisfies Assumption \ref{ass-gcb}.  See Proposition \ref{prop-cb-kernel} below for more details. 
\end{example}

\begin{example}
For any $T> 0$, the Bergman space $A^2(\D_d, \omega_T^{(d)})$ corresponds to the weight
\begin{align}\label{def-w-cb}
\omega_T^{(d)}(z) = \frac{1}{(1 - |z|^2) \log \left(\frac{4}{1 - |z|^2}\right) \log^{1+ T} \left( \log\left(\frac{4}{1 - |z|^2}\right)\right) }, \quad z \in \D_d, 
\end{align}
 satisfies  Assumption \ref{ass-gcb}.  See Proposition \ref{prop-w-cb} below for more details. 
\end{example}

\begin{example}
As a consequence of Proposition \ref{prop-mgc}, if a reproducing kernel Hilbert space  $\mathscr{H}(K) \subset \MM\HH(\D_d)$ has a reproducing kernel $K$ satisfying the inequalities
\[
\int_{A_k^{(d)}(0)} K(z,z) d\mu_{\D_d}(z) \le \Theta(k) k e^{2d k}, \quad \text{for all $k \in \N$},
\]
 with $\Theta: \N \rightarrow \R^{+}$ non-increasing and $\lim_{n\to\infty} \Theta(n)  = 0$,   then $K$ satisfies Assumption \ref{ass-gcb}. 
\end{example}

\begin{theorem}\label{thm-rep-cb}
Let $\mathscr{H}(K) \subset \MM\HH(\D_d)$ be a reproducing kernel Hilbert space satisfying Assumption \ref{ass-gcb}. Fix a countable dense subset $\mathcal{D}\subset \D_d$ and a  sequence $(s_n)_{n\ge 1}$  in $(d, \infty)$ converging to $d$  and  
\begin{align}\label{fast-s-to-d}
\sum_{n=1}^\infty (s_n-d)^2 \int_{\D_d} K(z, z) (1 - |z|^2)^{s_n - 1}  dv_d(z)<\infty.
\end{align}
  If $\Pi$ is a point process on $\D_d$ satisfying  Assumption \ref{ass-pp} and having first intensity measure $\lambda \mu_{\D_d}$ with $\lambda> 0$ a constant, then $\Pi$-almost any $X \in \Conf(\D_d)$ satisfies:
\begin{enumerate}
\item For any $n$ and any $z\in \mathcal{D}$, we have 
\[
\sum_{k=0}^\infty \Big\| \sum_{x \in X \atop k\le d_B (z,x) < k+1} e^{-s_n d_B(z, x)} K_x\Big\|_{\mathscr{H}(K)}< \infty.
\]
\item The following limit equality holds for any $z\in \mathcal{D}$:
\[
\lim_{n\to\infty}  \left\| \frac{\displaystyle{\sum_{k=0}^\infty \sum_{x \in X \atop k\le d_B (z,x) < k+1} e^{-s_n d_B(z, x)} K_x}}{\displaystyle{  \sum_{x \in X} e^{-s_n d_B(z,x)} }}  - K_z \right\|_{\mathscr{H}(K)} = 0
\]
\item   The limit equality 
\begin{align}\label{thm-cb-rep-goal}
 f(z) = \lim_{n\to\infty} \frac{\displaystyle{\sum_{k=0}^\infty \sum_{x \in X \atop k\le d_B (z,x) < k+1} e^{-s_n d_B(z, x)} f(x)}}{\displaystyle{  \sum_{x \in X} e^{-s_n d_B(z,x)} }}
\end{align}
 holds simultaneously for all $f \in \mathscr{H}(K)$ and any $z\in \mathcal{D}$. 
\end{enumerate}
\end{theorem}

\begin{remark}
The exponent $s_n-1$ in \eqref{fast-s-to-d} comes from the term $\alpha  + d -1$ in \eqref{fine-es-cb} and the equality $(s_n -d) + d -1 = s_n -1$. 
\end{remark}

\begin{remark}
Under the assumptions of Theorem \ref{thm-rep-cb},  $\Pi$-almost any $X \in \Conf(\D_d)$ is a uniqueness set for $\mathscr{H}(K)$. 
\end{remark}

Recall the definition of Poisson transformation \eqref{def-cb-poi} of a  signed measure on $\Sph_d$ of total variation. 
Given any finite Positive Borel measure $\mu$ on $\Sph_d$. Set
\begin{align}\label{def-cm-hardy}
h^2(\D_d; \mu) : &= \left\{f: \D_d\rightarrow \C\Big|  f = P^b[g\mu], \, g \in L^2(\Sph_d, \mu)\right\}.
\end{align}

For any fixed $ \zeta \in \Sph_d$, the function $w \mapsto P^b(w, \zeta)$ is $\MM$-harmonic on $\D_d$.
In what follows, for any $w \in \D_d$, we denote 
\begin{equation}\label{def-pwb}
P_w^b(\zeta) = P^b(w, \zeta), \quad \zeta \in \Sph_d.
\end{equation}

\begin{theorem}\label{thm-intro-cm-no-ss}
Let $\mu$ be any Borel probability measure on $\Sph_d$.  If $\Pi$ is a point process on $\D_d$ satisfying  Assumption \ref{ass-pp} and having first intensity measure $\lambda \mu_{\D_d}$ with $\lambda> 0$ a constant, then 
for $\Pi$-almost any $X \in \Conf(\D_d)$ we have:
\begin{enumerate}
\item For any $s>d$ and any $z\in \D_d$, the series $\sum_{x\in X} e^{-s d_B(x, z)} P_x^b$ converges unconditionally in $L^2(\mu)  = L^2(\Sph_d, \mu)$.
\item  For any $z\in \D_d$, the functions \eqref{def-pwb} satisfy
\[
\lim_{s\to d^{+}} \left\| \frac{\displaystyle{\sum_{x \in X}   e^{-s d_B(x, z)} P_x^b}}{\displaystyle{  \sum_{x \in X}  e^{-s d_B(x, z)} }} - P_{z}^b \right\|_{L^2(\mu)} = 0.
\]
\item For any $ f  \in h^2(\D_d; \mu)$, any $z\in \D_d$ and any $s >d$, the series  $\sum_{x \in X}   e^{-s d_B(x, z)} f(x)$ converges absolutely  and 
\begin{align}\label{thm-cb-no-subseq-goal}
f(z) =    \lim_{s\to d^{+}}  \frac{ \displaystyle{ \sum_{x \in X}   e^{-s d_B(x, z)} f(x) } }{\displaystyle{  \sum_{x \in X}   e^{-s d_B(x, z)} }}. 
\end{align}
\item For all $z\in \D_d$, the following weak convergence of probability measures on $\overline{\D}_d$ holds:
\[
\lim_{s\to d^{+}}\frac{\displaystyle{\sum_{x \in X}   e^{-s d_B(x, z)} \delta_x}}{\displaystyle{  \sum_{x \in X}  e^{-s d_B(x, z)} }}  =  P^b(z, \zeta) d \sigma_{\Sph_d}(\zeta),
\]
where $\delta_x$ is the Dirac measure at the point $x$.
\end{enumerate}
\end{theorem}

\begin{remark}
The passage from the limit equality \eqref{thm-cb-rep-goal} for $z$ belongs to a fixed countable subset $\mathcal{D} \subset \D_d$ and a certain fixed sequence $s_n \to d^{+}$  to the limit equality \eqref{thm-cb-no-subseq-goal}  for all $z\in \D_d$ and  using the limit $s \to d^{+}$  requires substantial effort. 
\end{remark}

\begin{remark} \label{no-abs-conv}
If $d\mu  = d\theta/2\pi$ is the Haar measure on $\T = \partial \D$ and  $\Pi$ is the Poisson point process on $\D$ with intensity measure $\mu_\D$ and $1 < s \le 3/2$, then the Kolmogorov Three Series Theorem implies that  the series $\sum_{x \in X} e^{-s d_\D(x, z)} \| P_x\|_{L^2(\mu)}$ diverges for $\Pi$-almost every $X$ and all $z\in \D$. 
\end{remark}

For  point processes on $\D_d$ whose first intensity measures are not necessarily conformally invariant, we have the following result.   

Given $\alpha> -1$, let $A_\alpha^2(\D_d)$ denote the weighted Bergman space with respect to the classical weight $( 1- |z|^2)^{\alpha}$, see e.g. Zhu \cite[Chapter 2]{Zhu-UB} for more details. The reproducing kernel of $A_\alpha^2(\D_d)$ is given by 
\begin{align}\label{K-alpha}
K^\alpha(z, w) = \frac{C_{d, \alpha}}{ ( 1 - z \cdot \bar{w})^{ d+1+ \alpha}},
\end{align}
where $C_{d, \alpha}>0$ is an explicit constant depending on $d$ and $\alpha$. 

\begin{theorem}\label{thm-berg-cb}
 Let $\alpha> -1$ and $\beta \ge \alpha + d + 1 > d$.  Fix a countable dense subset $\mathcal{D}\subset \D_d$  and a  sequence $(s_n)_{n\ge 1}$  with $s_n > \beta$ such that $\sum_{n= 1}^\infty (s _n -\beta) <\infty$. 
  If $\Pi$ is a point process on $\D_d$ satisfying  Assumption \ref{ass-pp} and having first intensity measure 
\[
\frac{\lambda dv_d(z)}{(1 - |z|^2)^{\beta + 1}}
\]
 with $\lambda> 0$ a constant. Then $\Pi$-almost any $X \in \Conf(\D_d)$ satisfies:
\begin{enumerate}
\item For any $n$ and any $z\in \mathcal{D}$,  we have 
\[
\sum_{k=0}^\infty \Big\| \sum_{x \in X \atop k\le d_B (z,x) < k+1} e^{-s_n d_B(z, x)}  \left(\frac{|1 - x \cdot\bar{z}|^2}{1 - |z|^2}\right)^{\beta-d} K_x^\alpha \Big\|_{A^2_\alpha(\D_d)}< \infty.
\]
\item For any $z\in \mathcal{D}$, the following limit equality holds: 
\[
\lim_{n\to\infty}  \left\| \frac{\displaystyle{\sum_{k=0}^\infty \sum_{x \in X \atop k\le d_B (z,x) < k+1} e^{-s_n d_B(z, x)} \left(\frac{|1 - x \cdot \bar{z}|^2}{1 - |z|^2}\right)^{\beta-d}K_x^\alpha }}{\displaystyle{  \sum_{x \in X} e^{-s_n d_B(z,x)}  \left(\frac{|1 - x\cdot \bar{z}|^2}{1 - |z|^2}\right)^{\beta-d} }}  - K_z^\alpha \right\|_{A^2_\alpha(\D_d)} = 0
\]
\item   The limit equality 
\[
 f(z) = \lim_{n\to\infty} \frac{\displaystyle{\sum_{k=0}^\infty \sum_{x \in X \atop k\le d_B (z,x) < k+1} e^{-s_n d_B(z, x)} \left(\frac{|1 - x \cdot \bar{z}|^2}{1 - |z|^2}\right)^{\beta-d} f(x) }}{\displaystyle{  \sum_{x \in X} e^{-s_n d_B(z,x)} \left(\frac{|1 - x \cdot \bar{z}|^2}{1 - |z|^2}\right)^{\beta-d} }}
\]
 holds simultaneously for all $f \in A^2_\alpha(\D_d)$ and any $z\in \mathcal{D}$. 
\end{enumerate}
\end{theorem}

\begin{remark}
As an immediate consequence of Theorem \ref{thm-berg-cb}, under the assumption of Theorem \ref{thm-berg-cb}, $\Pi$-almost any $X \in \Conf(\D_d)$ is an $A_\alpha^{2}(\D_d)$-uniqueness set. This should be compared to Proposition \ref{prop-unique-disk} below. In dimension $d\ge 2 $ case,  since there does not seem to be a  notion of density in higher dimension in dealing with the uniqueness set, except as a consequence of Theorem \ref{thm-berg-cb}, we do not see an alternative proof of such result.   
\end{remark}

\begin{remark}
Theorem \ref{thm-berg-cb} can be generalized to more general reproducing kernel Hilbert spaces $\mathscr{H}(K) \subset \MM\HH(\D_d)$ with certain growth condition on the diagonal of the kernel $K$.  
\end{remark}

\subsection{Relation with density type results in dimension $d=1$ case} In dimension $d=1$ case,  the reader should compare Theorem \ref{thm-berg-cb} with the following Proposition \ref{prop-unique-disk}. From this comparison, we see that for a fixed function space of holomorphic functions on $\D$, to prove the simultaneous reconstruction for a configuration requires more than to prove that this configuration is the uniqueness set for our fixed function space on $\D$. Note that we will prove Proposition \ref{prop-unique-disk} by using a notion of density, see Hedenmalm-Korenblum-Zhu \cite[Section 4.2]{HKZ-Bergman}.

For an open arc $I \subset \T$ with $| I| <1$, the associated {\it Carleson square} is the set $Q(I) = \{z\in \D \setminus \{0\}: 1 - |I| < |z|, z/|z| \in I  \}$; for open arcs of bigger length, we let $Q(I)$ be the entire sector $Q(I) = \{z \in \D\setminus \{0\}: z/|z|\in I\}$. Let $F\subset \T$ be a finite subset.  Let  $\{I_k\}_k$ be the  complementary arcs of the subset $F$ in the unit circle $\mathbb{T}$, and set 
\begin{align*}
\widehat{\kappa}(F): = 1 - \sum_{k} \frac{| I_k|}{2\pi} \log \frac{|I_k|}{2 \pi} \an Q_F = \D \setminus \bigcup_{k} Q(I_k).
\end{align*}
For a sequence  $A\subset \D$, set 
   \[
D^{+}(A) : = \frac{1}{2}\limsup_{\widehat{\kappa}(F)\to\infty} \frac{\sum_{x\in A \cap Q_F} (1-|x|^2)}{\widehat{\kappa}(F)}.
\]

For any $\alpha>0$, set $\mathcal{A}^{-\alpha}(\D)$ the space of holomorphic functions  $f: \D\rightarrow \C$ such that $\sup_{z\in \D} (1 - |z|^2)^{\alpha} |f(z)|<\infty$. Set $\mathcal{A}^{-\infty}(\D)  : = \bigcup_{0< \alpha< \infty} \mathcal{A}^{-\alpha}(\D)$. We have, see Hedenmalm-Korenblum-Zhu \cite[Section 4.3]{HKZ-Bergman},  for any $\alpha> -1$, the equality 
\[
\mathcal{A}^{-\infty}(\D) = \bigcup_{0<p <\infty} A_\alpha^p(\D).
\] We will use the following result, see Hedenmalm-Korenblum-Zhu \cite[Corollary 4.17]{HKZ-Bergman}:  a sequence $A\subset \D$ is $\mathcal{A}^{-\infty}(\D)$-unique if and only if $D^{+}(A) = \infty$.

\begin{proposition}\label{prop-unique-disk}
 Let $\beta >1$.   If $\Pi$ is a point process on $\D$ satisfying  Assumption \ref{ass-pp} and having first intensity measure 
\[
\frac{\lambda dA(z)}{(1 - |z|^2)^{\beta + 1}}
\]
 with $\lambda> 0$ a constant. Then $\Pi$-almost any $X \in \Conf(\D)$ satisfies $
D^{+}(X) = \infty$ and is therefore an $\mathcal{A}^{-\infty}(\D)$-uniqueness set. 
\end{proposition}

\begin{proof}
 It suffices to consider the sets $F_n$ of equally distributed $n$ points in $\T$. We have $\widehat{\kappa}(F_n) = 1 + \log n$ and $Q_{F_n} = \{z\in \D: |z| \le 1 - 2 \pi /n\}$. For proving  $D^{+}(X) = \infty$ for $\Pi$-almost every $X$, we first compute 
\begin{multline*}
\E_\Pi \left(\sum_{x \in \X \cap Q_{F_n}} (1 - |x|^2)\right) = \int_{|z|\le 1 - 2 \pi /n} (1 -|z|^2) \frac{\lambda}{(1 - |z|^2)^{\beta + 1}} dA(z)=
\\
 =  \underbrace{\frac{\lambda }{\beta-1} \left[\left(\frac{n}{4 \pi ( 1- \pi /n)}\right)^{\beta -1}-1\right]}_{ \text{denoted by $c_n$}}. 
\end{multline*}
By Assumption \ref{ass-pp}, we have 
\begin{multline*}
\Var_\Pi \left(\sum_{x \in \X \cap Q_{F_n}} (1 - |x|^2)\right) \le C \E_\Pi  \left(\sum_{x \in \X \cap Q_{F_n}} (1 - |x|^2)^2 \right) = 
\\ 
= \int_{|z|\le 1 - 2 \pi /n} (1 -|z|^2)^2 \frac{\lambda}{(1 - |z|^2)^{\beta + 1}} dA(z) = \left\{  \begin{array}{lc}  \frac{\lambda}{\beta-2} \left[\left(\frac{n}{4 \pi ( 1- \pi /n)}\right)^{\beta -2}-1\right], &  \text{if $\beta \ne  2$}
\\
\lambda  \log \frac{n}{4\pi (1 - \pi/n)},  & \text{if $\beta =  2$}
 \end{array}\right..
\end{multline*}
Therefore, there exists a subsequence $n_k$ of natural numbers such that for $\Pi$-almost every configuration $X \in \Conf(\D)$, we have 
\[
\lim_{k\to\infty}\frac{\sum_{x \in X \cap Q_{F_{n_k}}} (1 - |x|^2) }{c_{n_k}}  =1.
\]
Clearly, $c_{n_k}/ \widehat{\kappa}(F_{n_k}) \to \infty$, thus we have 
\[
D^{+}(X) \ge  \frac{1}{2} \limsup_{k\to\infty}\frac{\sum_{x \in X \cap Q_{F_{n_k}}} (1 - |x|^2) }{\widehat{\kappa}(F_{n_k})} = \infty. 
\]
The proof of Proposition \ref{prop-unique-disk} is complete. 
\end{proof}

\subsection{Proofs of Theorems \ref{thm-rep-cb}, \ref{thm-intro-cm-no-ss} and \ref{thm-berg-cb} in dimension $d =1$ case}

We apply the results obtained in \S \ref{sec-gen} to the triple $(\D, d_\D, \mu_\D)$ consisting of the unit disk $\C$ equipped with the Lobachevsky-Poincar\'e metric $d_\D(\cdot, \cdot)$ given by \eqref{lob-dist} and the associated hyperbolic volume measure $\mu_\D$ given by \eqref{lob-vol}. We choose the origin $0\in \D$ as the base point.

\begin{lemma}\label{lem-disk-ent}
For any $z\in \D$, we have 
\[
\lim_{r\to\infty} \frac{\mu_{\D}(B_{d_\D}(z,r))}{ e^{r}} = \frac{1}{4},
\]
where $B_{d_\D}(z, r)$ denotes the Lobachevskian ball with centre $z$ and radius $r$.  In particular, the volume entropy of $\mu_{\D}$ is given by $h_\D = 1$.
\end{lemma}

\begin{proof}
For any $z\in \D$ and any $r>0$, we have
\begin{align*}
\mu_{\D} (B_{d_\D}(z, r))  = \mu_{\D} (B_{d_\D}(0, r))  = \int_{B_{d_\D}(0, r)} \frac{dA(w)}{(1 - |w|^2)^{2}}= 2 \int_{0}^{\frac{e^r - 1}{e^r +1}}  \frac{ tdt}{(1 - t^2)^{2}}.
\end{align*}
  By change of variables $t = \frac{e^x -1}{e^x+1}$, we obtain 
\[
\mu_{\D} (B_{d_\D}(z, r)) =  \frac{1}{4} \int_{0}^r   e^{- x} (e^x -1) (e^x +1) dx.
\]
By l'H\^opital principle, we have 
\[
\lim_{r\to\infty} \frac{\mu_{\D}(B_{d_\D}(z,r))}{ e^{r}} =  \lim_{r\to\infty} \frac{  \frac{1}{4} e^{-r} (e^r-1) (e^r+1)}{e^{r}} = \frac{1}{4}.
\]
\end{proof}

Recall that we denote by $\varphi_z \in \Aut(\D)$  exchanging $z$ and $0$. 

\begin{lemma}\label{lem-observation}
Let $\mathscr{H}$ be a Hilbert space over $\R$ or $\C$.  Let $u: \D\rightarrow \mathscr{H}$ be a harmonic function.  If  $\nu$ is a radial finite Radon measure on $\D$, then for any $z \in \D$ such that $u\circ\varphi_z \in L^1( \nu; \mathscr{H})$, we have 
\begin{align}\label{sph-mean}
u(z) = \frac{\displaystyle{\int_\D} u(\varphi_{z}(x)) d\nu(x)   }{\displaystyle{\int_\D}d\nu(x)  }.
\end{align}
In particular, if $W: \D\rightarrow \R_{+}$ is a nonzero radial  function, then for any $z\in \D$ such that $W(\varphi_{z}) \in L^1(\mu_\D)$ and $W(\varphi_{z}) \cdot u \in L^1(\mu_\D; \mathscr{H})$, we have 
\begin{align}\label{mean-value-formula}
u(z) = \frac{\displaystyle{\int_\D} W(\varphi_{z}(x)) u(x)  d\mu_\D (x) }{\displaystyle{\int_\D} W(\varphi_{z}(x))d \mu_\D(x) }.
\end{align}
 For any  $z \in \D$ and $r > 0$, we have 
\begin{align}\label{mvp-disk}
u(z) =\frac{1}{\mu_\D(B_{d_\D}(z, r))}\int_{B_{d_\D}(z, r)} u(x) d\mu_\D(x)
\end{align}
\end{lemma}

\begin{proof}
By the radial assumption on $\nu$ and the harmonicity of the function $w \mapsto u(\varphi_{z} (xw))$ on a neighbourhood of $\overline{\D}$ for any fixed $ z, x \in \D$, we have 
\begin{align*}
\int_\D u(\varphi_{z}(x)) d\nu(x)  & = \int_0^{2 \pi} \frac{d \theta}{2 \pi} \int_\D u(\varphi_{z}(x e^{i \theta})) d\nu(x)  = \int_\D \left(\int_0^{2 \pi}   u(\varphi_{z}(x e^{i \theta}))  \frac{d \theta}{2 \pi}  \right) d\nu(x) 
\\
& = u(\varphi_{z}(0)) \int_\D d\nu(x) = u(z) \int_\D d\nu(x).
\end{align*}
The equality \eqref{sph-mean} is proved. 

The hyperbolic volume measure $\mu_\D$ is M\"obius invariant and $\varphi_{z}$ is an involutive M\"obius transformation on $\D$, whence
\begin{align*}
\int_\D W(\varphi_{z}(z)) u(z)  d\mu_\D(z)  = \int_\D W(\zeta) u(\varphi_{z}(\zeta))  d\mu_\D(\zeta).
\end{align*}
The equality \eqref{mean-value-formula} now follows from the equality \eqref{sph-mean} since  $W$ is radial and consequently so is the measure $W(z) \cdot d\mu_\D(z)$. 

The equality \eqref{mvp-disk} follows from the equality \eqref{mean-value-formula} by taking $W(z) = \mathds{1}(|z| < \frac{e^r -1}{e^r + 1})$. 
\end{proof}

Recall the definitions of  the spaces   $\widetilde{\mathrm{MVP}}(\mu_M; \mathscr{H})$ and   $\mathrm{MVP}(\mu_M; \mathscr{H})$ introduced in \S \ref{sec-general-Hilbert}.  Let $b^2(\D)$ denote the harmonic Bergman space on $\D$.

\begin{lemma}\label{lem-berg-mvp}
We have the inclusions $A^2(\D) \subset b^2(\D) \subset \mathrm{MVP}(\mu_\D; \C)$. 
\end{lemma}

\begin{proof}
It suffices to show the last inclusion $b^2(\D) \subset \mathrm{MVP}(\mu_\D; \C)$. Suppose that $f\in b^2(\D)$, then for any $k\in \N$, there exists $c>0$, such that 
\begin{multline*}
\int_{A_k(0)}| f(x)|^2 d\mu_\D(x)   =  \int_{A_k(0)} |f (x)|^2 \frac{dA(x)}{ ( 1 - |x|^2)^2}  \le
\\
\le \sup_{x\in A_k(0)} \frac{1}{(1 - |x|^2)^2} \int_\D |f(x)|^2 dA(x)   
 \le c \| f\|_{b^2(\D)}^2 e^{2 d_\D(x, 0)}.
\end{multline*}
By Lemmata \ref{lem-disk-ent}, \ref{lem-observation}  and Proposition \ref{prop-mgc}, we have $f \in \mathrm{MVP}(\mu_\D; \C)$. 
\end{proof}

\begin{proof}[Proof of Proposition \ref{prop-intro-single-bergman}]
Proposition \ref{prop-intro-single-bergman} follows from Lemma \ref{lem-berg-mvp}, Theorem  \ref{thm-H-L} and the observation that  the condition $s_n>  1$ and  $\sum_{n=1}^\infty (s_n-1)^2 <\infty$ implies the condition \eqref{fast-to-cri} in this case.  
\end{proof}

We proceed to the proof of Theorem \ref{thm-rep-cb} in dimension $d =1$ case. Recall that in this case, the $\MM$-harmonicity and the usual harmonicity coincides  and we denote $\MM\HH(\D)$ by $\mathrm{Harm}(\D)$. 

\begin{lemma}\label{lem-mvp-rep}
Let $\mathscr{H}(K) \subset \mathrm{Harm}(\D)$ be a reproducing kernel Hilbert space with reproducing kernel $K$ satisfying Assumption \ref{ass-gcb}. Let $f: \D \rightarrow \mathscr{H}(K)$ be the map  
\[
\D \ni x \mapsto f(x) := K_x   = K(\cdot, x)\in \mathscr{H}(K).
\]  
 Then $f \in \mathrm{MVP}(\mu_\D; \mathscr{H}(K))$. 
\end{lemma}

\begin{proof}
First of all, for any $g \in \mathscr{H}(K) \subset \mathrm{Harm}(\D)$, the function $\D\ni x \mapsto \langle g, f(x)\rangle_{\mathscr{H}(K)} = g(x)$ is harmonic, therefore the function $f : \D \rightarrow \mathscr{H}(K)$ is harmonic and  the equality  \eqref{mvp-disk} in Lemma \ref{lem-observation} holds for $f$. That is, $f \in \widetilde{\mathrm{MVP}}(\mu_\D; \mathscr{H}(K))$.

By the reproducing property, we have $\| f(x)\|_{\mathscr{H}(K)}^2 = \langle K_x, K_x\rangle_{\mathscr{H}(K)}  = K_x(x) = K(x, x)$. Recall that $d_\D(0, x) = \log \Big(\frac{ 1 + |x|}{1-|x|}\Big)$. The growth assumption \eqref{rough-es-cb} on $K$ implies that the function $f(x) = K_x$ satisfies the condition \eqref{sub-exp-correction}.

Note that there exist $c, C >0$, such that for any $s\in (1, 2)$ and any $z\in \D$, we have
\begin{align}\label{cal-disk-ent}
\frac{c}{s-1} \le \int_{\D} e^{-s d_\D(x, z)} d\mu_{\D} (x) \le \frac{C}{s-1}.
\end{align}
Indeed, by M\"obius invariance, sending $z$ to $0$ gives 
\begin{align}\label{cal-disk-ent-1}
\int_{\D} e^{-s d_\D(x, z)} d\mu_{\D} (x)    = \int_{\D} e^{-s d_\D(x, 0)} d\mu_{\D} (x)   = \int_{\D} \left( \frac{1 - |x|}{1 + |x|} \right)^s  \frac{dA(x)}{( 1 - |x|^2)^{2}},
\end{align}
and integrating \eqref{cal-disk-ent-1} in polar coordinates  gives \eqref{cal-disk-ent}. Therefore, the assumption \eqref{fine-es-cb} on $K$ implies that the function $f(x) = K_x$ satisfies the condition \eqref{global-growth-f}.  This completes the proof of $f \in \mathrm{MVP}(\mu_\D; \mathscr{H}(K))$. 
\end{proof}

{\flushleft \it Proof of Theorem \ref{thm-rep-cb} in dimension $d=1$ case.} 
Since $\mathcal{D}\subset \D$ is countable, it suffices to prove the same statements of the theorem for a fixed point $z\in \D$.  Now Lemma \ref{lem-disk-ent} implies that the triple $(\D, d_\D, \mu_\D)$ satisfies Assumption \ref{ass-M-pvg} and hence Assumption \ref{ass-M-A1}.  By \eqref{cal-disk-ent}, the assumption \eqref{fast-s-to-d} on the sequence $(s_n)_{n\ge 1}$   in $(1, \infty)$ converging to $1$  implies that $(s_n)_{n\ge 1}$ satisfies the condition \eqref{fast-to-cri}.  Therefore,  by Lemma \ref{lem-mvp-rep},  the  statements in item (1) and item (2) of Theorem \ref{thm-rep-cb} in dimension $d=1$ case follow from Theorem \ref{thm-H-L}. Finally, the reproducing property of $\mathscr{H}(K)$ implies that the statement in item (3) of Theorem \ref{thm-rep-cb} in dimension $d=1$ case is a consequence of the statement in item (2) of Theorem \ref{thm-rep-cb} in dimension $d=1$ case. \qed

\medskip

Now we proceed to the proof of Theorem \ref{thm-intro-cm-no-ss} in dimension $d = 1$ case. In this case, the Poisson kernel $P: \D \times \T \rightarrow \R^{+}$  is given by \eqref{def-Poi-kernel} and the Poisson transformation of signed Borel measure on $\T$ is  defined by \eqref{poi-trans}.  For $z\in \D$,  introduce a function $P_z:\T\to \mathbb R$ by setting $P_z(\zeta) = P(z, \zeta)$.

We first prove the following Lemma \ref{lem-cmvp}. Let  $\mu$ be a fixed Borel probability measure on $\T$ and let $L^2(\mu; \R)$ be the real-Hilbert space of $\R$-valued $\mu$-square-integrable functions on $\T$. Then $L^2(\mu; \R)$ has a natural partially ordered vector space structure and the norm is compatible in the sense of Definition \ref{def-com} with this natural partial order.   Note that for any $x\in \D$, we have $P_x \in L^2(\mu; \R)_{+}$. Recall Definition \ref{def-cmvp-sharp} for the space $\mathrm{CMVP}^\sharp(\mu_\D; L^2(\mu; \R)_{+})$ used in the following lemma.

\begin{lemma}\label{lem-cmvp}
Let $\mu$ be a Borel probability measure on $\T$. Then the map 
\[
\D \ni x \mapsto f(x) : = P_x \in L^2(\mu; \R)_{+}
\]
belongs to the class $\mathrm{CMVP}^\sharp(\mu_\D; L^2(\mu; \R)_{+})$. 
\end{lemma}

\begin{proof}
 For any $\zeta \in \T$, the function $\D\ni x \mapsto  P_x(\zeta) = P(x, \zeta)$ is harmonic, hence so is  the vector valued function $\D \ni x \mapsto f(x)= P_x$. Therefore, $f$ is continuous and  by the equality \eqref{mvp-disk}, we have $f \in \widetilde{\mathrm{MVP}}(\mu_\D; L^2(\mu; \R))$. 
Proposition 1.4.10 in \cite{Rudin-ball} implies there exists $c>0$ such that for any $x\in \D$, we have
\begin{align}\label{l-2-poi}
\int_{0}^{2 \pi}   \left(\frac{1- |x|^2}{ | 1 - \bar{x} e^{i \theta}|^2}\right)^2 \frac{ d \theta}{2\pi} \le  \frac{c}{1 - |x|^2} \le ce^{d_\D(x, 0)}. 
\end{align}
Therefore, for any positive integer $k \in \N$,  using the rotational invariance of the measure $\mu_\D$ and recall that $A_k(0) = \{z\in\D: k \le d_\D(z, 0) < k + 1\}$,   we have 
\begin{multline*}
\int_{A_k(0)} \| f(x) \|_{L^2(\mu)}^2 d\mu_\D(x) = \int_0^{2\pi} \Big[\int_{A_k(0)} \| f(x e^{i \theta}) \|_{L^2(\mu)}^2 d\mu_\D(x)\Big]\frac{ d \theta}{2 \pi} = 
\\ 
= \int_0^{2\pi} \Big[\int_{A_k(0)} \int_{\T} \left(\frac{1 - |x|^2}{| 1 - \bar{x} e^{-i \theta} e^{i \theta'}|^2 } \right)^2  d \mu(\theta') d\mu_\D(x)\Big]\frac{ d \theta}{2 \pi}  = 
\\
= \int_{A_k(0)} \int_{\T}  \Big[ \int_0^{2\pi} \left(\frac{1 - |x|^2}{| 1 - \bar{x} e^{-i \theta} e^{i \theta'}|^2 } \right)^2 \frac{ d \theta}{2 \pi} \Big]  d \mu(\theta') d\mu_\D(x)\le 
\\
\le \int_{A_k(0)} \int_{\T}  c e^{d_\D(x, 0)} d \mu(\theta') d\mu_\D(x) \le c e^{k+1} \mu_\D(B_{d_\D}(0, k+1)). 
\end{multline*}
By Lemma \ref{lem-disk-ent}, there exists $c'>0$ such that $\mu_{\D} (B_{d_\D}(0, k+1)) \le c' e^k$. Consequently, there exists a constant $C>0$, such that for any $k\in \N$, we have 
\begin{align}\label{mg-poi}
\int_{A_k(0)} \| f(x) \|_{L^2(\mu)}^2 d\mu_\D(x) \le  C e^{2k}. 
\end{align}
Since the triple $(\D, d_\D, \mu_\D)$ satisfies Assumption \ref{ass-M-pvg}, by Proposition \ref{prop-mgc}, the mean-growth estimate \eqref{mg-poi} implies that $f\in \mathrm{MVP}(\mu_\D; L^2(\mu; \R))$. Combining with the fact that  $f$ belongs to the class $C(\D; L^2(\mu; \R)_{+})$, we obtain that  $f\in \mathrm{CMVP}(\mu_\D; L^2(\mu; \R)_{+})$. 

Now note that for $\varepsilon \in (0, 1)$, there exists a constant $C'>0$ such that $1 - e^{- 2 \varepsilon} \ge C' \varepsilon$. Therefore, using the estimate \eqref{mg-poi}, we obtain, for any $\varepsilon\in (0, 1)$, that  
\begin{multline}\label{MVP-sharp-P}
\int_\D e^{-2 \varepsilon d_\D(x, 0)} \| f(x) \|_{L^2(\mu)}^2 e^{-2 d_\D(x, 0)} d\mu_\D(x) \le  \sum_{k = 0}^\infty e^{-(2 + 2 \varepsilon) k} \int_{A_k(0)} \| f(x)\|_{L^2(\mu)}^2 d\mu_\D(x)
\\
\le C \sum_{k=0}^\infty e^{-2 \varepsilon k }  = \frac{C}{1 - e^{-2 \varepsilon}} \le \frac{C}{C' \varepsilon}. 
\end{multline}
That is,  the function $f$ satisfies the estimate \eqref{bdd-e-a} with $\beta = 0$ and $\alpha = 1$ and thus  we can conclude that $f \in \mathrm{CMVP}^\sharp (\mu_\D; L^2(\mu; \R)_{+})$. 
\end{proof}

{\flushleft \it Proof of Theorem \ref{thm-intro-cm-no-ss} in dimension $d = 1$ case.}
Lemma \ref{lem-disk-ent} implies that  the triple $(\D, d_\D, \mu_\D)$ satisfies Assumption \ref{ass-M-pvg}. Lemma \ref{lem-cmvp} implies that the function $\D \ni x \mapsto f(x) = P_x \in L^2(\mu; \R)_{+}$ belongs to the class  $\mathrm{CMVP}^\sharp(\mu_\D; L^2(\mu; \R)_{+})$. Therefore,  the statements in item (1) and item (2) of Theorem \ref{thm-intro-cm-no-ss} in dimension $d = 1$ case follows from  Theorem \ref{thm-n-sub}. Finally,  using the definition of $h^2(\D; \mu)$ and the equivalence between the unconditional convergence and the absolute convergence for scalar series, we derive the statement in  item (3) of Theorem \ref{thm-intro-cm-no-ss} in dimension $d = 1$ case from the statement in item (2) of Theorem \ref{thm-intro-cm-no-ss} in this case. \qed

\subsection{Proof of Theorem \ref{thm-berg-cb} in dimension $d = 1$ case}
We now prove Theorem \ref{thm-berg-cb} in dimension $d = 1$ case uisng modified arguments in the proof of Theorem \ref{thm-H-L}. 

Since $\mathcal{D}\subset \D$ is countable, it suffices to prove the same statements of the theorem for a fixed point $z\in \D$. Let $\alpha, \beta, \Pi, (s_n)_{n\ge 1}$ be as in Theorem \ref{thm-berg-cb} and fix $z\in \D$. Without loss of generality, we may assume that the first intensity measure of $\Pi$ is $( 1- |x|^2)^{- \beta -1} dA(x)$.

To any $s> \beta, k \in \N_0  = \N \cup \{0\}, X\in \Conf(\D)$, we assign 
\begin{align*}
T_k^{\alpha, \beta} (z, s; X): & = \sum_{x \in X \atop k\le d_\D (z,x) < k+1} e^{-s d_\D(z, x)}  \left(\frac{|1 - x \bar{z}|^2}{1 - |z|^2}\right)^{\beta-1} K_x^\alpha;
\\
\overline{g_{\Pi, k}}(z, s): = & \E_\Pi \left( \sum_{x \in \X \atop k \le d_\D(z, x) < k + 1} e^{-s d_\D(z,x)}  \left(\frac{|1 - x \bar{z}|^2}{1 - |z|^2}\right)^{\beta-1}  \right);
\\
\overline{g_\Pi}(z, s): = & \E_\Pi \left( \sum_{x \in \X} e^{-s d_\D(z,x)}  \left(\frac{|1 - x \bar{z}|^2}{1 - |z|^2}\right)^{\beta-1}  \right).
\end{align*}

\begin{lemma}\label{lem-beta-av}
For any $s > \beta, k \in \N_0$, we have 
\[
\E_\Pi [T_k^{\alpha, \beta}(z, s; \X)] =   \overline{g_{\Pi, k}}(z, s) \cdot K_z^\alpha. 
\]
\end{lemma}

\begin{proof}
The M\"obius  involution $\varphi_z(x) = (z - x) ( 1 - \bar{z}x)^{-1}$ satisfies 
\begin{align}\label{M-identity}
 1 - |\varphi_z(x)|^2 = \frac{ (1 - |x|^2) (1 - |z|^2)}{ | 1 - x \bar{z}|^2}. 
\end{align}
Set $A_k(z) = \{x \in \D : k \le d_\D(x, z) < k +1\}$. Using the conformal invariance of the measure $d \mu_\D(z)$,  we have 
\begin{multline*}
\E_\Pi [T_k^{\alpha, \beta}(z, s; \X)]  =  \int_{A_k(z) }   e^{-s d_\D(z, x)}  \left(\frac{|1 - x \bar{z}|^2}{1 - |z|^2}\right)^{\beta-1} K_x^\alpha   \frac{dA(x) }{ ( 1 - |x|^2)^{\beta + 1}}  = 
\\
= \int_{A_k(z) }   e^{-s d_\D(z, x)}  \left(\frac{1}{1 - |\varphi_z(x)|^2}\right)^{\beta-1} K_x^\alpha  d\mu_\D(x)  = 
\\
 = \int_{A_k(0)} e^{-s d_\D(0, x)} \left( \frac{1}{1 - |x|^2}\right)^{\beta -1} K_{\varphi_z(x)}^\alpha d\mu_\D(x). 
\end{multline*}
It is clear that the map $\D \ni x \mapsto K_x^{\alpha} \in A^2_\alpha(\D)$ is harmonic, thus by \eqref{sph-mean}, we obtain 
\[
\E_\Pi [T_k^{\alpha, \beta}(z, s; \X)]  
 =   K_{z}^\alpha  \int_{A_k(0)} e^{-s d_\D(0, x)} \left( \frac{1}{1 - |x|^2}\right)^{\beta -1} d\mu_\D(x). 
\]
Similar argument (by replacing $K_x^\alpha$ by the constant function in the previous calculation) shows that 
\begin{align}\label{g-k-s}
\overline{g_{\Pi, k}}(z, s) =  \int_{A_k(0)} e^{-s d_\D(0, x)} \left( \frac{1}{1 - |x|^2}\right)^{\beta -1} d\mu_\D(x)
\end{align}
and thus lemma is proved completely. 
\end{proof}

\begin{lemma}\label{lem-g-s}
For any $s> \beta$, we have 
\begin{align}\label{lem-g-s-goal}
\frac{2^{-s}}{s - \beta}  \le \overline{g_\Pi}(z, s) \le \frac{1}{s - \beta}. 
\end{align}
\end{lemma}

\begin{proof}
Similar to \eqref{g-k-s}, we have 
\[
\overline{g_\Pi}(z, s) =  \int_{\D} e^{-s d_\D(0, x)} \left( \frac{1}{1 - |x|^2}\right)^{\beta -1} d\mu_\D(x)   
= \int_\D \left(\frac{1 - |x|}{1 + |x|}\right)^s \frac{dA(x) }{( 1 - |x|^2)^{\beta +1}}.
\]
The inequalities  \eqref{lem-g-s-goal} follows from $ (1 - |x|^2)^s /2^s \le \left(\frac{1 - |x|}{1 + |x|}\right)^s \le ( 1 - |x|^2)^s$.  
\end{proof}

 For proving Theorem \ref{thm-berg-cb} in dimension $d = 1$ case, it suffices to prove that the following two claims.

{\flushleft \bf Claim I:} for $s> \beta$, we have 
\[
\sum_{k=0}^\infty  \left\{ \E_\Pi  \left( \Big\| T_k^{\alpha, \beta} (z, s; \X) \Big\|_{A^2_\alpha(\D)}^2\right) \right\}^{1/2} < \infty.
\]

{\flushleft \bf Claim II:} there exists $C> 0$ such that for any $s \in ( \beta, \beta+1)$, we have  
\[
\E_\Pi \left( \left\| \frac{\sum_{k=0}^\infty  T_k^{\alpha, \beta}(z, s; \X) }{\overline{g_\Pi}(z, s)}  - K_z^\alpha  \right\|_{A^2_\alpha(\D)}^2 \right)    \le C (s -\beta). 
\]

For proving Claim I, note that since $\Pi$ satisfies Assumption \ref{ass-pp}, by Proposition \ref{rem-scalar-vector} and Lemma \ref{lem-beta-av},  we have 
\begin{multline*}
\E_\Pi \left( \Big\| T_k^{\alpha, \beta} (z, s; \X) \Big\|_{A^2_\alpha(\D)}^2\right)   =  \Var_\Pi \left( T_k^{\alpha, \beta} (z, s; \X) \right) + \left\|   \E_\Pi [ T_k^{\alpha, \beta} (z, s; \X) ] \right\|_{A^2_\alpha(\D)}^2 \le 
\\
\le C  \int_{A_k(z)}   e^{-2 s d_\D(z, x)}  \left(\frac{|1 - x \bar{z}|^2}{1 - |z|^2}\right)^{2\beta-2}  \| K_x^\alpha\|_{A^2_\alpha(\D)}^2 \frac{dA(x) }{( 1 - |x|^2)^{\beta + 1}}  +  \overline{g_{\Pi, k}}(z, s)^2 \cdot \| K_z^\alpha\|_{A^2_\alpha(\D)}^2  = 
\\
 = C  \int_{A_k(0)}   e^{-2 s d_\D(0, x)}  \frac{( 1- | \varphi_z(x)|^2)^{\beta-1}}{( 1 - |x|^2)^{2\beta }}  \| K_{\varphi_z(x)}^\alpha\|_{A^2_\alpha(\D)}^2 dA(x)   +  \overline{g_{\Pi, k}}(z, s)^2 \cdot \| K_z^\alpha\|_{A^2_\alpha(\D)}^2 \le 
\\
\le  C_z \int_{A_k(0)}   ( 1- |x|^2)^{2s - \alpha - \beta -3} dA(x)   + \overline{g_{\Pi, k}}(z, s)^2 \cdot \| K_z^\alpha\|_{A^2_\alpha(\D)}^2. 
\end{multline*}
Therefore, 
\begin{multline*}
\E_\Pi \left( \Big\| T_k^{\alpha, \beta} (z, s; \X) \Big\|_{A^2_\alpha(\D)}^2\right)  \le 2 \pi C_z \int_{\frac{e^k -1}{e^k + 1}}^{\frac{e^{k+1}-1}{ e^{k+1} + 1}}  ( 1 - r^2)^{2s - \alpha - \beta -3} rdr  + \overline{g_{\Pi, k}}(z, s)^2 \cdot \| K_z^\alpha\|_{A^2_\alpha(\D)}^2\le 
\\
\le   \pi C_z  \max_{ \frac{e^{k}-1}{ e^{k} + 1} \le r \le \frac{e^{k+1}-1}{ e^{k+1} + 1} }   ( 1 - r^2)^{2s - \alpha - \beta -3} \cdot \left( \left(\frac{e^{k+1}-1}{ e^{k+1} + 1}\right)^2 - \left( \frac{e^{k}-1}{ e^{k} + 1}\right)^2 \right) + \overline{g_{\Pi, k}}(z, s)^2 \cdot \| K_z^\alpha\|_{A^2_\alpha(\D)}^2\le 
\\
\le  C_{s, z}    e^{-(2s - \alpha - \beta - 2)k }  + \overline{g_{\Pi, k}}(z, s)^2 \cdot \| K_z^\alpha\|_{A^2_\alpha(\D)}^2
\end{multline*}
and thus 
\[
 \left\{ \E_\Pi  \left( \Big\| T_k^{\alpha, \beta} (z, s; \X) \Big\|_{A^2_\alpha(\D)}^2\right) \right\}^{1/2} \le  \sqrt{C_{s, z}} e^{- (2s - \alpha - \beta - 2)k/2} +  \overline{g_{\Pi, k}}(z, s) \cdot \| K_z^\alpha\|_{A^2_\alpha(\D)}.
\]
If $s> \beta$, then $2s - \alpha - \beta - 2 > \beta - \alpha -2 \ge 0$, hence $\sum_{k =0}^\infty e^{-(2s - \alpha - \beta -2) k/2}<\infty$. Lemma \ref{lem-g-s} implies that for any $s> \beta$, we have $\sum_{k  =0}^\infty \overline{g_{\Pi, k}} (z, s) <\infty$. Therefore, we complete the proof of Claim I. 

Now we turn to the proof of Claim II. Note that Claim I implies in particular  that the series $\sum_{k=0}^\infty  T_k^{\alpha, \beta}(z, s; \X)$ converges in $L^2( \Conf(\D), \Pi; A^2_\alpha(\D))$.   By Lemmata \ref{lem-beta-av} and \ref{lem-g-s}, for proving Claim II, it suffices to prove that for any $s \in (\beta, \beta+1)$, we have 
\[
\E_\Pi \left( \left\|  \sum_{k=0}^\infty  T_k^{\alpha, \beta}(z, s; \X)  -  \overline{g_\Pi}(z, s) \cdot  K_z^\alpha  \right\|_{A^2_\alpha(\D)}^2 \right)    = \Var_\Pi \left(   \sum_{k=0}^\infty  T_k^{\alpha, \beta}(z, s; \X)\right)   \le \frac{C}{s- \beta}.  
\] 
By Proposition \ref{rem-scalar-vector}, for $s \in (\beta, \beta +1)$,  we have 
\begin{multline*}
\Var_\Pi \left(   \sum_{k=0}^\infty  T_k^{\alpha, \beta}(z, s; \X)\right)   = \lim_{N\to\infty} \Var_\Pi \left(   \sum_{k=0}^N  T_k^{\alpha, \beta}(z, s; \X)\right)  \le
\\
\le \lim_{N\to\infty } C \E_\Pi \left( \sum_{x \in \X \atop d_\D (z,x) < N+1} e^{- 2 s d_\D(z, x)}  \left(\frac{|1 - x \bar{z}|^2}{1 - |z|^2}\right)^{2\beta-2}  \| K_x^\alpha\|_{A^2_\alpha(\D)}^2\right) = 
\\
 = C  \int_\D e^{- 2 s d_\D(z, x)}  \left(\frac{|1 - x \bar{z}|^2}{1 - |z|^2}\right)^{2\beta-2}   \frac{\alpha + 1}{\pi ( 1  - |x|^2)^{2 + \alpha}} \frac{dA(x) }{( 1- |x|^2)^{\beta + 1}} \le 
\\
 \le  C_z \int_\D  ( 1 - |x|^2)^{2s - \alpha - \beta -3} dA(x) = \frac{C_z}{2s - \alpha - \beta -2} \le \frac{C_z}{s - \beta},
\end{multline*}
where  we used $2s - \alpha - \beta - 2   = s- \beta + (s - \alpha -2)\ge s- \beta + (\beta - \alpha -2) \ge s-\beta$.
This completes the proof of Claim II and thus completes the proof of Theorem \ref{thm-berg-cb} in dimension $d = 1$ case.

\subsection{Proof of Theorems \ref{thm-rep-cb}, \ref{thm-intro-cm-no-ss} and \ref{thm-berg-cb} for arbitrary dimension $d \ge 2$}

In this section, we will assume that the dimension $d \ge 2$. 

\begin{lemma}\label{lem-cb-ent}
There exists a constant $c_d>0$ such that for any $z\in \D_d$, we have
\begin{align}\label{pvg-comp-ball}
\lim_{r\to\infty} \frac{\mu_{\D_d}(B(z,r))}{ e^{rd}}  = c_d.
\end{align}
In particular, the volume entropy of $\mu_{\D_d}$ is given by  $h_{\D_d} = d$.
\end{lemma}
\begin{proof}
The proof is similar to that of Lemma \ref{lem-disk-ent}. 
\end{proof}

Recall the definition of the $\MM$-harmonic functions on $\D_d$  in \S \ref{sec-intro-chs}.   By Rudin \cite[Corollary 2 of Theorem 4.2.4]{Rudin-ball}, a continuous function $u\in C(\D_d)$ is $\MM$-harmonic if and only if it satisfies the invariant mean-value property:  for any $\psi \in \Aut(\D_d)$ and any $0 < t< 1$, we have
\begin{align}\label{inv-mvp}
u(\psi(0)) = \int_{\Sph_d}  u(\psi(t \zeta)) d\sigma_{\Sph_d}(\zeta).
\end{align}

\begin{lemma}\label{lem-cb-mvp}
Let $\mathscr{H}$ be a Hilbert space over $\R$ or $\C$.  Let $u: \D_d\rightarrow \mathscr{H}$ be an $\MM$-harmonic function. Then for any  $z \in \D_d$ and $r > 0$, we have 
\begin{align}\label{mvp-cb}
u(z) =\frac{1}{\mu_{\D_d}(B(z, r))}\int_{B(z, r)} u(x) d\mu_{\D_d}(x).
\end{align}
\end{lemma}

\begin{proof}
It suffices to consider the case $\mathscr{H} = \R$. Fix $z \in \D_d, r >0$ and an $\MM$-harmonic function $u: \D_d \rightarrow \R$. Since $\mu_{\D_d}$ is $\Aut(\D_{d})$-invariant, using the equality \eqref{inv-mvp} and the formula of integration in polar coordinates, we have  
\begin{multline}\label{f-ball-i}
\begin{split}
\int_{B(z, r)} u(x) d\mu_{\D_d}(x)  = \int_{B(0, r)} u(\varphi_z(x)) d\mu_{\D_d}(x) = \int_{B(0, r)} u(\varphi_z(x)) \frac{dv_d(x)}{(1 - |x|^2)^{d + 1}} =
\\
=    c_d  \int_{0}^{\frac{e^r - 1}{e^r + 1}} \frac{t^{2d-1}}{(1- t^2)^{d+1}} dt \int_{\Sph_d}  u(\varphi_z(t\zeta)) d\sigma_{\Sph_d}(\zeta) =  c_d  \int_{0}^{\frac{e^r - 1}{e^r + 1}} \frac{t^{2d-1}}{(1- t^2)^{d+1}} dt \cdot  u(z),
\end{split}
\end{multline}
where $c_d>0$ is a constant depending on $d$ and $\varphi_z$ is defined by the formula \eqref{inv-auto}. By taking $u \equiv 1$, we also have 
\begin{align}\label{f-vol}
\mu_{\D_d}(B(z, r)) = \int_{B(z, r)} d\mu_{\D_d}(x) =  c_d  \int_{0}^{\frac{e^r - 1}{e^r + 1}} \frac{t^{2d-1}}{(1- t^2)^{d+1}} dt.
\end{align}
The equalities \eqref{f-ball-i} and \eqref{f-vol} imply the desired equality \eqref{mvp-cb}. 
\end{proof}

Recall the definition of  the space   $\mathrm{MVP}(\mu_M; \mathscr{H})$ introduced in \S \ref{sec-general-Hilbert}. 
\begin{lemma}\label{lem-mvp-ball}
Let $\mathscr{H}(K) \subset \MM\HH(\D_d)$ be a reproducing kernel Hilbert space with reproducing kernel $K$ satisfying Assumption \ref{ass-gcb}. Let $f: \D_d \rightarrow \mathscr{H}(K)$ be the map  
\[
\D_d \ni x \mapsto f(x) := \widehat{K}_x = K(\cdot, x) \in \mathscr{H}(K).
\]  
 Then $f \in \mathrm{MVP}(\mu_{\D_d}; \mathscr{H}(K))$. 
\end{lemma}

\begin{proof}
The proof is similar to that of Lemma \ref{lem-mvp-rep}, the role played by the equality \eqref{mvp-disk} for harmonic functions on $\D$ is now played by the equality \eqref{mvp-cb} for $\MM$-harmonic functions on $\D_d$. 
\end{proof}

\begin{proof}[Proof of Theorem \ref{thm-rep-cb}]
Lemma \ref{lem-cb-ent} implies that the triple $(\D_d, d_B, \mu_{\D_d})$ satisfies Assumption \ref{ass-M-pvg} and hence Assumption \ref{ass-M-A1}. It is easy to see that \eqref{fine-es-cb} is equivalent to \eqref{global-growth-f} in this case. Therefore,  by Lemma \ref{lem-mvp-ball}, Theorem \ref{thm-rep-cb} follows from Theorem \ref{thm-H-L}. 
\end{proof}

Now we proceed to the proof of Theorem \ref{thm-intro-cm-no-ss}. 

\begin{lemma}\label{lem-cmvp-cb}
Let $\mu$ be a Borel probability measure on $\Sph_d$. Then the map 
\[
\D_d \ni x \mapsto f(x) : = P_x^b \in L^2(\mu; \R)_{+}
\]
belongs to the class $\mathrm{CMVP}^\sharp(\mu_{\D_d}; L^2(\mu; \R)_{+})$. 
\end{lemma}

\begin{proof}
 For any $\zeta \in \Sph_d$, the function $\D_d \ni x \mapsto  P_x^b (\zeta) = P^b(x, \zeta)$ is $\MM$-harmonic, hence so is  the vector valued function $\D_d \ni x \mapsto f(x)= P_x^b$. By the equality \eqref{mvp-cb}, we have $f \in \widetilde{\mathrm{MVP}}(\mu_{\D_d}; L^2(\mu; \R))$. 
Proposition 1.4.10 in \cite{Rudin-ball} implies there exists $c>0$ such that for any $x\in \D_d$, we have
\begin{align}\label{l-2-poi-cb}
\int_{\Sph_d}   P^b(x, \zeta)^2  d\sigma_{\Sph_d}(\zeta)              = \int_{\Sph_d} \left(  \frac{(1-|x|^2)^d}{|1 - \zeta\cdot \bar{x}|^{2d}} \right)^2  d\sigma_{\Sph_d}(\zeta) \le ce^{d \cdot d_B(x, 0)}. 
\end{align}
Therefore, for any positive integer $k \in \N$,  using the rotational invariance of the measure $\mu_{\D_d}$ and recall that $A_k^{(d)}(0) = \{z\in\D_d: k \le d_B(z, 0) < k + 1\}$,  we have 
\begin{multline*}
\int_{A_k^{(d)}(0)} \| f(x) \|_{L^2(\mu)}^2 d\mu_{\D_d}(x) 
=  \int_{A_k^{(d)}(0)} \int_{\Sph_d}   \left(  \frac{(1-|x|^2)^d}{|1 - \zeta\cdot \bar{x}|^{2d}} \right)^2   d \mu(\zeta) d\mu_{\D_d}(x) = 
\\
= \int_{A_k^{(d)}(0)} \int_{\Sph_d}  \Big[  \int_{\mathcal{U}_d}\left(  \frac{(1-|U x|^2)^d}{|1 - \zeta\cdot \overline{U x}|^{2d}} \right)^2  dm_{\mathcal{U}_d} (U) \Big]  d \mu(\zeta) d\mu_\D(x),
\end{multline*}
where $\mathcal{U}_d$ is the group of $d\times d$ unitary matrices and $m_{\mathcal{U}_d}$ is the normalized Haar measure on it.  Since $\sigma_{\Sph_d}$ coincides with the orbital measure under the transitive action of $\mathcal{U}_d$, for any $\zeta \in \Sph_d$, we have 
\begin{multline*}
\int_{\mathcal{U}_d}\left(  \frac{(1-|U x|^2)^d}{|1 - \zeta\cdot \overline{U x}|^{2d}} \right)^2  dm_{\mathcal{U}_d} (U) =  \int_{\mathcal{U}_d}\left(  \frac{(1-|x|^2)^d}{|1 - (U^{-1}\zeta)\cdot \bar{x}|^{2d}} \right)^2  dm_{\mathcal{U}_d} (U)  =
\\
=  \int_{\Sph_d} \left(  \frac{(1-|x|^2)^d}{|1 - \zeta'\cdot \bar{x}|^{2d}} \right)^2  d\sigma_{\Sph_d}(\zeta') \le  ce^{d \cdot d_B(x, 0)}.
\end{multline*}
Thus we obtain 
\begin{multline*}
\int_{A_k^{(d)}(0)} \| f(x) \|_{L^2(\mu)}^2 d\mu_{\D_d}(x) \le \int_{A_k^{(d)}(0)} \int_{\Sph_d}   ce^{d \cdot d_B(x, 0)} d \mu(\zeta) d\mu_{\D_d}(x) = 
\\
=\int_{A_k^{(d)}(0)}     ce^{d \cdot d_B(x, 0)}  d\mu_{\D_d}(x)  \le c e^{(k+1)d} \mu_{\D_d}(B(0, k+1)).
\end{multline*}
By Lemma \ref{lem-cb-ent}, there exists $c'>0$ such that $\mu_{{\D_d}} (B(0, k+1)) \le c' e^{kd}$. Consequently, there exists a constant $C>0$, such that for any $k\in \N$, we have 
\begin{align}\label{mg-poi-cb}
\int_{A_k^{(d)}(0)} \| f(x) \|_{L^2(\mu)}^2 d\mu_{\D_d}(x) \le  C e^{2kd}. 
\end{align}
Since the triple $(\D_d, d_B, \mu_{\D_d})$ satisfies Assumption \ref{ass-M-pvg}, by Proposition \ref{prop-mgc}, the mean-growth estimate \eqref{mg-poi-cb} implies that $f\in \mathrm{MVP}(\mu_{\D_d}; L^2(\mu; \R))$. Clearly,  $f$ is continuous and non-negative, whence $f\in \mathrm{CMVP}(\mu_{\D_d}; L^2(\mu; \R)_{+})$. 

Now note that for $\varepsilon \in (0, 1)$, there exists a constant $C'>0$ such that $1 - e^{- 2 \varepsilon} \ge C' \varepsilon$. Therefore, using the estimate \eqref{mg-poi-cb}, we obtain, for any $\varepsilon\in (0, 1)$, that  
\begin{multline*}
\int_{\D_d} e^{-2 \varepsilon d_B(x, 0)} \| f(x) \|_{L^2(\mu)}^2 e^{-2 d\cdot d_B(x, 0)} d\mu_{\D_d}(x) \le  \sum_{k = 0}^\infty e^{-(2d + 2 \varepsilon) k} \int_{A_k^{(d)}(0)} \| f(x)\|_{L^2(\mu)}^2 d\mu_{\D_d}(x) \le 
\\
\le C \sum_{k=0}^\infty e^{-2 \varepsilon k }  = \frac{C}{1 - e^{-2 \varepsilon}} \le \frac{C}{C' \varepsilon}. 
\end{multline*}
That is,  the function $f$ satisfies the estimate \eqref{bdd-e-a} with $\beta = 0$ and $\alpha = 1$ and thus  we can conclude that $f \in \mathrm{CMVP}^\sharp (\mu_{\D_d}; L^2(\mu; \R)_{+})$. 
\end{proof}

\begin{proof}[Proof of Theorem \ref{thm-intro-cm-no-ss}]
By Lemmata \ref{lem-cb-ent} and \ref{lem-cmvp-cb}, Theorem \ref{thm-intro-cm-no-ss} follows immediately from Theorem \ref{thm-n-sub}.  
\end{proof}

\begin{proof}[Proof of Theorem \ref{thm-berg-cb}]
The proof of Theorem \ref{thm-berg-cb} for arbitrary dimension $d \ge 2$ is similar to that of Theorem \ref{thm-berg-cb} for dimension $d = 1$, where the circle mean-value property for the harmonic functions is played by the surface mean-value property \eqref{inv-mvp} of $\MM$-harmonic functions and thus for the holomorhphic functions and  the identity \eqref{M-identity} is played by similar identity (see e.g. Rudin \cite[Theorem 2.2.2]{Rudin-ball}):
\[
 \frac{1}{1 - | \varphi_{z}(x)|^2} = \frac{|  1- \bar{x} \cdot z|^2}{(1 - |z|^2) (1 - |x|^2)} \quad x, z \in \D_d,
\]
where $\varphi_z$ is defined by the formula \eqref{inv-auto}. 
\end{proof}

\subsection{Estimates of kernels}

\begin{proposition}\label{prop-cb-kernel}
For  the non-negative definite kernel $K: \D_d \times \D_d \rightarrow \C$ of the form \eqref{rep-kernel-cb}, if the coefficients $a_n$ in \eqref{rep-kernel-cb} satisfy the limit relation \eqref{cb-sm-coef}, then there exists $\Theta: \R^{+} \rightarrow \R^{+}$ with $\lim_{t\to\infty} \Theta(t)=0$ such that the kernel defined in \eqref{rep-kernel-cb} satisfies
\begin{align}\label{cb-up-diag}
K(z, z)  \le \Theta\Big(\frac{1}{1 - |z|^2}\Big) \cdot \frac{1}{(1  - |z|^2)^d} \log \Big( \frac{2}{1 - |z|^2}\Big).
\end{align}
In particular, the reproducing kernel defined in \eqref{rep-kernel-cb} satisfies Assumption \ref{ass-gcb}. 
\end{proposition}

\begin{proof}
Set $b_k: = \sum_{n: |n| = k} a_n$. By \eqref{rep-kernel-cb}, we have 
\[
K(z, z) = \sum_{n\in \N_0^d} a_n |z_1|^{2n_1} \cdots  |z_d|^{2n_d} \le  \sum_{n \in \N_0^d} a_n |z|^{2 |n|}  = \sum_{k =0}^\infty b_k |z|^{2k}. 
\]
 By \eqref{cb-sm-coef}, there exists  a sequence $(c_k)_{k\ge 0}$ of positive numbers with $\lim_{k\to\infty} c_k = 0$ such that $0 \le b_k \le c_k k^{d-1} \log (k+2)$ for any positive integer $k \ge 1$. For proving the estimate \eqref{cb-up-diag}, it suffices to prove that there exists $C>0$ such that for any $t\in (0, 1)$, 
\begin{align}\label{tp-cb-up}
\sum_{k =1}^\infty k^{d-1} \log (k + 2) t^k \le  \frac{C}{(1  - t)^d} \log \Big( \frac{2}{1 - t}\Big).
\end{align}
Set $G(r): = \sum_{n=0}^\infty \log (n + 2) r^n$ for $r \in (0, 1)$. Then 
\begin{align}\label{f-G-r}
(1 - r) G(r) = \log 2 + \sum_{n=1}^\infty \log \Big(1 + \frac{1}{n+1}\Big) r^n \le  \log 2 + \sum_{n=1}^\infty \frac{r^n}{n} =  \log\Big( \frac{2}{ 1 - r}\Big).
\end{align}
Thus
\[
\sum_{k= 1}^\infty  \log (k+2) t^k  \le \frac{1}{1 -t} \log \left( \frac{2}{1-t}\right), \text{\, for $t\in (0, 1)$.}
\] 
Therefore, for proving \eqref{tp-cb-up}, it suffices to prove that for any $m \in \N$, there exists $C_m > 0$, such that 
\begin{align}\label{m-ineq-ind}
(1 - t) \sum_{k=1}^\infty k^m \log (k+2) t^k \le C_m \sum_{k = 1}^\infty  k^{m-1}\log (k +2) t^k,  \text{\, for $t\in (0, 1)$.}
\end{align}
The inequality \eqref{m-ineq-ind} follows from the following inequalities 
\begin{multline*}
(1 - t) \sum_{k=1}^\infty k^m \log (k+2) t^k  =  \sum_{k=1}^\infty \Big[k^m \log (k+2) - (k -1)^m  \log (k +1) \Big] t^k  \le 
\\
\le \sum_{k=1}^\infty  \sup_{k-1 \le \theta \le k}  \Big[m \theta^{m-1} \log (\theta +2) + \frac{\theta^m }{\theta +2}  \Big] t^k  \le  2m \sum_{k=1}^\infty k^{m-1} \log (k +2) t^k. 
\end{multline*}
Finally, by Corollary \ref{cor-in-MVP}, the pointwise estimate \eqref{cb-up-diag} implies both \eqref{rough-es-cb} and \eqref{fine-es-cb} and thus the kernel $K$ defined in  \eqref{rep-kernel-cb} satisfies Assumption \ref{ass-gcb}. 
\end{proof}

\begin{proposition}\label{prop-w-cb}
For any $T> 0$, there exists $\Theta: \R^{+} \rightarrow \R^{+}$ with $\lim_{t\to\infty} \Theta(t)=0$ such that the reproducing kernel $K^{\omega_T^{(d)}}$  of the weighted Bergman space $A^2(\D_d, \omega_T^{(d)})$ corresponds to the weight on $\D_d$ defined by the formula \eqref{def-w-cb} satisfies  
\[
K^{\omega_T^{(d)}}(z, z)  \le \Theta\Big(\frac{1}{1 - |z|^2}\Big) \cdot \frac{1}{(1  - |z|^2)^d} \log \Big( \frac{2}{1 - |z|^2}\Big).
\]
In particular, the weighted Bergman space $A^2(\D_d, \omega_T^{(d)})$ satisfies Assumption \ref{ass-gcb}. 
\end{proposition}

\begin{proof}
Since $\omega_T^{(d)}$ is radial, the polynomials $(z^{n})_{n\in \N_0^d}$ are orthogonal in $L^2(\D_d, \omega_T^{(d)})$. Therefore, the reproducing kernel  of the corresponding weighted Bergman space is given by  $K^{\omega_T^{(d)}}(z, w) = \sum_{n\in \N_0^d} a_n z^n \bar{w}^n$ with $a_n= \| z^n\|_{A^2(\D, \omega_T^{(d)})}^{-2}$. Now for any $n\in \N_0^d$,  using the rotational invariance of $\omega_T^{(d)}$ and the following identity, see e.g. Zhu \cite[Lemma 1.11]{Zhu-UB}, 
\[
\int_{\Sph_d} | \zeta^n|^2 d\sigma_{\Sph_d}(\zeta) = \frac{(d-1)! n_1! \cdots n_d!}{ (d-1 + |n|)!},
\]
the exists a constant $c_d > 0$ such that 
\begin{multline*}
a_n^{-1} = \int_{\D_d} | z^n|^2 \omega_T^{(d)} (z) dV(z) = \int_{\D_d} | z|^{2|n|} \omega_T^{(d)} (z) \int_{\Sph_d} | \zeta^n|^2 d\sigma_{\Sph_d}(\zeta)  dV(z) =
\\
= c_d \frac{n_1! \cdots n_d!}{ (d-1 + |n|)!} \int_0^1 t^{|n| + d -1} ( 1 - t)^{-1} \log ^{- 1} \Big(  \frac{4}{ 1 - t}\Big) \log^{- (1 + T)} \Big( \log \Big( \frac{4}{ 1 -t}\Big)\Big)  dt.  
\end{multline*}
Note that for any $k \ge 2$,  we have 
\begin{multline*}
 \int_0^1    t^k ( 1 - t)^{-1} \log ^{- 1} \Big(  \frac{4}{ 1 - t}\Big) \log^{- (1 + T)} \Big( \log \Big( \frac{4}{ 1 -t}\Big)\Big)  dt =
\\
 =   \int_{\log (\log 4)}^\infty (1 - 4 e^{- e^x})^k \frac{dx}{x^{1 + T}} \ge    \int_{\log (\log (4k))}^\infty (1 - 4 e^{- e^x})^k \frac{dx}{x^{1 + T}} \ge 
\\
\ge  \int_{\log (\log (4k))}^\infty \Big(1 - \frac{1}{k}\Big)^k \frac{dx}{x^{1 + T}}  = \frac{1}{T} \Big(1 - \frac{1}{k}\Big)^k   \log^{-T} (\log(4k)).   
\end{multline*}
Therefore, there exists a constant $c_{d, T}> 0$ such that  for  any $n\in \N_0^d$ with $|n| \ge 2$, we have 
\[
a_n^{-1} \ge  c_{d, T} \frac{n_1! \cdots n_d!}{ (d-1 + |n|)!}  \log^{-T} \Big(\log(4 (|n| + d -1))\Big),
\]
whence 
\[
a_n \le c_{d, T}^{-1} \frac{ (d-1 + |n|)!}{n_1! \cdots n_d!} \log^{T} \Big(\log(4 (|n| + d -1))\Big).
\]
Now by the identity (see e.g. \cite[formula (1.1)]{Zhu-UB}) 
\[
\sum_{|n| = k } \frac{1}{n_1!\cdots n_d!} = \frac{d^k}{k!},
\]
there exists a constant $C_{d, T}> 0$, such that for any $k \ge 2$, 
\begin{multline*}
\sum_{|n| = k}  a_n \le c_{d, T}^{-1} \frac{(d - 1 + k)! d^k}{k!}\log^{T} \Big(\log(4 (k + d -1))\Big) \le 
\\
\le C_{d, T} k^{d-1} \log^{T} \Big(\log(4 (k + d -1))\Big).
\end{multline*}
Therefore, Proposition \ref{prop-w-cb} follows from Proposition \ref{prop-cb-kernel}.
\end{proof}

\section{The real hyperbolic spaces}\label{sec-rhs}

\subsection{Main results for real hyperbolic spaces}

Let $m\ge 2$ be an integer.  We apply the results of \S \ref{sec-gen}  to the triple $(\B_m, d_h, \mu_{\B_m})$ consisting of the real hyperbolic space $\B_m$ equipped with the hyperbolic metric $d_h(\cdot, \cdot)$ given by \eqref{def-d-h} and the associated hyperbolic volume measure $\mu_{\B_m}$ given by \eqref{hyper-vol-def}.  We choose the origin $0\in \B_m$ as the base point.  The hyperbolic volume measure $\mu_{\B_m}$ is invariant under the action of the group $\Aut(\B_m)$ of all M\"obius transformations preserving  $\B_m$. Recall, cf. e.g.  Stoll \cite[Theorem 2.1.2]{Stoll-hyper-ball}, that any  $\phi \in \Aut(\B_m)$ has the form $ \phi  = A \psi_a$, where $A$ is an orthogonal transformation of $\R^m$ and $\psi_a$ is the involution defined in \eqref{def-psi-a}.

Recall the definition in \S \ref{sec-intro-rhs} of  the $\HH$-harmonic functions on $\B_m$.  We denote by $\HH(\B_m)$  the set of all complex-valued $\HH$-harmonic functions on $\B_m$. Consider a reproducing kernel Hilbert space $\mathscr{H}(\widetilde{K}) \subset \HH (\B_m)$ with  a reproducing kernel $\widetilde{K}$.   To any $x\in \B_m$, we assign a function $\widetilde{K}_x \in \mathscr{H}(\widetilde{K})$ by setting $\widetilde{K}_x(y) = \widetilde{K}(y, x)$ for $y \in \B_m$.

\begin{assumption}\label{ass-grb}
There exists a non-decreasing sub-exponential function $\Lambda: \N \rightarrow \R^{+}$ such that for any $k\in \N$, we have 
\begin{align}\label{rough-es-rb}
\int_{A_k^{[m]}(0)} \widetilde{K}(x, x) d\mu_{\B_m}(x)   \le \Lambda (  k )   e^{2 (m-1) k}, 
\end{align}
where $A_k^{[m]}(0) = \{x\in\B_m: k \le d_h(x, 0) < k + 1\}$. The limit equality holds: 
\begin{align}\label{fine-es-rb}
\lim_{\alpha \to 0^{+}}  \alpha^2 \int_{\B_m} \widetilde{K}(x, x) (1- |x|^2)^{\alpha + m-2}  dV_m(x) = 0.
\end{align}
\end{assumption}

\begin{theorem}\label{thm-rep-rb}
Let $\mathscr{H}(\widetilde{K}) \subset \HH(\B_m)$ be a reproducing kernel Hilbert space satisfying Assumption \ref{ass-grb}. Fix $a \in \B_m$ and a  sequence $(s_n)_{n\ge 1}$  in $(m-1, \infty)$ converging to $m-1$  and  
\[
\sum_{n=1}^\infty (s_n-m+1)^2 \int_{\B_m} \widetilde{K}(x, x) (1 - |x|^2)^{(s_n - 1)}  dV_m(z)<\infty.
\]
  If $\Pi$ is a point process on $\B_m$ satisfying  Assumption \ref{ass-pp} and having first intensity measure $\lambda \mu_{\B_m}$ with $\lambda> 0$ a constant, then $\Pi$-almost any $X \in \Conf(\B_m)$ satisfies:
\begin{enumerate}
\item For any $n$, we have 
\[
\sum_{k=0}^\infty \Big\| \sum_{x \in X \atop k\le d_h (a,x) < k+1} e^{-s_n d_h(a, x)} \widetilde{K}_x\Big\|_{\mathscr{H}(\widetilde{K})}< \infty.
\]
\item The following limit equality holds: 
\[
\lim_{n\to\infty}  \left\| \frac{\displaystyle{\sum_{k=0}^\infty \sum_{x \in X \atop k\le d_h (a,x) < k+1} e^{-s_n d_h(a, x)} \widetilde{K}_x}}{\displaystyle{  \sum_{x \in X} e^{-s_n d_h(a,x)} }}  - \widetilde{K}_a \right\|_{\mathscr{H}(\widehat{K})} = 0
\]
\item   The limit equality 
\begin{align*}
 f(z) = \lim_{n\to\infty} \frac{\displaystyle{\sum_{k=0}^\infty \sum_{x \in X \atop k\le d_h (a,x) < k+1} e^{-s_n d_h(a, x)} f(x)}}{\displaystyle{  \sum_{x \in X} e^{-s_n d_h(a,x)} }}
\end{align*}
 holds simultaneously for all $f \in \mathscr{H}(\widetilde{K})$ at our fixed point $a\in \B_m$. 
\end{enumerate}
\end{theorem}

Recall the definition of Poisson transformation \eqref{def-cb-poi} of a  signed measure on $\Sph_d$ of total variation. 
Given any finite Positive Borel measure $\mu$ on $S^{m-1}$. Set
\begin{align}\label{def-rb-hardy}
h^2(\B_m; \mu) : &= \left\{f: \B_m\rightarrow \C\Big|  f = P^h[g\mu], \, g \in L^2(S^{m-1}, \mu)\right\}.
\end{align}

For any fixed $ t \in S^{m-1}$, the function $x \mapsto P^h(x, t)$ is $\HH$-harmonic on $\B_m$.
In what follows, for any $x \in \B_m$, we denote 
\begin{equation}\label{def-rb-pk}
P_x^h(t) = P^h(x, t), \quad t \in S^{m-1}.
\end{equation}

\begin{theorem}\label{thm-intro-rb-no-ss}
Let $\mu$ be any Borel probability measure on $S^{m-1}$.  If $\Pi$ is a point process on $\B_m$ satisfying  Assumption \ref{ass-pp} and having first intensity measure $\lambda \mu_{\B_m}$ with $\lambda> 0$ a constant, then 
for $\Pi$-almost any $X \in \Conf(\B_m)$ we have:
\begin{enumerate}
\item For any $s>m-1$ and any $a\in \B_m$, the series $\sum_{x\in X} e^{-s d_h(x, a)} P_x^b$ converges unconditionally in $L^2(\mu)  = L^2(S^{m-1}, \mu)$.
\item  For any $a \in \B_m$, the functions \eqref{def-rb-pk} satisfy
\[
\lim_{s\to (m-1)^{+}} \left\| \frac{\displaystyle{\sum_{x \in X}   e^{-s d_h(x, a)} P_x^h}}{\displaystyle{  \sum_{x \in X}  e^{-s d_h(x, a)} }} - P_{a}^h \right\|_{L^2(\mu)} = 0.
\]
\item For any $ f  \in h^2(\B_m; \mu)$, any $a\in \B_m$ and any $s >m-1$, the series  $\sum_{x \in X}   e^{-s d_h(x, a)} f(x)$ converges absolutely  and 
\[
f(a) =    \lim_{s\to (m-1)^{+}}  \frac{ \displaystyle{ \sum_{x \in X}   e^{-s d_h(x, a)} f(x) } }{\displaystyle{  \sum_{x \in X}   e^{-s d_h(x, a)} }}. 
\]
\item For all $a\in \B_m$, the following weak convergence of probability measures on $\overline{\B}_m$ holds:
\[
\lim_{s\to (m-1)^{+}}\frac{\displaystyle{\sum_{x \in X}   e^{-s d_h(x, a)} \delta_x}}{\displaystyle{  \sum_{x \in X}  e^{-s d_h(x, a)} }}  =  P^h(a, \zeta) d \sigma_{S^{m-1}}(\zeta).
\]
\end{enumerate}
\end{theorem}

Similar result as Theorem \ref{thm-berg-cb} for  point processes on the real hyperbolic space $\B_m$  whose first intensity measures are not necessarily conformally invariant is stated in Theorem \ref{thm-berg-rb} below. 

For any $\alpha> -1$, the weighted $\HH$-harmonic Bergman space $B_\alpha^2(\B_m)$ is defined by 
\[
B_\alpha^{2}(\B_m): = \Big\{f: \B_m \rightarrow \C \Big| \text{$f$ is $\HH$-harmonic and $\int_{\B_m} |f(x)|^2 (1 - |x|^2)^\alpha dV_m(x) < \infty$}\Big\}. 
\]
See  Stoll \cite[Chapter 10]{Stoll-hyper-ball} for more details, note that the space $B_\alpha^{2}(\B_m)$ defined as above coincides with the space $\mathcal{B}^2_{\alpha + m}$ defined in Stoll's book. The space $B_\alpha^2(\B_m)$ is a reproducing kernel Hilbert space, while its reproducing kernel, denoted by $\mathcal{R}^\alpha(x, y)$, is not known explicitly (see Stoll \cite[Exercise 10.8.11]{Stoll-hyper-ball}). In this paper, we however only need the following estimates (which is an immediate consequence of \cite[inequality (10.1.5)]{Stoll-hyper-ball}): there exists a constant $C = C_{m, \alpha}>0$, such that 
\begin{align}\label{diag-R-kernel}
\mathcal{R}^\alpha(x, x) \le  \frac{C}{(1- |x|^2)^{\alpha  + m}}.
\end{align}
To any $x\in \B_m$, we assign a function $\mathcal{R}_x^\alpha \in B_\alpha^2(\B_m)$ by setting $\mathcal{R}^\alpha_x(y) = \mathcal{R}^\alpha(y, x)$. 

\begin{theorem}\label{thm-berg-rb}
 Let $\alpha> -1$ and $\beta \ge \alpha  + m >  m-1$.  Fix $a \in \B_m$ and a  sequence $(s_n)_{n\ge 1}$  with $s_n > \beta$ such that $\sum_{n= 1}^\infty (s _n -\beta) <\infty$. 
  If $\Pi$ is a point process on $\B_m$ satisfying  Assumption \ref{ass-pp} and having first intensity measure 
\[
\frac{\lambda dV_m(x)}{(1 - |x|^2)^{\beta + 1}}
\]
 with $\lambda> 0$ a constant. Then $\Pi$-almost any $X \in \Conf(\B_m)$ satisfies:
\begin{enumerate}
\item For any $n$, we have 
\[
\sum_{k=0}^\infty \Big\| \sum_{x \in X \atop k\le d_h (a,x) < k+1} e^{-s_n d_h(a, x)}  \left( \frac{1-|x|^2}{1 - |\psi_a(x)|^2}\right)^{\beta + 1 -m} \mathcal{R}_x^\alpha \Big\|_{B^2_\alpha(\B_m)}< \infty,
\]
where $\psi_a$ is defined by the formula \eqref{def-psi-a}.
\item The following limit equality holds: 
\[
\lim_{n\to\infty}  \left\| \frac{\displaystyle{\sum_{k=0}^\infty \sum_{x \in X \atop k\le d_h (a,x) < k+1} e^{-s_n d_h(a, x)}   \left( \frac{1-|x|^2}{1 - |\psi_a(x)|^2}\right)^{\beta + 1 -m} \mathcal{R}_x^\alpha }}{\displaystyle{  \sum_{x \in X} e^{-s_n d_h(a,x)}   \left( \frac{1-|x|^2}{1 - |\psi_a(x)|^2}\right)^{\beta + 1 -m}} }- \mathcal{R}_a^\alpha \right\|_{A^2_\alpha(\D_d)} = 0
\]
\item   The limit equality 
\[
 f(z) = \lim_{n\to\infty} \frac{\displaystyle{\sum_{k=0}^\infty \sum_{x \in X \atop k\le d_h (a,x) < k+1} e^{-s_n d_h(a, x)}   \left( \frac{1-|x|^2}{1 - |\psi_a(x)|^2}\right)^{\beta + 1 -m}  f(x) }}{\displaystyle{  \sum_{x \in X} e^{-s_n d_h(a,x)}   \left( \frac{1-|x|^2}{1 - |\psi_a(x)|^2}\right)^{\beta + 1 -m} }}
\]
 holds simultaneously for all $f \in B^2_\alpha(\B_m)$ at the point $a\in \B_m$. 
\end{enumerate}
\end{theorem}

\begin{remark}
Note that by \cite[Formula (2.1.7)]{Stoll-hyper-ball}, we have 
\[
\frac{1-|x|^2}{1 - |\psi_a(x)|^2} = \frac{ |x- a|^2 + ( 1 - |x|^2)(1 - |a|^2)}{1 - |a|^2}, \, a, x \in \B_m.
\]
\end{remark}

\subsection{Proof of Theorems \ref{thm-rep-rb}, \ref{thm-intro-rb-no-ss} and \ref{thm-berg-rb}}
\begin{lemma}\label{lem-rb-ent}
There exists a constant $c_m>0$ such that for any $x\in \B_m$, we have
\begin{align}\label{pvg-real-ball}
\lim_{r\to\infty} \frac{\mu_{\B_m}(B(x,r))}{ e^{r(m-1)}}  = c_m.
\end{align}
In particular, we have  $h_{\B_m} = m-1$.
\end{lemma}
\begin{proof}
The proof is similar to that of Lemma \ref{lem-disk-ent}. 
\end{proof}

Recall the definition of the $\HH$-harmonic functions on $\B_m$  in \S \ref{sec-intro-rhs}.

\begin{lemma}[{See Stoll \cite[Theorem 4.3.5]{Stoll-hyper-ball}}]\label{lem-rb-mvp}
Let $\mathscr{H}$ be a Hilbert space over $\R$ or $\C$.  Let $u: \B_m\rightarrow \mathscr{H}$ be an $\HH$-harmonic function. Then for any  $a \in \D_m$ and $r > 0$, we have 
\begin{align}\label{mvp-rb}
u(a) =\frac{1}{\mu_{\B_m}(B(a, r))}\int_{B(a, r)} u(x) d\mu_{\B_m}(x).
\end{align}
\end{lemma}


Recall the definition of  the space   $\mathrm{MVP}(\mu_M; \mathscr{H})$ introduced in \S \ref{sec-general-Hilbert}. 
\begin{lemma}\label{lem-mvp-rb}
Let $\mathscr{H}(\widetilde{K}) \subset \HH(\B_m)$ be a reproducing kernel Hilbert space with reproducing kernel $\widehat{K}$ satisfying Assumption \ref{ass-grb}. Let $f: \B_m \rightarrow \mathscr{H}(\widetilde{K})$ be the map  
\[
\B_m \ni x \mapsto f(x) := \widetilde{K}_x = \widetilde{K}(\cdot, x) \in \mathscr{H}(\widetilde{K}).
\]  
 Then $f \in \mathrm{MVP}(\mu_{\B_m}; \mathscr{H}(\widetilde{K}))$. 
\end{lemma}

\begin{proof}
The proof is similar to that of Lemma \ref{lem-mvp-rep}, the role played by the equality \eqref{mvp-disk} for harmonic functions on $\D$ is now played by the equality \eqref{mvp-rb} for $\HH$-harmonic functions on $\B_m$. 
\end{proof}

\begin{proof}[Proof of Theorem \ref{thm-rep-rb}]
Lemma \ref{lem-rb-ent} implies that the triple $(\B_m, d_h, \mu_{\B_m})$ satisfies Assumption \ref{ass-M-pvg} and hence Assumption \ref{ass-M-A1}. It is easy to see that \eqref{fine-es-rb} is equivalent to \eqref{global-growth-f} in this case. Therefore,  by Lemma \ref{lem-mvp-rb}, Theorem \ref{thm-rep-cb} follows from Theorem \ref{thm-H-L}. 
\end{proof}

\begin{lemma}\label{lem-cmvp-rb}
Let $\mu$ be a Borel probability measure on $S^{m-1}$. Then the map 
\[
\B_m \ni x \mapsto f(x) : = P_x^h \in L^2(\mu; \R)_{+}
\]
belongs to the class $\mathrm{CMVP}^\sharp(\mu_{\B_m}; L^2(\mu; \R)_{+})$. 
\end{lemma}

\begin{proof}
 For any $t \in S^{m-1}$, the function $\B_m \ni x \mapsto  P_x^h (t) = P^h(x, t)$ is $\HH$-harmonic, hence so is  the vector valued function $\B_m \ni x \mapsto f(x)= P_x^h$. By the equality \eqref{mvp-rb}, we have $f \in \widetilde{\mathrm{MVP}}(\mu_{\B_m}; L^2(\mu; \R))$. 
By  Stoll \cite[Theorem 5.5.7]{Stoll-hyper-ball}, there exists $c> 0$ such that 
\begin{align}\label{2-norm-poi-hyp}
\int_{S^{m-1}} P^h(x, t)^2 d \sigma_{S^{m-1}}(t) \le \frac{c}{(1 - |x|^2)^{m-1}} \quad \text{for all $x\in \B_m$}.
\end{align}
This implies that there exists a constant $c'> 0$  such that 
\[
\int_{S^{m-1}} P^h(x, t)^2 d \sigma_{S^{m-1}}(t) \le   c' e^{h_{\B_m} d_h(x, 0)} \quad \text{for all $x\in \B_m$}.
\]
Therefore, for any positive integer $k \in \N$,  using the rotational invariance of the measure $\mu_{\B_m}$ and recall that $A_k^{[m]}(0) = \{x\in\B_m: k \le d_h(x, 0) < k + 1\}$,  we have 
\begin{multline*}
\int_{A_k^{[m]}(0)} \| f(x) \|_{L^2(\mu)}^2 d\mu_{\B_m}(x) 
=  \int_{A_k^{[m]}(0)} \int_{S^{m-1}}     \left(\frac{ 1 - |x|^2}{| x - t|^2}\right)^{2m-2}   d \mu(t) d\mu_{\B_m}(x) = 
\\
= \int_{A_k^{[m]}(0)} \int_{S^{m-1}}  \Big[  \int_{\mathcal{O}_m}\left(  \frac{1-|V x|^2}{|V x-  t|^{2}} \right)^{2m-2}  dm_{\mathcal{O}_m} (V) \Big]  d \mu(t) d\mu_\D(x),
\end{multline*}
where $\mathcal{O}_m$ is the group of $m\times m$ orthogonal matrices and $m_{\mathcal{O}_m}$ is the normalized Haar measure on it.  Since $\sigma_{S^{m-1}}$ coincides with the orbital measure under the transitive action of $\mathcal{O}_m$, for any $t \in S^{m-1}$, we have 
\begin{multline*}
 \int_{\mathcal{O}_m}\left(  \frac{1-|V x|^2}{|V x-  t|^{2}} \right)^{2m-2}  dm_{\mathcal{O}_m} (V) =  \int_{\mathcal{O}_m}\left(  \frac{1-|x|^2}{|x-  V^{-1}t|^{2}} \right)^{2m-2}  dm_{\mathcal{O}_m} (V) 
=
\\
=  \int_{S^{m-1}} \left(  \frac{1-|x|^2}{|x - t'|^{2}} \right)^{2m-2}  d\sigma_{S^{m-1}}(t')  = \int_{S^{m-1}} P^h(x, t)^2 d \sigma_{S^{m-1}}(t) \le   c' e^{h_{\B_m} d_h(x, 0)}. 
\end{multline*}
Thus we obtain 
\begin{multline*}
\int_{A_k^{[m]}(0)} \| f(x) \|_{L^2(\mu)}^2 d\mu_{\B_m}(x) 
\le  \int_{A_k^{[m]}(0)} \int_{S^{m-1}}  c' e^{h_{\B_m} d_h(x, 0)} d \mu(t) d\mu_\D(x)    = 
\\
 =  \int_{A_k^{[m]}(0)}   c' e^{h_{\B_m} d_h(x, 0)}  d\mu_\D(x) \le c' e^{h_{\B_m} (k+1)} \mu_{\B_m} (B(0, k+1)). 
\end{multline*}
By Lemma \ref{lem-rb-ent}, there exists $c''>0$ such that $\mu_{{\B_m}} (B(0, k+1)) \le c'' e^{k h_{\B_m}}$. Consequently, there exists a constant $C>0$, such that for any $k\in \N$, we have 
\begin{align}\label{mg-poi-rb}
\int_{A_k^{[m]}(0)} \| f(x) \|_{L^2(\mu)}^2 d\mu_{\B_m}(x) \le  C e^{2k h_{\B_m}}. 
\end{align}
Since the triple $(\B_m, d_h, \mu_{\B_m})$ satisfies Assumption \ref{ass-M-pvg}, by Proposition \ref{prop-mgc}, the mean-growth estimate \eqref{mg-poi-rb} implies that $f\in \mathrm{MVP}(\mu_{\B_m}; L^2(\mu; \R))$. Clearly,  $f$  is continuous and non-negative, whence $f\in \mathrm{CMVP}(\mu_{\B_m}; L^2(\mu; \R)_{+})$. 

Now note that for $\varepsilon \in (0, 1)$, there exists a constant $C'>0$ such that $1 - e^{- 2 \varepsilon} \ge C' \varepsilon$. Therefore, using the estimate \eqref{mg-poi-rb}, we obtain, for any $\varepsilon\in (0, 1)$, that  
\begin{multline*}
\int_{\B_m} e^{-2 \varepsilon d_h(x, 0)} \| f(x) \|_{L^2(\mu)}^2 e^{-2 h_{\B_m}\cdot d_h(x, 0)} d\mu_{\B_m}(x) \le 
\\
\le \sum_{k = 0}^\infty e^{-(2 h_{\B_m} + 2 \varepsilon) k} \int_{A_k^{[m]}(0)} \| f(x)\|_{L^2(\mu)}^2 d\mu_{\B_m}(x) \le 
C \sum_{k=0}^\infty e^{-2 \varepsilon k }  = \frac{C}{1 - e^{-2 \varepsilon}} \le \frac{C}{C' \varepsilon}. 
\end{multline*}
That is,  the function $f$ satisfies the estimate \eqref{bdd-e-a} with $\beta = 0$ and $\alpha = 1$ and thus  we can conclude that $f \in \mathrm{CMVP}^\sharp (\mu_{\B_m}; L^2(\mu; \R)_{+})$. 
\end{proof}

\begin{proof}[Proof of Theorem \ref{thm-intro-rb-no-ss}]
By Lemmata \ref{lem-rb-ent} and \ref{lem-cmvp-rb}, Theorem \ref{thm-intro-rb-no-ss} follows immediately from Theorem \ref{thm-n-sub}.  
\end{proof}

\begin{proof}[Proof of Theorem \ref{thm-berg-rb}]
The proof of Theorem \ref{thm-berg-rb} is similar to that of Theorem \ref{thm-berg-cb}. 
\end{proof}

\section{Sharpness of the simultaneous reconstructions}\label{sec-more-disk}

\subsection{Compactly supported radial weights}
 We return to  the determinantal  point processes governed by the Bergman kernels on the complex hyperbolic spaces $\D_d$ for arbitrary dimension $d$. We will show that, although as a consequence of our result on Lyons-Peres completeness conjecture \cite{BQS-LP},  almost surely, the configuration $X \subset \D_d$ selected randomly with respect to the determinantal point process in question is a uniqueness set for  $A^2(\D_d)$, it is  impossible to recover simultaneously all functions $f\in A^2(\D_d)$  using  averaging  with compactly supported radial weights.

Let $dv_d$ denote the normalized Lebesgue measure $dv_d$ on $\D_d$ such that $v_d(\D_d) = 1$. Then, see Rudin \cite[\S 3.1.2]{Rudin-ball},  the Bergman kernel $K_{\D_d}$ of $\D_d$ with respect to the measure $dv_d$,  is given by 
\begin{align}\label{berg-exp}
K_{\D_d} (z, w) = \frac{1}{ ( 1 - z \cdot \bar{w})^{ d+1}}.
\end{align}
For any $x\in \D_d$, let  $K_{\D_d}^x \in A^2(\D_d)$ be defined by 
\[
K_{\D_d}^x ( z) = K_{\D_d}(z, x).
\]

 Given a compactly supported {radial} weight function $W:\D_d \rightarrow \R^{+}$, introduce  the family $(W^{z})_{z\in\D_d}$ of weight functions by the formula
\[
W^{z}(x) :=  W(\varphi_{z}(x)), \quad x\in \D_d,
\]
where $\varphi_z$ is the involutive biholomorphism of $\D_d$ given by the formula \eqref{inv-auto}.

For simplifying notation, in this section, we will denote the determinantal point process induced by $K_{\D_d}$ by $\PP_{d}$, in notation, this means we set
\[
\PP_d: = \Pi_{K_{\D_d}}. 
\]

\begin{proposition}\label{prop-failure}
For any $z\in \D_d$, we have 
\[
\inf_{W} \frac{\displaystyle \Var_{\PP_{d}} \Big[ \sum_{x\in \X} W^{z} (x)  K_{\D_d}^x  \Big]   }{  \displaystyle  \Big[  \E_{\PP_{d}} \Big( \sum_{x\in \X} W^{z} (x)  \Big) \Big]^2 } > 0,
\]
where the infimum is over  all compactly supported bounded radial weights $W: \D_d\rightarrow \R^{+}$.
\end{proposition}

\subsection{Estimation  of the variances}

We start with 

\begin{proposition}\label{prop-new-var-f}
There exists a constant $c_d$ depending only on $d$, such that for any compactly supported bounded radial weight $W: \D_d \rightarrow \R^{+}$ and   any  $z_o\in \D_d$,
\[
\Var_{\PP_{d}} \Big[ \sum_{x\in \X} W^{z_o} (x)  K_{\D_d}^x  \Big] \ge \frac{c_d}{( 1 - |z_o|^2)^{d+1}} \int_{\D_d} \int_{\D_d} | W(z) - W(w)|^2  \cdot I_d(z, w) dv_d(z) dv_d(w),
\]
where $I_d(z,w)$ is given by the formula:
\begin{align}\label{I-1-id}
I_d(z, w)  = \left\{
\begin{array}{lc} 
(1 - |z|^2 |w|^2)^{-5} & \text{if $d =1$}
\vspace{2mm}
\\
(1 - |z|^2|w|^2)^{-3d -1} & \text{if $d \ge 2$}
\end{array}.
\right.
\end{align}  
\end{proposition}

\begin{proposition}\label{prop-Sob}
If $g\in L^2(\D_d, dv_d)$ takes real values and has compact support, then
\begin{align}\label{var-rep-stat}
\Var_{\PP_d} \Big[ \sum_{x\in \X} g(x) K_{\D_d}^x  \Big] = \frac{1}{2} \int_{\D_d} \int_{\D_d} | g(z) - g(w)|^2 \cdot   K_{\D_d}(z, w) |K_{\D_d}(z, w)|^2 dv_d(z) dv_d(w). 
\end{align}
\end{proposition}

Recall the standard Sobolev norm expression  for the variance of linear statistics under determinantal point processes induced by an orthogonal projection.
\begin{lemma}\label{lem-var-sob}
If $\mathscr{H}$  is a Hilbert space and  $f\in L^2( \D_d, dv_d; \mathscr{H})$ has compact support, then 
\[
\Var_{\PP_{d}} \Big[    \sum_{x\in  \X} f(x)  \Big]= \frac{1}{2} \int_{\D_d} \int_{\D_d} \|f(z) - f(w) \|_{\mathscr{H}}^2 \cdot  |K_{\D_d}(z, w)|^2 dv_d(z) dv_d(w). 
\]
\end{lemma}

\begin{lemma}\label{lem-rep}
For any $z\in\D_d$, the Bergman kernel $K_{\D_d}$ satisfies
\[
\int_{\D_d} K_{\D_d}(z, w) | K_{\D_d}(z, w)|^2 dv_d (w)  = \int_{\D_d} K_{\D_d}(w, z) | K_{\D_d}(z, w)|^2 dv_d (w) = K_{\D_d}(z, z)^2.
\]
\end{lemma}

\begin{proof}
Fix $z\in \D_d$, define  $G(w): = K_{\D_d}(w, z)^2$. Clearly,  $G\in A^2(\D_d)$. Therefore, $\int_{\D_d}  K_{\D_d}(z, w) G(w) dv_d(w) = G(z)$. That is, 
$
\int_{\D_d} K_{\D_d}(w, z) | K_{\D_d}(z, w)|^2 dv_d(w) = K_{\D_d}(z, z)^2. 
$
Taking conjugate on both sides  gives 
$
\int_{\D_d} K_{\D_d}(z, w) | K_{\D_d}(z, w)|^2 dv_d(w) = K_{\D_d}(z, z)^2. 
$
\end{proof}

\begin{proof}[Proof of Proposition \ref{prop-Sob}]
Set $L(z, w): = K_{\D_d}(z, w) |K_{\D_d}(z, w)|^2$.  By exchanging the role of the integrated variables $z, w$ and using the fact that $g$ is real, we have  \[
\int_{\D_d} \int_{\D_d} g(z)^2 L(w, z) dv_d(z) dv_d(w) = \int_{\D_d} \int_{\D_d} g(w)^2 L(w, z) dv_d(z) dv_d(w)
\]
 and  \[\int_{\D_d} \int_{\D_d} g(z)  g(w) L(w, z) dv_d(z) dv_d(w)  = \int_{\D_d} \int_{\D_d} g(z)  g(w) L(z, w) dv_d(z) dv_d(w).
\] 
The above identities combined with Lemmata \ref{lem-var-sob} and \ref{lem-rep}  and the identity $K_{\D_d}(z, z) = \int_{\D_d} | K_{\D_d}(z, w)|^2 dv_d(w)$ imply 
\begin{multline*}
 \Var_{\PP_{d}} \Big[  \sum_{x\in  \X} g(x) K_{\D_d}^x   \Big]   =  \int_{\D_d} g(z)^2  K_{\D_d}(z,z)^2 dv_d(z)    - \int_{\D_d} \int_{\D_d} g(z)  g(w) L(w, z) dv_d(z) dv_d(w)
\\
 =  \int_{\D_d} \int_{\D_d} g(z)^2 L(w, z) dv_d(z) dv_d(w) -  \int_{\D_d} \int_{\D_d} g(z)  g(w) L(w, z) dv_d(z) dv_d(w)
\\
=\frac{1}{2} \int_{\D_d} \int_{\D_d} | g(z) - g(w)|^2  L(z, w) dv_d(z) dv_d(w). 
\end{multline*}
This is the desired equality. 
\end{proof}

\begin{proof}[Proof of Proposition \ref{prop-new-var-f}]
 Fix $z_o\in \D_d$. Using Proposition \ref{prop-Sob}, the invariance of the measure $(1 - |z|^2)^{-d -1} dv_d(z)$ under the biholomorphic automorphisms on $\D_d$ (or using the conformally invariance of the Bergman kernel $K_{\D_d}(z,w)$, that is, the measure $|K_{\D_d}(z, w)|^2 dv_d(z)dv_d(w)$ is invariant under the diagonal action of the biholomorphic automorphisms) and the identity (see e.g. Rudin \cite[Theorem 2.2.2]{Rudin-ball}):
\begin{align}\label{mob-id}
1 -  \varphi_{z_o}(z) \cdot \overline{\varphi_{z_o}(w)} = \frac{(1 - |z_o|^2) (1 -  z \cdot \bar{w})}{(1 - z \cdot \bar{z_o}) ( 1 - z_o \cdot \bar{w})} \quad  z, w \in \D_d,
\end{align}
 we have 
 \begin{multline*}
\Var_{\PP_{d}} \Big[ \sum_{x\in \X} W^{z_o} (x)  K_{\D_d}^x  \Big] =  \frac{1}{2} \int_{\D_d}\int_{\D_d}   |W^{z_o}(z) - W^{z_o}(w) |^2 K_{\D_d}(z, w) | K_{\D_d}(z, w)|^2 dv_d(z) dv_d(w)
\\
=   \frac{1}{2} \int_{\D_d}\int_{\D_d}   |W(z) - W(w) |^2  K_{\D_d}(\varphi_{z_o}(z) , \varphi_{z_o}(w)) |    K_{\D_d}( z, w)|^2 dv_d(z) dv_d(w)
\\
 =  \frac{1}{2} \int_{\D_d}\int_{\D_d}    |W(z) - W(w) |^2   \underbrace{\left[\frac{(1 - z \cdot \bar{z_o}) ( 1 - z_o \cdot \bar{w})}{(1 - |z_o|^2) (1 -  z \cdot \bar{w})}  \right]^{d+1}  \frac{1}{| 1 - z \cdot \bar{w}|^{2d + 2}}}_{ \text{denoted $T(z, w; z_o)$}}dv_d(z) dv_d(w) . 
\end{multline*}
Assume now that $U$ is a $d\times d$ unitary matrix such that \[
Uz_o = |z_o| e_1, \quad e_1 = (1, 0, \cdots, 0)\in \C^d.
\] Then by using the rotational invariance of $W$ and the measure $dv_d$, we have 
\[
\Var_{\PP_{d}} \Big[ \sum_{x\in \X} W^{z_o} (x)  K_{\D_d}^x  \Big] = \frac{1}{2} \int_{\D_d}\int_{\D_d} |W(z) - W(w)|^2 T(z, w; |z_o|e_1) dv_d(z) dv_d(w).
\]
Using again  the rotational invariance of $W$ and the measure $dv_d$, we obtain 
\[
\Var_{\PP_{d}} \Big[ \sum_{x\in \X} W^{z_o} (x)  K_{\D_d}^x  \Big] 
 = \frac{1}{2} \int_{\D_d}\int_{\D_d} |W(z) - W(w)|^2  \widehat{T}(z, w; |z_o|) dv_d(z) dv_d(w)
\]
where $\widehat{T}(z, w; |z_o|)$ is given by 
\[
\widehat{T}(z, w; |z_o|) : = \int_0^{2\pi}T(e^{i\theta}z, e^{i\theta} w; |z_o|e_1) \frac{d \theta}{2 \pi}. 
\]
Direct computation yields
\[
\widehat{T}(z, w; |z_o|) = \frac{K_{\D_d}(z,w) | K_{\D_d}(z, w)|^2}{(1 - |z_o|^2)^{d+1} }  \sum_{k=0}^{d+1} \binom{d+1}{k}^2 |z_o|^{2k} \langle z, e_1\rangle^k \overline{\langle w, e_1 \rangle}^k
\]  Using once more the rotational invariance of $W$ and the measure $dv_d$, we obtain  
\begin{align}\label{var-non-sym}
\Var_{\PP_{d}} \Big[ \sum_{x\in \X} W^{z_o} (x)  K_{\D_d}^x  \Big] 
 = \frac{1}{2} \int_{\D_d}\int_{\D_d} |W(z) - W(w)|^2  \widetilde{T}(z, w; |z_o|) dv_d(z) dv_d(w)
\end{align}
where $\widetilde{T}(z, w;  |z_o|)$  is given  by 
\[
\widetilde{T}(z, w;  |z_o|) : = \int_{\Sph_d}\int_{\Sph_d} \widehat{T}(|z| \cdot \zeta, |w| \cdot \xi; |z_o|) d\sigma_{\Sph_d}(\zeta) d\sigma_{\Sph_d}(\xi). 
\]

Note that we have
\[
K_{\D_d}(z, w) = \sum_{k = 0}^\infty a_k (z\cdot \bar{w})^k, \quad a_k \ge 0. 
\]
Since the function $(\zeta, \xi) \mapsto \zeta \cdot \bar{\xi}$ is non-negative definite, by a classical result due to Schur (the pointwise  products of non-negative definite functions are still non-negative), the functions $(\zeta, \xi) \mapsto (\zeta \cdot \bar{\xi})^k$ for all $k \in \N$ are all non-negative definite. Thus  the  function  
\[
 (\zeta, \xi) \mapsto K_{\D_d}(|z|\zeta, |w| \xi) =  \sum_{k = 0}^\infty a_k |z| |w| (\zeta \cdot \bar{\xi})^k
\]
is non-negative definite and so is the function $(\zeta, \xi) \mapsto K_{\D_d}(|w|\xi, |z| \zeta)$.  Using again Schur's result on pointwise product of non-negative definite functions, the function \[
(\zeta, \xi) \mapsto F_{z, w}(\zeta, \xi) : =  \frac{K_{\D_d}(|z|\zeta, |w| \xi) \cdot  | K_{\D_d}(|z|\zeta, |w| \xi)|^2}{(1 - |z_o|^2)^{d+1}}
\] is non-negative definite.  Therefore, we have 
\begin{align*}
\widetilde{T}(z, w; |z_o|)  &  =  \int_{\Sph_d}\int_{\Sph_d}  F_{z, w}(\zeta, \xi)  \sum_{k=0}^{d+1} \binom{d+1}{k}^2 (|z_o|^{2} |z| |w|)^k \langle \zeta, e_1\rangle^k \overline{\langle \xi , e_1 \rangle}^k d\sigma_{\Sph_d}(\zeta) d\sigma_{\Sph_d}(\xi)
\\
& \ge   (d+1)^2 \int_{\Sph_d}\int_{\Sph_d}  F_{z, w}(\zeta, \xi)   d\sigma_{\Sph_d}(\zeta) d\sigma_{\Sph_d}(\xi)  
\\ 
& =  \frac{(d+1)^2}{(1 - |z_o|^2)^{d+1}} \int_{\Sph_d}\int_{\Sph_d}  \frac{1}{ ( 1 - |z||w| \zeta \cdot \bar{\xi})^{d+1}}  \frac{1}{ | 1 - |z||w| \zeta \cdot \bar{\xi}|^{2d+2}}   d\sigma_{\Sph_d}(\zeta) d\sigma_{\Sph_d}(\xi). 
\end{align*}

{\flushleft\bf The case $d = 1$.}  By writing  $C$ the unit circle oriented counterclockwise, we have 
\begin{multline*}
\widetilde{T} (z, w; |z_o|) \ge \frac{4}{(1 - |z_o|^2)^2} \int_{0}^{2 \pi} \int_0^{2\pi }  \frac{1}{ ( 1 - |z||w| e^{i \theta_1} e^{- i \theta_2})^{2}}  \frac{1}{ | 1 - |z||w| e^{i \theta_1} e^{- i \theta_2}|^{4}}   \frac{d \theta_1}{2 \pi} \frac{d \theta_2}{2 \pi} = 
\\ 
= \frac{4}{(1 - |z_o|^2)^2}  \int_0^{2\pi }  \frac{1}{ ( 1 - |z||w| e^{i \theta})^2}  \frac{1}{ | 1 - |z||w| e^{i \theta}|^{4}}   \frac{d \theta}{2 \pi} = 
\\
= \frac{4}{(1 - |z_o|^2)^2} \frac{1}{2 \pi i } \oint_C \frac{\eta}{( 1 - |z||w| e^{i \theta})^4 ( \eta - |z||w|)^2} d\eta = 
\\
=\frac{4}{(1 - |z_o|^2)^2}   \left[ \frac{1}{(1 - |z|^2 |w|^2)^4} + \frac{4 |z|^2 |w|^2}{ (1 - |z|^2|w|^2)^5}\right] \ge \frac{4}{(1-|z_o|^2)^2} \frac{1}{(1 - |z|^2|w|^2)^5}. 
\end{multline*}

{\flushleft\bf The case $d \ge 2$.} In this case, by using  \cite[Lemma 1.9  \& formula (1.13)]{Zhu-UB}, we have 
\begin{multline*}
\widetilde{T}(z, w; |z_o|) \ge \frac{(d+1)^2 (d-1)}{(1 - |z_o|^2)^{d+1}} \int_{\D} \frac{1}{ ( 1 - |z||w| z')^{d+1}}  \frac{1}{ |( 1 - |z||w| z')|^{2d+2}}   \frac{dV(z')}{2\pi}  = 
\\
= \frac{(d+1)^2 (d-1)}{(1 - |z_o|^2)^{d+1}} \int_{0}^{2 \pi} \int_0^1 \frac{1}{ ( 1 - |z||w| r e^{i \theta})^{2d+2}}  \frac{1}{ ( 1 - |z||w| r e^{-i \theta})^{d+1}}   \frac{r dr d\theta}{2\pi} =
\\
= \frac{(d+1)^2 (d-1)}{(1 - |z_o|^2)^{d+1}} \int_0^1 r dr   \frac{1}{2\pi i}\oint_C\frac{\eta^d }{ ( 1 - |z||w| r \eta)^{2d+2}}  \frac{1}{ ( \eta - |z||w| r )^{d+1}}    d \eta  = 
\\
 = \frac{(d+1)^2 (d-1)}{(1 - |z_o|^2)^{d+1}} \int_0^1  \left[ \frac{1}{d!} \frac{\partial^d}{ \partial \eta^d}\Big|_{\eta  = |z||w| r}  \left( \frac{\eta^d }{ ( 1 - |z||w| r \eta)^{2d+2}} \right)  \right]    r dr.
\end{multline*}
Since for any integer $0\le k\le d$, 
\[
\frac{\partial^k}{ \partial \eta^k}\Big|_{\eta  = |z||w| r}  ( \eta^d ) \ge 0, \quad \frac{\partial^k}{ \partial \eta^k}\Big|_{\eta  = |z||w| r}  \left( \frac{1 }{ ( 1 - |z||w| r \eta)^{2d+2}} \right) \ge 0,
\]
we have 
\begin{multline*}
\frac{\partial^d}{ \partial \eta^d}\Big|_{\eta  = |z||w| r}  \left( \frac{\eta^d }{ ( 1 - |z||w| r \eta)^{2d+2}} \right) = 
\\
 =  \sum_{k = 0}^d \binom{d}{k}   \frac{\partial^{d-k}}{\partial \eta^{d-k}} (\eta^d) \frac{\partial^k}{ \partial \eta^k} \left( \frac{1}{ ( 1 - |z||w| r \eta)^{2d+2}} \right)\Big|_{\eta = |z||w|r} \ge 
\\ 
\ge \left[     \frac{\partial^d }{\partial \eta^d} (\eta^d) \frac{1}{(1 - |z||w|r)^{2d +2}}   +  \eta^d \frac{\partial^d}{ \partial \eta^d} \left( \frac{1}{ ( 1 - |z||w| r \eta)^{2d+2}} \right)\right]\Big|_{\eta = |z||w|r}  = 
\\
=     \frac{ d ! (1 - |z|^2 |w|^2 r^2)^d +  \frac{(2d +1+d)!}{(2d+1)!} (|z|^2|w|^2r^2)^{d} }{ (1 - |z|^2 |w|^2 r^2)^{3d +2}} \ge \frac{c_d'}{(1 - |z|^2 |w|^2 r^2)^{3d +2}},
\end{multline*}
where 
\[
c_d': =  \min_{x \in [0, 1]}  \left[ d ! (1 - x)^d +  \frac{(2d +1+d)!}{(2d+1)!} x^{d}\right] > 0.
\]
Therefore, there exist constants $c_d''>0, c_d'''>0$ depending only on $d$, such that 
\begin{multline*}
\widetilde{T}(z, w; |z_o|) \ge  \frac{c_d''}{(1 - |z_o|^2)^{d+1}} \int_0^1 \frac{1}{(1 - |z|^2 |w|^2 r^2)^{3d +2}} r dr \ge 
\\ 
\ge \frac{c_d'''}{(1 - |z_o|^2)^{d+1}}  \frac{1}{(1 - |z|^2 |w|^2)^{3d  +1}}  \left[ \frac{1 - (1 -|z|^2|w|^2)^{3d + 1}}{|z|^2 |w|^2} \right].
\end{multline*}
By taking 
\[
c_d : = c_d''' \cdot \min_{x \in (0, 1]} \left[ \frac{1 - (1 -x )^{3d + 1}}{x} \right] >0,
\]
we obtain 
\[
\widetilde{T}(z, w; |z_o|) 
\ge \frac{c_d}{(1 - |z_o|^2)^{d+1}}  \frac{1}{(1 - |z|^2 |w|^2)^{3d  +1}}.
\]

Finally, by substituting the lower bound obtained above for $\widetilde{T}(z, w; |z_o|)$ into the equality \eqref{var-non-sym}, we complete the proof of Proposition \ref{prop-new-var-f}. 
\end{proof}

\begin{remark}\label{rem-precise-f}
In dimension $d = 1$ case, by applying carefully the argument in the proof of Proposition \ref{prop-new-var-f}, one obtain a precise equality as follows:  for any $z_o \in \D$ and any compactly supported bounded radial $W: \D\rightarrow \R^{+}$, 
\[
\Var_{\PP_{1}} \Big[ \sum_{x\in \X} W^{z_o} (x)  K_\D^x  \Big] =\frac{1}{2 ( 1 - |z_o|^2)^2} \int_\D \int_\D | W(z) - W(w)|^2  \cdot I(z, w) dA(z) dA(w),
\]
where $I$ is given by the formula:
\[
I(z, w) =  \frac{1 + 8 |z_o zw|^2 + 3 |z_o zw|^4}{ ( 1 - |zw|^2)^4} + \frac{4 | zw|^2 ( 1 + 4 |z_o zw|^2 + | z_o zw|^4)}{(1 - |zw|^2)^5}.
\]
\end{remark}

\subsection{Proof of Proposition \ref{prop-failure}}
Fix $z_o \in \D_d$. Since $W: \D_d \rightarrow \R^{+}$ is radial and compactly supported, there exists a function $\widehat{W}: [0, 1) \rightarrow \R^{+}$ whose support is contained in $[0, 1 - \varepsilon]$ for some $\varepsilon > 0$ such that $W(z) = \widehat{W}(|z|)$.  

Set  
\begin{align}\label{def-V-g}
V(t): = \widehat{W}((1-t)^{1/2d}), \quad g(t):   = \frac{V(t)}{t^{d+1}},\quad  \text{where $t\in (0, 1)$}.
\end{align} Since $W$ is compactly supported, 
there exists $\varepsilon> 0$ such that  $\supp (g) \subset [\varepsilon, 1]$. Using the conformal invariance of $\PP_d$ (here we only used the conformal invariance of the first intensity of $\PP_d$) and the Bernoulli's inequality  $(1 - t)^{1/d} \le 1 - t/d$ for any $t \le    1$,   we have 
\begin{multline*}
 \E_{\PP_d} \Big( \sum_{x\in \X} W^{z_o} (x)  \Big)   =  \E_{\PP_{d}} \Big( \sum_{x\in \X} \widehat{W} (|x|)  \Big) = \int_{\D_d} \frac{\widehat{W}(|z|)}{ ( 1 - |z|^2)^{d+1}} dv_d(z)  = 
\\
 = 2d \int_0^1 \frac{\widehat{W}(r)}{(1-r^2)^{d+1}} r^{2d-1}dr =  \int_0^1 \frac{\widehat{W}(t^{1/2d})}{(1 - t^{1/d})^{d+1}} dt   = 
\\
  =\int_0^1 \frac{\widehat{W}((1-t)^{1/2d})}{(1- (1 - t)^{1/d})^{d+1}}dx \le  d ^{d+1}\int_0^1 \frac{\widehat{W}((1-t)^{1/2d})}{t^{d+1}}dx    =  d^{d+1}\int_0^1 g(t) dt. 
\end{multline*}

Now set $\alpha_d = 5$ if $d = 1$ and $\alpha_d = 3d + 1$ if $d \ge 2$. Then by Proposition \ref{prop-new-var-f}, we have 
\begin{multline*}
\Var_{\PP_{d}} \Big[ \sum_{x\in \X} W^{z_o} (x)  K_{\D_d}^x  \Big] \ge \frac{c_d}{( 1 - |z_o|^2)^{d+1}} \int_{\D_d} \int_{\D_d}  \frac{ | W(z) - W(w)|^2}{ (1 - |z|^2 |w|^2)^{\alpha_d}} dv_d(z) dv_d(w) = 
\\
 = \frac{c_d \cdot (2d)^2}{( 1 - |z_o|^2)^{d+1}} \int_{0}^1 \int_{0}^1  \frac{ | \widehat{W}(r_1) - \widehat{W}(r_2)|^2}{ (1 - r_1^2 r_2^2)^{\alpha_d}} (r_1r_2)^{2d-1}dr_1 dr_2 = 
\\
 =\frac{c_d }{( 1 - |z_o|^2)^{d+1}} \int_{0}^1 \int_{0}^1  \frac{ | \widehat{W}(t_1^{1/2d}) - \widehat{W}(t_2^{1/2d})|^2}{ (1 - (t_1t_2)^{1/d})^{\alpha_d}} dt_1 dt_2  . 
\end{multline*}
Since $(t_1 t_2)^{1/d} \ge t_1t_2$ whenever $0 \le t_1t_2 \le 1$, we have 
\begin{multline*}
\Var_{\PP_{d}} \Big[ \sum_{x\in \X} W^{z_o} (x)  K_{\D_d}^x  \Big] \ge
\frac{c_d }{( 1 - |z_o|^2)^{d+1}} \int_{0}^1 \int_{0}^1  \frac{ | \widehat{W}(t_1^{1/2d}) - \widehat{W}(t_2^{1/2d})|^2}{ (1 - t_1t_2)^{\alpha_d}} dt_1 dt_2 = 
\\
 = \frac{c_d }{( 1 - |z_o|^2)^{d+1}} \int_{0}^1 \int_{0}^1  \frac{ | V(t_1) - V(t_2)|^2}{ (t_1 + t_2 - t_1t_2)^{\alpha_d}} dt_1 dt_2  \ge 
\\
\ge \frac{c_d }{( 1 - |z_o|^2)^{d+1}} \int_{0}^1 \int_{0}^1  \frac{ | V(t_1) - V(t_2)|^2}{ (t_1 + t_2 )^{\alpha_d}} dt_1 dt_2   = 
\\
 =   \frac{2c_d }{( 1 - |z_o|^2)^{d+1}} \int_{0\le t_1 \le t_2 \le 1}  \frac{ | V(t_1) - V(t_2)|^2}{ (t_1 + t_2)^{\alpha_d}} dt_1 dt_2. 
\end{multline*}
Using change of variables $t_1 = \lambda t, t_2 = t$ in the last integral and recall the definition for the function $g$ in \eqref{def-V-g}, we obtain 
\begin{multline*}
\Var_{\PP_{d}} \Big[ \sum_{x\in \X} W^{z_o} (x)  K_{\D_d}^x  \Big] \ge
   \frac{2c_d }{( 1 - |z_o|^2)^{d+1}} \int_{0}^1\int_0^1  \frac{ | V(\lambda t) - V(t)|^2}{  (\lambda + 1)^{\alpha_d}t^{\alpha_d-1}} d\lambda dt\ge 
\\
\ge \frac{c_d }{2^{\alpha_d-1}( 1 - |z_o|^2)^{d+1}} \int_{0}^1\int_0^1  \frac{ | V(\lambda t) - V(t)|^2}{  t^{\alpha_d-1}} d\lambda dt  = 
\\
=  \frac{c_d }{2^{\alpha_d-1}( 1 - |z_o|^2)^{d+1}} \int_{0}^1\int_0^1  \frac{ | \lambda^{d+1}g(\lambda t) - g(t)|^2}{  t^{\alpha_d-2d-3}} d\lambda dt. 
\end{multline*}
By the definition of $\alpha_d$, we have $\alpha_d - 2d -3 \ge 0$ for any integer $d\ge 1$. Therefore 
 \[
\Var_{\PP_{d}} \Big[ \sum_{x\in \X} W^{z_o} (x)  K_{\D_d}^x  \Big] \ge
   \frac{c_d }{2^{\alpha_d-1}( 1 - |z_o|^2)^{d+1}} \int_{0}^1\int_0^1  | \lambda^{d+1}g(\lambda t) - g(t)|^2 d\lambda dt. 
\]
Now note that 
\begin{align*}
\int_0^1\int_0^1 \left|  \lambda^{d +1}  g(\lambda t)  \right|^2 d\lambda dt = \int_0^1 \lambda^{2d +1} d\lambda \int_0^\lambda g(t')^2 dt' \le 
\\
\le   \int_0^1 \lambda^{2d +1} d\lambda \int_0^1 g(t')^2 dt' = \frac{1}{2d+2} \int_0^1 g(t)^2 dt. 
\end{align*}
The above inequality combined with the triangle inequality and then Cauchy-Schwarz inequality yields
\begin{align*}
& \left(\int_0^1\int_0^1 \left|  \lambda^{d+1}  g(\lambda t) - g(t) \right|^2 d\lambda dt\right)^{1/2} 
\\
 \ge&   \left(\int_0^1\int_0^1 \left|  g(t) \right|^2 d\lambda dt\right)^{1/2} - \left(\int_0^1\int_0^1 \left|  \lambda^{d+1}  g(\lambda t)\right|^2 d\lambda dt\right)^{1/2}
\\
 \ge &   \Big(1 - \frac{1}{\sqrt{2d + 2}} \Big) \left(\int_0^1 g(t)^2 dt\right)^{1/2} \ge \Big(1 - \frac{1}{\sqrt{2d + 2}} \Big) \int_0^1 g(t) dt.
\end{align*}
Therefore, we obtain 
\begin{align*}
 \frac{\displaystyle \Var_{\PP_{d}} \Big[ \sum_{x\in \X} W^{z_o} (x)  K_{\D_d}^x  \Big]   }{  \displaystyle  \Big[  \E_{\PP_{d}} \Big( \sum_{x\in \X} W^{z_o} (x)  \Big) \Big]^2 }  \ge     \frac{c_d }{2^{\alpha_d-1}( 1 - |z_o|^2)^{d+1}}  \frac{1}{d^{2d+2}}\Big(1 - \frac{1}{\sqrt{2d + 2}} \Big)^2.   \qed
\end{align*}

\subsection{Proof of Propositioin \ref{prop-sharp}}

Recall that we denote by $K^\alpha$ the reproducing kernel of $A^2_\alpha(\D)$ and set $K^\alpha_x\in A^2_\alpha(\D)$ by setting $K_\alpha^x(y)  = K^\alpha(y, x)$. Note that for any $z_o \in \D$, we have 
\[
\sup_{f \in A^2_{\alpha}(\D)_1}\Big| \sum_{k = N}^{M}  T(s, z_o, Z(\frak{g}_\D); f)\Big| =  \Big \|  \sum_{x \in Z(\frak{g}_\D)} (\widehat{W}^{s}_{N, M})^{z_o}(x) K^\alpha_x  \Big\|_{A^2_\alpha(\D)}, 
\]
where $\widehat{W}^{s}_{N, M} (x)$ is a compactly supported radial weight given by 
\[
\widehat{W}^{s}_{N, M} (x): = e^{-s d_\D(x, 0)} \mathds{1} (N \le d_\D(x, 0) <M). 
\]
Using the same argument as in the proof of Proposition \ref{prop-Sob}, we have 
\begin{multline}\label{R-z-w}
\Var \Big[ \sum_{x \in Z(\frak{g}_\D)} (\widehat{W}^{s}_{N, M})^{z_o}(x) K^\alpha_x  \Big] =
\\
 =  \frac{1}{2} \int_\D \int_\D | (\widehat{W}^{s}_{N, M})^{z_o} (z) - (\widehat{W}^{s}_{N, M})^{z_o} (w)|^2 \cdot   K^\alpha(z, w) |K_{\D}(z, w)|^2 dA(z) dA(w) = 
\\
 = \frac{1}{2} \int_\D \int_\D | \widehat{W}^{s}_{N, M} (z) - \widehat{W}^{s}_{N, M} (w)|^2 \cdot   \underbrace{K^\alpha(\varphi_{z_o}(z), \varphi_{z_o}(w)) |K_{\D}(z, w)|^2}_{R(z, w)} dA(z) dA(w). 
\end{multline}
 Using the rotational invariance of $\widehat{W}^{s}_{N, M} (x)$ and the measure $dA$, the term $R(z, w)$ in the equality \eqref{R-z-w} can be replaced first by 
\[
\widehat{R}(z, w)  :  = \frac{1}{2\pi}\int_{0}^{2\pi} R(e^{i\theta}z, e^{i\theta} w) d\theta
\]
and  then by 
\[
\widetilde{R}(z, w) : =  \frac{1}{2\pi}\int_{0}^{2\pi} \widehat{R}(e^{i\theta}z, w) d\theta.
\]
Recall the formula \eqref{K-alpha} for $K^\alpha$ and  the identity \eqref{mob-id}.  For $x \in \D$, write
$
(1 - x)^{2 + \alpha} = \sum_{k = 0}^\infty c_k x^k,
$ with all coefficients $c_k \in \R$, 
we have 
\[
\frac{1}{2\pi}\int_0^{2 \pi}( 1-   e^{i\theta}z \bar{z}_o)^{2 + \alpha}  ( 1-  e^{- i \theta} z_o \bar{w})^{2 + \alpha} d \theta =  \sum_{k= 0}^\infty  c_k^2 |z_o|^{2k}  z^k \bar{w}^k.  
\]
Note also for any non-negative integer $k$, we have 
\begin{multline*}
\frac{1}{2\pi} \int_0^{2\pi} \frac{e^{ik  \theta } z^k \bar{w}^k}{(1 -  e^{i \theta}z \bar{w})^{2 + \alpha}} \frac{1}{| 1 - e^{i \theta} z \bar{w}|^4} d\theta =  \frac{\partial }{\partial \zeta}\Big|_{\zeta = \bar{z}w} \left(\frac{   z^k \bar{w}^k \zeta^{k+1}}{ (1 - \zeta  z \bar{w})^{4 + \alpha}}\right)   = 
\\
 = \frac{ (k + 1) |zw|^{2k} }{ (1 - |zw|^2)^{4 + \alpha}} +  \frac{(4 + \alpha)  |zw|^{2k +2}  }{ (1 - |zw|^2)^{5 + \alpha}} \ge 0. 
\end{multline*}
It follows that there exists a constant $c  =  c_\alpha> 0$ such  that 
\begin{multline*}
\widetilde{R}(z, w)  \ge \frac{c_\alpha}{(1- |z_o|^2)^{2 + \alpha}}  \left[  \frac{ 1 }{ (1 - |zw|^2)^{4 + \alpha}} +  \frac{(4 + \alpha)  |zw|^{2}  }{ (1 - |zw|^2)^{5 + \alpha}} \right] \ge 
\\
\ge  \frac{c_\alpha}{(1- |z_o|^2)^{2 + \alpha}}  \frac{1 }{ (1 - |zw|^2)^{5 + \alpha}}. 
\end{multline*}
Therefore, there exists a constant $C = C_{\alpha, z_o}> 0$ such that 
\begin{multline*}
\Var \Big[ \sum_{x \in Z(\frak{g}_\D)} (\widehat{W}^{s}_{N, M})^{z_o}(x) K^\alpha_x  \Big] \ge  C  \int_\D \int_\D  \frac{| \widehat{W}^{s}_{N, M} (z) - \widehat{W}^{s}_{N, M} (w)|^2}{( 1- |zw|^2)^{5 + \alpha }}dA(z) dA(w) = 
\\
= C\int_0^1 \int_0^1 \frac{| \widehat{W}^{s}_{N, M} (\sqrt{x}) - \widehat{W}^{s}_{N, M} (\sqrt{y})|^2}{ (1- xy)^{5+\alpha}} dxdy. 
\end{multline*}

{\flushleft \bf Claim.} If $1< s \le \frac{3 + \alpha}{2}$, then 
\[
\lim_{M\to\infty} \int_\D \int_\D  \frac{| \widehat{W}^{s}_{N, M} (z) - \widehat{W}^{s}_{N, M} (w)|^2}{( 1- |zw|^2)^{5 + \alpha }}dA(z) dA(w)  = \infty. 
\]
Indeed, clearly, by setting \[
\widehat{W}^{s}_{N} (x): = e^{-s d_\D(x, 0)} \mathds{1} ( d_\D(x, 0) \ge N)  = \left(\frac{1 - |x|}{ 1+|x|}\right)^s \mathds{1}(|x|\ge \delta_N), \quad   \log \left(\frac{1 - \delta_N}{ 1+\delta_N}\right)= N, 
\]  we have 
\begin{multline*}
\lim_{M\to\infty} \int_\D \int_\D  \frac{| \widehat{W}^{s}_{N, M} (z) - \widehat{W}^{s}_{N, M} (w)|^2}{( 1- |zw|^2)^{5 + \alpha }}dA(z) dA(w)  =
\\
 =    \int_\D \int_\D  \frac{| \widehat{W}^{s}_{N} (z) - \widehat{W}^{s}_{N} (w)|^2}{( 1- |zw|^2)^{5 + \alpha }}dA(z) dA(w)  
= \int_\D \int_\D  \frac{| \widehat{W}^{s}_{N} (z) - \widehat{W}^{s}_{N} (w)|^2}{( 1- |zw|^2)^{5 + \alpha }}dA(z) dA(w) \ge
\\
  = 4\int_0^1 \int_0^1 \frac{| \widehat{W}^{s}_{N} (r_1) - \widehat{W}^{s}_{N} (r_2) |^2}{( 1- |r_1r_2|^2)^{5 + \alpha }}r_1r_2dr_1dr_2 \ge 
\\
\ge \frac{4 \delta_N^2}{(1 + \delta_N^2)^{5 + \alpha}} \int_0^1 \int_0^1 \frac{| \widehat{W}^{s}_{N} (r_1) - \widehat{W}^{s}_{N} (r_2) |^2}{( 1- r_1r_2)^{5 + \alpha }}dr_1dr_2.
\end{multline*}
Now by setting 
\[
g(t) = \widehat{W}^{s}_{N} (1-t) =   \frac{t^s}{(2 -t)^s} \mathds{1}(t \le 1- \delta_N), 
\] we have 
\begin{multline*}
\int_0^1 \int_0^1 \frac{| \widehat{W}^{s}_{N} (r_1) - \widehat{W}^{s}_{N} (r_2) |^2}{( 1- r_1r_2)^{5 + \alpha }}dr_1dr_2
 = \int_0^1 \int_0^1 \frac{| g(t_1)  - g(t_2) |^2}{ (t_1 +t_2 - t_1 t_2)^{5+\alpha}} dt_1dt_2 =
\\
 =  2 \int_{0 \le t_1 \le t_2 \le 1} \frac{| g(t_1)  - g(t_2) |^2}{ (t_1 +t_2 - t_1t_2)^{5+\alpha}} dt_1dt_2
\\ 
 \text{(by changing of variables $t_1  = \lambda t, t_2 = t$)}= 2 \int_0^1 \int_0^1 \frac{| g(\lambda t )  - g(t) |^2}{ (\lambda  +1 - \lambda t)^{5+\alpha}}  t^{-4 - \alpha} d\lambda dt\ge 
\\
 \ge  \frac{2}{2^{5+\alpha}} \int_0^1 \int_0^1 | g(\lambda t )  - g(t) |^2   t^{-4 - \alpha} d\lambda dt = 
\\
 =  \frac{2}{2^{5+\alpha}} \int_0^1 d\lambda  \int_0^{1- \delta_N} \Big| \frac{\lambda ^s}{(2 -\lambda t)^s}  - \frac{1}{(2 -t)^s} \Big|^2   t^{2s -4 - \alpha}  dt  \ge 
\\
\ge \frac{2}{2^{5+\alpha}} \int_0^{1/2} d\lambda  \int_0^{1- \delta_N} \Big| \frac{1}{(2/\lambda - t)^s}  - \frac{1}{(2 -t)^s} \Big|^2   t^{2s -4 - \alpha}  dt  \ge 
\\
\ge \frac{1}{2^{5+\alpha}} \min_{t\in [0, 1]}   \Big( \Big| \frac{1}{2 - t)^s}  - \frac{1}{(4 -t)^s} \Big|^2 \Big)  \int_0^{1- \delta_N}   t^{2s -4 - \alpha}  dt.
\end{multline*}
Since $1< s \le \frac{3 + \alpha}{2}$, we have $2s - 4 -\alpha \le -1$ and thus 
\[
\int_0^{1- \delta_N} t^{2s - 4 - \alpha} dt = \infty. 
\]
The claim is proved.

Finally, since 
\begin{multline*}
\E\Big[ \Big \|  \sum_{x \in Z(\frak{g}_\D)} (\widehat{W}^{s}_{N, M})^{z_o}(x) K^\alpha_x  \Big\|_{A^2_\alpha(\D)}^2 \Big] =  \Big \| \E  \sum_{x \in Z(\frak{g}_\D)} (\widehat{W}^{s}_{N, M})^{z_o}(x) K^\alpha_x  \Big\|_{A^2_\alpha(\D)} + 
\\
+ \Var \Big[ \sum_{x \in Z(\frak{g}_\D)} (\widehat{W}^{s}_{N, M})^{z_o}(x) K^\alpha_x  \Big], 
\end{multline*}
we obtain the desired limit equality 
\[
\limsup_{N, M \to\infty} \E \left(\sup_{f \in A^2_{\alpha}(\D)_1}\Big| \sum_{k = N}^{M}  T(s, z_o, Z(\frak{g}_\D); f)\Big|^2 \right) = \infty. 
\]
\section{Appendix}

\subsection{A trivial example for simultaneous reconstruction}\label{sec-tr-ex} 
Denote by $L^1(\T)$ the space of $\C$-valued integrable functions on $\T$ with respect to the normalized Lebesgue measure $dm$ on $\T$. 
If $g \in L^1(\T)$, we write $P[g] : = P[gdm]$.  Set 
\[
h^1_{ac}(\D):  = \left\{h: \D\rightarrow \C\Big| h = P[g],  \, g \in L^1(\T)\right\}. 
\]
 A simple simultaneous reconstruction algorithm for all functions $h \in h_{ac}^1(\D)$ is described as follows. First, one shows that a typical realization $X = Z(\mathfrak{g}_\D)$ satisfies for Lebesgue-almost every $\zeta \in \T$, the Stolz angle $S_\zeta$, the closed convex  hull of $\{\zeta\} \cup \{z\in \D: |z| \le 1/\sqrt{2}\}$, contains infinitely many points. Fix such a typical realization $X = Z(\mathfrak{g}_\D)$.  Then for any $h = P[g]$ with $g \in L^1(\T)$, we know that for Lebesgue almost all $\zeta$ (the implied full measure subset of $\T$ depends on the function $h$ and on $X$), the non-tangential limit
\[
h^*(\zeta) =  \lim_{S_{\zeta} \ni z \to \zeta} h (z) =  \lim_{X\cap S_{\zeta} \ni z \to \zeta} h|_{X\cap S_{\zeta}} (z)
\]
exists and $h^*(\zeta) = g(\zeta)$.  
Therefore, for all $z\in \D$,  we have $h(z)=  P[g](z) = P[h^{*}](z)$. 

This simple reconstruction for functions in $h^1_{ac}(\D)$ is based on two facts: 
\begin{itemize}
\item The existence of non-tangential limit for Lebesgue almost every point of $\T$.
\item The original function coincides with the Poisson integral of its non-tangential limit. 
\end{itemize}
Therefore,  such reconstruction algorithm can not be applied to any weighted Bergman spaces, since all these spaces contain functions without non-tangential limit (the first point is not satisfied); it can not be applied to any space  of harmonic functions containing Poisson integral of signed Borel measure $\nu$ which is not absolutely continuous with respect to the Lebesgue measure (the second point is not satisfied).


\def\cprime{$'$} \def\cydot{\leavevmode\raise.4ex\hbox{.}}

\end{document}